\documentclass{amsart}

\usepackage{latexsym,amsmath,amssymb,amsfonts,amscd,graphics,appendix,amsxtra}
\usepackage[mathscr]{eucal}
\usepackage[all]{xypic}
\usepackage[normalem]{ulem}

\newtheorem{mainthm}{Theorem}
\numberwithin{equation}{section}
\theoremstyle{plain}
\newtheorem{theorem}[equation]{Theorem}

\newtheorem{lemma}[equation]{Lemma}

\newtheorem{condition}[equation]{Condition}

\newtheorem{proposition}[equation]{Proposition}
\newtheorem{corollary}[equation]{Corollary}

\theoremstyle{remark}
\newtheorem{remark}[equation]{Remark}

\theoremstyle{definition}
\newtheorem{definition}[equation]{Definition}
\newtheorem{notation}[equation]{Notation}
\newtheorem{convention}[equation]{Convention}

\newcommand{\D}{\displaystyle}

\def\Zee{\mathbb{Z}}
\def\Q{\mathbb{Q}}
\def\R{\mathbb{R}}
\def\Cee{\mathbb{C}}
\def\Pee{\mathbb{P}}
\def\Id{\operatorname{Id}}
\def\Pic{\operatorname{Pic}}

\def\End{\operatorname{End}}
\def\Hom{\operatorname{Hom}}

\def\Ker{\operatorname{Ker}}

\def\Sym{\operatorname{Sym}}
\def\scrO{\mathcal{O}}

\def\Lie{\operatorname{Lie}}

\newcommand{\bP}{\mathbb{P}}
\newcommand{\bA}{\mathbb{A}}
\newcommand{\bG}{\mathbf{G}}
\newcommand{\bK}{\mathbf{K}}
\newcommand{\bfP}{\mathbf{P}}
\newcommand{\bD}{\mathbf{D}}
\newcommand{\bZ}{\mathbb{Z}}
\newcommand{\bQ}{\mathbb{Q}}
\newcommand{\bR}{\mathbb{R}}
\newcommand{\bC}{\mathbb{C}}
\newcommand{\bH}{\mathbb{H}}

\newcommand{\Res}{\mathrm{Res}}
\newcommand{\disc}{\mathrm{disc}}
\newcommand{\Nm}{\mathrm{Nm}}

\newcommand{\Aut}{\mathrm{Aut}}
\newcommand{\Stab}{\mathrm{Stab}}

\newcommand{\Spin}{\mathrm{Spin}}
\newcommand{\U}{\mathrm{U}}
\newcommand{\E}{\mathrm{E}}

\newcommand{\rank}{\mathrm{rank}}

\newcommand{\id}{\mathrm{id}}
\newcommand{\ad}{\mathrm{ad}}

\newcommand{\bS}{\mathbb{S}}     
\newcommand{\SO}{\mathrm{SO}}
\newcommand{\Sp}{\mathrm{Sp}}
\newcommand{\Gal}{\mathrm{Gal}}

\newcommand{\SL}{\mathrm{SL}}
\newcommand{\SU}{\mathrm{SU}}
\newcommand{\PGL}{\mathrm{PGL}}
\newcommand{\GL}{\mathrm{GL}}
\newcommand{\Gm}{\mathbb G_m}
\newcommand{\Hg}{\mathrm{Hg}}
\newcommand{\MT}{\mathrm{MT}}
\newcommand{\calB}{\mathcal{B}} 
\newcommand{\calD}{\mathcal{D}} 
 
\newcommand{\fH}{\mathfrak{H}}

\newcommand{\calS}{\mathcal{S}}

\newcommand{\sX}{\mathscr{X}}

\newcommand{\calV}{\mathscr{V}}
\newcommand{\sV}{\mathscr{V}}

\newcommand{\fp}{\mathfrak{p}}
\newcommand{\fa}{\mathfrak{a}}
\newcommand{\fm}{\mathfrak{m}}

\newcommand{\fg}{\mathfrak{g}}

\newcommand{\fh}{\mathfrak{h}}
\newcommand{\fk}{\mathfrak{k}}
\newcommand{\I}{\mathrm{I}}
\newcommand{\II}{\mathrm{II}}
\newcommand{\III}{\mathrm{III}}
\newcommand{\IV}{\mathrm{IV}}
\newcommand{\EIII}{\mathrm{EIII}}
\newcommand{\EVII}{\mathrm{EVII}}
\newcommand{\hg}{\mathfrak{hg}}

\title[]{Semi-algebraic horizontal subvarieties of Calabi--Yau type} 
\date{\today}
\begin{document}

\author[R. Friedman]{Robert Friedman}
\address{Columbia University, Department of Mathematics, New York, NY 10027}
\email{rf@math.columbia.edu}
\author[R. Laza]{Radu Laza}
\address{Stony Brook University, Department of Mathematics,  Stony Brook, NY 11794}
\email{rlaza@math.sunysb.edu}
\thanks{The  second author was partially supported by NSF grant DMS-0968968 and a Sloan Fellowship}

\begin{abstract}
We study  horizontal subvarieties $Z$ of a Griffiths period domain $\bD$. If $Z$ is defined by algebraic equations, and if $Z$ is also invariant under a large discrete subgroup in an appropriate sense,  
we prove that $Z$ is a Hermitian symmetric domain $\calD$, embedded  via a totally geodesic embedding in  $\bD$. Next we  discuss the case when $Z$ is in addition of Calabi--Yau type. We  classify the possible VHS of Calabi--Yau type parametrized by Hermitian symmetric domains $\calD$ and show that they are essentially those found by  Gross and Sheng--Zuo, up to taking factors of symmetric powers and certain shift operations. In the weight three case, we explicitly describe the embedding $Z\hookrightarrow \bD$ from the perspective of Griffiths  transversality and relate this description to the Harish-Chandra realization of $\calD$ and to the  Kor\'anyi--Wolf tube domain description.  There are further connections  to homogeneous  Legendrian varieties  and the  four Severi varieties of Zak. 
\end{abstract}

\bibliographystyle{alpha}
\maketitle

\section*{Introduction}
Let $\bD$ be a period domain, i.e.\ a classifying space for polarized Hodge structures of weight $n$ with Hodge numbers $\{h^{p,q}\}$, $p+q=n$. It has been known since Griffiths' pioneering work that, unless $\bD$ is a classifying space for weight one Hodge structures (polarized abelian varieties) or weight two Hodge structures satisfying $h^{2,0} =1$, then  most of the points of $\bD$ do not come from algebraic geometry in any sense. More precisely, any geometrically defined variation of Hodge structure is contained in a horizontal subvariety of $\bD$, an integral manifold for the differential system corresponding to Griffiths transversality, and the union of all such arising from algebraic geometry is a countable union of proper subvarieties of $\bD$. It is thus of interest to write down specific examples of such horizontal subvarieties $Z$. One of the simplest cases where Griffiths transversality is a nontrivial condition is the case of weight three Hodge structures of Calabi--Yau type, i.e.\ $h^{3,0} =1$. This case is also very important for geometric reasons.

In the general case, since $\bD$ is an open subset of its compact dual $\check\bD$, and $\check\bD$ is a projective variety, it is natural to look at those $Z$ which can be defined algebraically, i.e.\ such that $Z$ is a connected component of $\hat{Z}\cap \bD$, where $\hat{Z}$ is a closed algebraic subvariety of $\check\bD$. We will refer to such $Z$ as \textsl{semi-algebraic in $\bD$}. 

A closed horizontal subvariety $Z$ of $\bD$ coming from algebraic geometry satisfies an additional condition: if $\Gamma$ is the stabilizer of $Z$ in the natural arithmetic group $\bG(\Zee)$ acting on $\bD$, then $\Gamma$ acts properly discontinuously on $Z$, the image $\Gamma \backslash Z \subseteq \bG(\Zee)\backslash \bD$ is a closed subvariety, and it is the image of a quasi-projective variety under a proper holomorphic map. Thus it is reasonable in general to look at closed horizontal subvarieties $Z$ of $\bD$ such that $\Gamma \backslash Z$ is quasi-projective. Actually, we will need a mild technical strengthening of this hypothesis which we call \textsl{strongly quasi-projective} (Definition \ref{strongqp}). Under this hypothesis, we prove the following theorem in Section~\ref{sectalgebraic}:

\begin{mainthm}\label{maintheorem} Let $Z$ be a closed horizontal subvariety of a classifying space $\bD$ for Hodge structures and let $\Gamma$ be the stabilizer of $Z$ in the appropriate arithmetic group $\bG(\Zee)$. Assume that 
\begin{itemize}
\item[(i)] $S=\Gamma\backslash Z$ is strongly quasi-projective;
\item[(ii)] $Z$ is semi-algebraic in $\bD$. 
\end{itemize}
Then  $Z$ is a Hermitian symmetric domain  whose embedding  in $\bD$ is an equivariant, holomorphic, horizontal embedding.
\end{mainthm} 

The main ingredients used in the proof are the theorem of the fixed part as proved by Schmid for variations of Hodge structure over quasi-projective varieties,   Deligne's characterization of Hermitian symmetric domains  \cite{dshimura},  and  the recent theory of Mumford--Tate domains as developed by Green--Griffiths--Kerr \cite{mtbook}.   Theorem~\ref{maintheorem}, in the case where $\bD$ itself is Hermitian symmetric (and thus $S$ is a subvariety of a Shimura variety), has been proved independently by Ullmo--Yafaev \cite{ullmo}, using  similar methods.  This result is related in spirit,  but in a somewhat different direction, to  a conjecture of Koll\'ar \cite{kollar}, which says roughly that, if $Z$ is simply connected and semi-algebraic, and $S=\Gamma \backslash Z$ is projective for some discrete group $\Gamma$ of biholomorphisms of $Z$, then $Z$ is the product of a Hermitian symmetric space and a simply connected projective variety.

The remainder of the paper is concerned with  Hodge structures of \textsl{Calabi--Yau type}, in other words effective weight $n$ Hodge structures such that $h^{n,0} = 1$. For Hermitian symmetric spaces of tube type, Gross \cite{bgross} has constructed certain natural variations of Hodge structure of Calabi--Yau type. This construction was extended by Sheng--Zuo \cite{sz} to the non-tube case to construct complex variations of Hodge structure. (We note that the existence of natural variations of Hodge structure of Calabi--Yau type over the exceptional Hermitian symmetric domains was previously noticed by Looijenga \cite{looijengaletter}.) Their methods can easily be adapted to construct real variations of Hodge structure of Calabi--Yau type. In Section~\ref{sectclassify}, we show that all real variations of Hodge structure of Calabi--Yau type over a Hermitian symmetric space can be constructed via standard techniques from those of Gross and Sheng--Zuo. In the spirit of Gross, we classify  real variations of Hodge structure or $\Q$-variations of Hodge structure which remain irreducible over $\R$. More precisely, we show:

\begin{mainthm}\label{thm2intro}
For every irreducible Hermitian symmetric domain of non-compact type $\calD=G(\bR)/K$, there exists a canonical $\bR$-variation of Hodge structure $\calV$ of Calabi--Yau type. Any other irreducible equivariant $\bR$-variation of Hodge structure  of Calabi--Yau type on $\calD$ can be obtained  as a summand  of $\Sym^n\calV$ or $\Sym^n\calV\left \{-\frac{a}{2}\right\}$ (if $\calD$ is not a tube domain), where $\{\ \}$ denotes the half-twist operation (cf.\ \cite{halftwist0}). 
\end{mainthm}

The case of weight three is described in detail in \S\ref{listrep} (see especially Corollary \ref{thmclassify3}). In Section \ref{cmsect}, we make some remarks about the more difficult problem of classifying irreducible $\Q$-variations of Hodge structure of Calabi--Yau type.

In Sections \ref{sectequations} and \ref{sectequations2}, we specialize further to the case of weight three and discuss various methods for constructing semi-algebraic variations of Hodge structure, not necessarily strongly quasi-projective. The local study of maximal weight three horizontal variations of Calabi-Yau type (with a certain non-degeneracy condition) dates back to work of Bryant-Griffiths \cite{bg} (see also (\cite{friedmancy} and \cite{voisinmirror}). The method of Section \ref{sectequations} is a global variant of \cite{bg} which gives a construction of semi-algebraic maximal horizontal subvarieties (see \eqref{eqalld2}) with a rational or real structure and with a (rational or real) unipotent  group action on the horizontal subvariety whose general elements are maximally unipotent.
These considerations lead to a homogeneous cubic polynomial $\varphi(z_1, \dots, z_h)$, where $h$ is the dimension of the horizontal subvariety.  However, for a general cubic polynomial $\varphi$, the horizontal subvarieties so constructed will have no symplectic automorphisms other than the unipotent group (Theorem \ref{thmstabilizersmooth}) and so will not be strongly quasi-projective. Additionally, we discuss the Hodge--Riemann bilinear relation in terms of the cubic $\varphi$ (Theorem \ref{thmhr}). In Section \ref{sectequations2}, we give a related   construction  which leads to complex polynomials analogous to $\varphi$, as well as a variant which leads to non-maximal variations (see \eqref{eqallddegen1} and \eqref{eqallddegen3} respectively), which are relevant in the non-tube case. 

In Section \ref{sect5}, we show that the weight three Hermitian symmetric examples can be described by the methods of Section \ref{sectequations2} in general (Theorem \ref{thmexplicit2}) and by those of Section \ref{sectequations} in the tube domain case (Theorem \ref{hermcubiccase}). Perhaps not surprisingly, the realization of Hermitian symmetric domains as horizontal subvarieties of a period domain of Calabi--Yau type is closely  related to the general theory of realizations of these symmetric domains.  Roughly speaking, the methods of Section  \ref{sectequations} correspond to the unbounded realizations of Hermitian symmetric spaces due to Kor\'anyi--Wolf \cite{korwolf}, while those of Section \ref{sectequations2} correspond to the Harish-Chandra embedding of a Hermitian symmetric space as a bounded domain. It is very likely that similar explicit constructions can describe the general weight $n$ case. However, the case of weight three, aside from being the simplest nontrivial case,  also has many connections with other geometric questions. For example, over $\bC$, the classification in the tube domain case is equivalent to the theory of homogeneous Legendrian varieties as studied extensively by Landsberg--Manivel (e.g.\ \cite{lm4}, \cite{lm1}). It is also related to Zak's classification of Severi varieties (see for instance \cite{lazarsfeld}). Finally, both of the exceptional Hermitian symmetric domains (namely type $\EIII$ and $\EVII$) appear as horizontal subvarieties of weight three variations of Hodge structure of Calabi--Yau type.

Lastly, in Section \ref{sectconclude}, we describe some of the interesting Hodge theory in the weight three tube domain case in terms of the cubic form $\varphi$ as well as the Hermitian symmetric space structure. In particular we analyze (1) the locus where the intermediate Jacobian of the weight three Hodge structure is an almost direct product, where one factor is a polarized abelian variety,  and (2) degenerations and limiting mixed Hodge structures. The discussion is not meant to be definitive in any sense.

Although there is a great deal of literature on the subject, we are not concerned in this paper with realizing the Hermitian symmetric  examples geometrically, i.e.\ as variations of Hodge structure associated  to a family of Calabi--Yau or other smooth projective varieties or via some motivic construction beginning with such a family. Some of the known Hermitian type examples of geometric nature include those constructed by  Borcea \cite{borcea} and Voisin \cite{voisincy}, and the more recent examples of  ball quotient type due to Rohde \cite{rohde}  and van Geemen and his coauthors \cite{vangeemen} (for further discussion see \S\ref{geometricex}). Such examples tend to be quite rare: the horizontal subvarieties associated to most geometric examples are far from being equivariantly embedded Hermitian symmetric. For example, if the Zariski closure of the monodromy group is the full symplectic group, or (as Deligne noted) contains monodromy transformations of Picard--Lefschetz type (see   Definition~\ref{defpl}), then the horizontal subvariety is not a Hermitian symmetric space equivariantly embedded in the classifying space. This fact rules out the quintic threefold and its mirror and tends to rule out most complete intersections in toric  varieties as well.  Similarly, Gerkmann et al.\ \cite{zuo} (see also \cite{yoshida} for related results) show that the moduli space of  Calabi--Yau threefolds obtained via double covers of $\bP^3$ branched in $8$ general planes is not Hermitian symmetric. 

\subsection*{Acknowledgements} We thank B. Hassett, M. Kerr, and K. O'Grady for discussions relevant to this  paper. We also thank the referees for useful comments that have helped improve the paper and for pointing out an inaccuracy in an earlier version. Finally, we are grateful to I. Dolgachev for suggesting some additional references and sharing with us some unpublished correspondence with E. Looijenga. 

\subsection*{Convention:} We abbreviate variation of Hodge structure by VHS. All VHS are polarizable/polarized, defined over $\bQ$ unless otherwise specified, and satisfy the Griffiths  transversality condition. A Hodge structure or VHS (of weight $n$) will be assumed to be effective ($h^{p,q} \neq 0$ only for $p, q\geq 0$, $h^{n,0}\neq 0$) unless otherwise noted. By a Tate twist, we can always arrange that a Hodge structure is effective. We denote by $\bD$ a Griffiths period domain. Thus, $\bD = \bG(\R)/\bK$, where $\bG$ is an orthogonal or symplectic group defined over $\Q$ (the group preserving a pair $(V, Q)$, where $V$ is a $\Q$-vector space and $Q$ is a non-degenerate symmetric or alternating form defined over $\Q$) and $\bK$ is a compact subgroup of $\bG(\R)$, not in general maximal.

\section{Semi-algebraic implies Hermitian symmetric}\label{sectalgebraic}
Our goal in this section is to prove Theorem~\ref{maintheorem}.
 
\begin{definition} Let $\bD=\bG(\R)/\bK$ be a classifying space for Hodge structures with compact dual $\check \bD = \bG(\Cee)/\bfP(\Cee)$, where $\bfP(\Cee)$ is an appropriate parabolic subgroup of $\bG(\Cee)$. A closed horizontal subvariety $Z$ of $\bD$ will be called \textsl{semi-algebraic in $\bD$} if $Z$ is an open subset of its Zariski closure $\hat Z\subseteq \check \bD$. Equivalently, there exists a closed subvariety $\hat Z$ of the projective variety $\check \bD$ such that $Z$ is a connected component of $\hat Z \cap \bD$. Note that, if $Z$ is semi-algebraic in $\bD$, then $Z$ is a semi-algebraic set.
\end{definition}

\begin{definition}\label{strongqp} 
Let $\bD=\bG(\R)/\bK$ be a classifying space for Hodge structures as above, and let $Z$ be a closed horizontal subvariety of $\bD$. Let $\Gamma = \Gamma_Z$ be the stabilizer of $Z$ in $\bG({\Zee})$, i.e.\ $\Gamma = \{\gamma \in \bG(\Zee): \gamma(Z) = Z\}$.
Thus $\Gamma$ acts properly discontinuously on $Z$. We call $\Gamma \backslash Z$ \textsl{strongly quasi-projective} if, for every subgroup $\Gamma'$ of $\Gamma$ of finite index, the analytic space  $\Gamma' \backslash Z$ is quasi-projective, and thus the morphism $\Gamma' \backslash Z \to \Gamma \backslash Z$ is a morphism of quasi-projective varieties. In particular, if $\Gamma \backslash Z$ is strongly quasi-projective, then $\Gamma \backslash Z$ is quasi-projective.
\end{definition}

\begin{remark} \begin{itemize} \item[(i)] If $\Gamma$ acts on $Z$ without fixed points and $\Gamma \backslash Z$ is quasi-projective, then by Riemann's existence theorem $\Gamma \backslash Z$ is automatically strongly quasi-projective.
\item[(ii)] If $\bD$ is Hermitian symmetric, so that the quotient of $\bD$ by every arithmetic subgroup admits a Baily-Borel compactification, then it is easy to check that $\Gamma \backslash Z$ is strongly quasi-projective  if and only if $\Gamma \backslash Z$ is quasi-projective.
\end{itemize}
\end{remark}

We can now restate Theorem~\ref{maintheorem} as follows:

\begin{theorem}\label{thmalgebraic}
Let $Z$ be a closed horizontal subvariety of a classifying space $\bD=\bG(\R)/\bK$ for Hodge structures and let $\Gamma=\Stab_Z\cap \bG(\bZ)$ as above. Assume that 
\begin{itemize}
\item[(i)] $\Gamma\backslash Z$ is strongly quasi-projective;
\item[(ii)] $Z$ is semi-algebraic in $\bD$. 
\end{itemize}
Then  $Z$ is a Hermitian symmetric domain $G(\R)/K$, whose embedding  in $\bD$ is an equivariant, holomorphic, horizontal embedding. In other words, $G$ is a closed algebraic subgroup of $\bG$ defined over $\Q$,   $K=\bK\cap G(\R)$ is a maximal compact subgroup, the complex structures on $Z=G(\R)/K$ and $\bD$ are induced by a morphism $U(1)\to G(\R)\subseteq \bG(\R)$, and $Z=G(\R)/K$ is a horizontal submanifold of $\bD$. 
 \end{theorem}
\begin{proof}
The embedding $Z\subseteq \bD$, induces a variation of Hodge structures $\sV$ on $S:=\Gamma\backslash Z$ with associated monodromy group $\Gamma$. Consider the generic Mumford--Tate group $M$ associated to $\sV$ and the derived subgroup $M'=M_{der}$.  Let $M_1=\overline{\Gamma}^0$ be the neutral connected component of the  Zariski closure of the monodromy group $\Gamma$. Note that $M_1$ is an algebraic subgroup of $\bG$ such that $M_1(\Cee)$ is invariant under $\operatorname{Gal}(\Cee/\bQ)$ and hence $M_1$ is defined over $\bQ$. By a theorem of  Deligne (see e.g.\ \cite[Theorem 6.19(a)]{milne}), $M_1\subseteq M'$. Since $S$ is quasi-projective, a theorem of Andr\'e (\cite[Theorem  1]{andre}) implies that $M_1$ is in fact a normal subgroup of $M'$.  Since $M'$ is   semi-simple,   $M'$  is an almost direct product of $M_1$ and another semi-simple subgroup $M_2\subseteq M'$, isogenous to $M'/M_1$. Without loss of generality,  possibly after passing to a  subgroup of $\Gamma$ of finite index, we can assume that 
\begin{itemize}
\item[i)] $\Gamma\subseteq M_1(\Cee)\cap \bG(\bZ)$ is a subgroup of  finite index, and
\item[ii)] $\Gamma\cap M_2(\Cee)=\{1\}$.   
\end{itemize}
By construction,  $\Gamma$ is a Zariski dense arithmetic subgroup of $M_1(\Cee)$. Note also that the real group $M_1(\R)$ is the intersection of $M_1(\Cee)$ with $\bG(\R)$.

Let $D_M\subseteq \bD$ be the associated Mumford--Tate subdomain, i.e.\ $D_M=M'(\R)\cdot z_0$, the $M'(\R)$-orbit of a general point $z_0\in Z$ (see \cite[Chapter II.B]{mtbook}). Then $D_M$ is a  homogeneous complex submanifold of $\bD$. Moreover, the Zariski closure of $D_M$ in $\check \bD=\bG(\bC)/\bfP(\bC)$ is the homogeneous space 
$$\check D_M=M'(\bC)/(M'(\bC)\cap \bfP(\bC))=M'(\bC)\cdot z_0,$$ 
and $D_M$ is an open subset of $\check D_M$ (\cite[Remark  on p. 56]{mtbook}, \cite[Prop. VI.B.11]{mtbook}).   Since $M$ is the generic Mumford--Tate group, it follows that $Z\subseteq D_M$ (\cite[\S6.15]{milne}, \cite[Theorem  III.A.1 (Step 3)]{mtbook}). Clearly, the Zariski closure $\hat Z$ of $Z$ in $\check \bD$ will then satisfy $\hat Z\subseteq \check D_M$. 

Let $z_0\in Z\subseteq D_M$ be a general reference Hodge structure. Since $M_1\times M_2\to M'$ is an isogeny, at the level of Lie algebras we have $\fm'=\fm_1\oplus \fm_2$. For the reference Hodge structure on $\fm'$ induced by $z_0$, the subalgebras $\fm_1$ and $\fm_2$ are sub-Hodge structures. Define $D_{M_i}=M_i(\R)\cdot z_0\subseteq \bD$ and similarly for $\check D_{M_i}\subseteq \check \bD$. 
 By arguing as in the proof of \cite[Theorem  III.A.1 (Step 3)]{mtbook} (see also \cite[Theorem  IV.F.1]{mtbook}), we conclude that the spaces $D_{M_i}$ and $\check D_{M_i}$ are complex submanifolds of $\check \bD$ and that the natural holomorphic map $\check D_{M_1}\times \check D_{M_2}\to \check D_M$ is a finite covering space. By passing to the adjoint forms of $M$ and hence $M_1$ and $M_2$, we may further assume that $\check D_{M_1}\times \check D_{M_2}\cong \check D_M$, and similarly that $D_{M_1}\times  D_{M_2}\cong D_M$. This has the effect of replacing $Z$ (and $\hat{Z}$) by a finite quotient  by a subgroup of the center of $M_1$, and it is clearly enough to prove the theorem for such a finite quotient.
 
  We first claim that the algebraic assumption on $Z$ gives 
\begin{equation}\label{claim1}
\hat Z\cap \check D_{M_1} =\check D_{M_1}.
\end{equation}
In fact, $\hat Z':=\hat Z\cap \check D_{M_1}$ is 
 an intersection of two closed algebraic subvarieties in $\check \bD$, and thus it is a closed algebraic subvariety in $\check \bD$. By assumption,  $\Gamma\cdot Z  = Z$. Hence $\Gamma\cdot \hat{Z}  = \hat{Z}$ and so $\Gamma$ preserves $\hat Z'$. Since $\hat Z'$ is algebraic, the  Zariski closure $\overline{\Gamma}^0(\bC)$ also stabilizes $\hat Z'$. Since $\overline{\Gamma}^0(\bC)=M_1(\bC)$ acts transitively on $\check D_{M_1}$,  \eqref{claim1} follows. 

By taking the intersection with $\bD$,  \eqref{claim1} implies that 
\begin{equation}\label{claim2}
D_{M_1}\subseteq Z,
\end{equation}
or at least that a connected component of $D_{M_1}$ is contained in $Z$. Hence the induced variation of Hodge structure on $D_{M_1}$ satisfies Griffiths transversality.  It then follows from a result of Deligne \cite{dshimura} that $D_{M_1}$ is a Hermitian symmetric domain (with a totally geodesic embedding in $\bD$) (see also \cite[Theorem  7.9]{milne}). To complete the proof, we will show that equality holds in \eqref{claim2}. 

Note that $Z\subseteq D_{M_1}\times D_{M_2}$ induces a projection morphism
$$\pi\colon\Gamma\backslash Z\to \Gamma\backslash D_{M_1}.$$  
Since $\Gamma$ is an arithmetic subgroup of $M_1$,  $\Gamma\backslash D_{M_1}$ is quasi-projective (by the Baily--Borel theorem) and the morphism $\pi$ is algebraic (by Borel's extension theorem). Furthermore, \eqref{claim2} gives a section of this map over the connected component of $\Gamma\backslash D_{M_1}$ containing the image of $\pi$. 

Let $F$ be a  fiber  of the morphism $\pi\colon\Gamma\backslash Z\to \Gamma\backslash D_{M_1}$. By the discussion of the previous paragraph, $F$ is a quasi-projective variety. On the other hand, 
$F=Z\cap \left(\{z_1\}\times D_{M_2}\right)$ for some $z_1\in D_{M_1}$. We conclude that $F$ is quasi-projective and carries a VHS with trivial monodromy (since $\Gamma\cap M_2(\Cee)=\{1\}$, since we have replaced $\Gamma$ by a suitable subgroup of finite index). The theorem of the fixed part as proved by Schmid then implies that the VHS on $F$ is constant.  Hence $F$ is  a finite set of points. We conclude that $\pi\colon \Gamma\backslash Z\to \Gamma\backslash D_{M_1}$ is a quasi-finite map of quasi-projective varieties.  It also has a section given by \eqref{claim2}. But then $D_{M_1}$ is a component of $Z$. Since $Z$ was assumed irreducible, it follows that $Z=D_{M_1}$ is a Hermitian symmetric domain as described in the statement of Theorem~\ref{thmalgebraic}.
\end{proof}

\begin{remark}  (i) In general, a finite ramified cover of a quasi-projective variety need not be quasi-projective, or even compactifiable (i.e.\ embedded in a compact analytic space as the complement of a closed analytic subset). However, we know of no examples of a  closed horizontal subvariety $Z$ of a classifying space $\bD$ such that $\Gamma\backslash Z$ is quasi-projective and such that there exists a finite cover $\Gamma' \backslash Z$ of $\Gamma\backslash Z$ as in Definition~\ref{strongqp} which is not quasi-projective. It is thus natural to ask if, in the statement of Theorem~\ref{thmalgebraic}, it is sufficient to assume only that $\Gamma\backslash Z$ is  quasi-projective.

\smallskip
\noindent (ii) More generally, suppose that, instead of assuming that $\Gamma\backslash Z$ is strongly quasi-projective, we assume  $\Gamma\backslash Z$ is \textsl{strongly compactifiable}, in other words that for every subgroup $\Gamma'$ of finite index in $\Gamma$, $\Gamma' \backslash Z$ is compactifiable, in such a way that the morphism $\Gamma' \backslash Z \to \Gamma \backslash Z$ extends to a meromorphic map on the compactifications. The proof of Theorem~\ref{thmalgebraic} then carries over to this case given the full strength of Schmid's result, that the theorem of the fixed part holds in this case and noting that Andr\'e's theorem on the monodromy only uses the fixed part assumption (see \cite[Theorem  6.19 ii)]{milne}). More precisely, in this case $Z$ is Hermitian symmetric, $\Gamma$ is an arithmetic subgroup of the corresponding Lie group $G(\R)$, the quotient $\Gamma\backslash Z$ is quasi-projective and a compactification of $\Gamma\backslash Z$ as above is unique up to bimeromorphic equivalence (via Borel's extension theorem). 

\smallskip
\noindent (iii)  There are also variants of Theorem~\ref{thmalgebraic} where we assume that $\Gamma\backslash Z$ is the image of a quasi-projective variety $S$ under a proper morphism, such that all finite covers $S\times _{\Gamma\backslash\bD}(\Gamma'\backslash\bD)$ are quasi-projective, and similarly where we just assume that $S$ and its related finite covers are compactifiable. In case $S$ is quasi-projective and the morphism $S\to \bG({\Zee})\backslash \bD$ comes from algebraic geometry, i.e.\ is induced by taking the Hodge structures on the primitive cohomology of a family $\pi \colon\mathcal{X} \to S$ of smooth projective varieties, or from a related construction, then the finite covers $S\times _{\Gamma\backslash\bD}(\Gamma'\backslash\bD)$ of $S$ are all \'etale, and hence quasi-projective, by the Riemann existence theorem, and so the proof of Theorem~\ref{thmalgebraic} applies in this situation.

\smallskip
\noindent (iv)  Finally,  one can ask if  Theorem~\ref{thmalgebraic} extends to the case where  $\Gamma\backslash Z$ is only assumed to have finite volume.
\end{remark}

\section{Classification of the Hermitian VHS of Calabi--Yau type}\label{sectclassify}
In this section, we classify the  VHS parametrized by Hermitian symmetric domains as in Theorem \ref{thmalgebraic} for Hodge structures of Calabi--Yau type.  A similar classification problem is the classical Shimura case, namely the classification of VHS of weight $1$ (polarized  abelian varieties) parametrized by Hermitian symmetric domains due to Satake  \cite{satake} and Deligne \cite[\S1.3]{dshimura} (see also \cite[Chapter 10]{milne}). Some examples of Hermitian VHS of  Calabi--Yau type were constructed by Gross \cite{bgross} and Sheng--Zuo \cite{sz}. Here, we complete the classification in the Calabi--Yau case (Theorem \ref{thmclassify} and Corollary \ref{thmclassify3} for weight $3$) by showing that all Hermitian VHS of Calabi--Yau type are derived from the examples of Gross and Sheng--Zuo. 

\begin{definition}\label{defherm}
Let $\bD$ be a classifying space of (polarized) Hodge structures. We say that a horizontal subvariety $\calD\hookrightarrow \bD$ is of \textsl{Hermitian type} if $\calD$ is a Hermitian symmetric domain of non-compact type embedded in $\bD$ via an equivariant, holomorphic, horizontal embedding. When $\calD\subset \bD$ is of Hermitian type, the induced variation of Hodge structures $\sV$ on $\calD$ is called \textsl{a Hermitian VHS}. 
\end{definition}
\begin{remark}
The Hermitian VHS are the VHS parametrized by Hermitian symmetric domains considered by Deligne \cite{dshimura}. Also, in the terminology of \cite{mtbook}, a subvariety $\calD\subset \bD$ of Hermitian type is the same thing as a  Mumford--Tate domain which is Hermitian symmetric and unconstrained. 
\end{remark}

\begin{definition}\label{defcy}
A \textsl{Hodge structure $V$ of Calabi--Yau (CY) type} 
is an effective weight $n$  Hodge structure such that $V^{n,0}$ is $1$-dimensional. If $n=1$, we say that that $V$ is \textsl{of elliptic curve type}, and if $n=2$ we say that $V$ is \textsl{of $K3$ type}.
\end{definition}

After some preliminaries on Hermitian symmetric domains and Hermitian VHS, we classify the Hermitian VHS of CY type in   \S\ref{proofclassify}. A detailed discussion of the Hodge representations that occur in this classification is   done in \S\ref{listrep}. Finally, in \S\ref{geometricex}, we review the known examples of Hermitian VHS of CY type coming from geometry. We will discuss some related rationality questions in Section \ref{cmsect}.

\subsection{Preliminaries on Hermitian VHS}\label{sectprelim}
We fix $\calD=G(\bR)/K$ an irreducible Hermitian symmetric domain of non-compact type and $z_0\in \calD$ a reference point. 
Let $\bar G=G/Z(G)$  be the adjoint form of $G$, and $\bar K\subset \bar G$  the corresponding compact subgroup. The choice of a reference point $z_0\in \calD$ determines the maximal compact subgroup $\bar K$ with  center $U(1)$. In particular, the choice of a reference point gives a cocharacter (which we fix throughout):
\begin{equation*}\varphi\colon U(1)\to \bar G.
\end{equation*} 

 Let $\fg$ be the corresponding (real) Lie algebra, $\fg=\fk\oplus\fp$ the Cartan decomposition, and $H_0$ the differential of $\varphi$. Note that $H_0$ spans the center $\mathfrak z$ of $\fk$, $\fp$ is identified with the tangent space of $\calD$ at $z_0$, and the $\ad(H_0)$ action on $\fp$ gives the complex structure (e.g.\ \cite[Theorem  7.117]{knapp}). We fix a compact Cartan subalgebra $\fh_0\subset \fk$ with $H_0\in \fh_0$. Let $\fh=\fh_0\otimes \bC\subset \fg_\bC$ be the Cartan subalgebra, and let $R$ be the associated root system. In this situation, the roots are either  compact or non-compact. Fix a good ordering of the roots (e.g.\ \cite[p. 441]{knapp}), and let   $\alpha_1, \dots, \alpha_d$ be the simple roots. Then there exist a single simple non-compact root $\alpha_i$ (\cite[p. 449]{knapp}). Thus, up to a factor of $2\pi \sqrt{-1}$ that we ignore in what follows,
\begin{equation}\label{mucond}
\left\{
\begin{matrix}
\alpha_i(H_0)=1,\\ 
\alpha_j(H_0)=0,& \textrm{if } j\neq i
\end{matrix}
\right.
\end{equation}
for simple roots. The root $\alpha_i$ is  a special root, in the following sense:

\begin{definition}\label{specialcomin} Let $R$ be an irreducible root system with a choice of simple roots $\alpha_1, \dots, \alpha_d$ and highest root $\tilde\alpha$, and write $\tilde \alpha = \sum_in_i\alpha_i$. Then $\alpha_i$ is a  \textsl{special root} if $n_i=1$. In fact,  the choice of $\calD$ is equivalent  to a root system $R$ and the choice of a special root $\alpha_i$  (e.g.\ $(R,\alpha_i)$ gives the Vogan diagram of the real Lie algebra $\fg$; see also \cite[Chapter 2]{milne}). If $\alpha_i$ is a special root and $\varpi_i$ is the corresponding fundamental weight (i.e.\ $\varpi_i(\alpha_j\spcheck) = \delta_{ij}$), then we call $\varpi_i$ a \textsl{cominuscule weight} and  the corresponding fundamental representation $V_{\varpi_i}$ of $G(\bC)$ with highest weight $\varpi_i$ a \textsl{cominuscule representation}.
\end{definition}

\begin{remark}\label{remminuscule}
The relevance of the cominuscule weights in the classification of Hermitian symmetric domains can also be understood as follows: The non-compact Hermitian symmetric domains $\calD=G(\bR)/K$ are in one to one correspondence with their compact duals $\check \calD=G(\bC)/P_0(\bC)$.  Then, the minimal homogeneous embedding $\check \calD$ in a projective space is given  by $\check \calD=G(\bC)\cdot[v]\hookrightarrow \bP(V_{\varpi_i})$, where $v$ is a highest weight vector in the cominuscule representation $V_{\varpi_i}$. In this picture, the parabolic subgroup $P_0(\bC)$ is just $\Stab_{G(\bC)}(v)$ (see also Section \ref{sect5}).
\end{remark}

 An embedding  $\calD\subset \bD$ as in  Theorem \ref{thmalgebraic} is induced by a representation 
\begin{equation*}
\rho\colon G\to \GL(V)
\end{equation*}
and a compatible polarization  $Q$ on $V$, so that $\rho(G)\subseteq \bG=\Aut(V,Q)$. Typically a compatible polarization exists and it is unique (see \cite[\S IV.A, (Step 4)]{mtbook}); thus the issue of polarization does not play a major role here. 
For our purposes, there is no loss of generality in assuming that $V$ is   irreducible  over $\bQ$. In what follows, we will assume that $V$ remains an irreducible representation over $\bR$ as well. We will discuss  the general case in Section \ref{cmsect}. Specifically,  we assume:

\begin{convention}\label{conventionired} 
$G_\bR$ is an almost simple and simply connected algebraic group of Hermitian type. We assume  that the group $G$ and the representation $\rho$ are defined over $\bQ$ and that the induced real representation $\rho\colon  G_\bR\to \GL(V)_\bR$ is irreducible.  
\end{convention}

\subsubsection{Hodge representations}\label{existHodgereps} Following Deligne \cite{dshimura} (see also \cite[Chapter 10]{milne} and \cite[Chapter IV]{mtbook}), a necessary and sufficient condition for the representation $\rho$ to arise from a VHS is that there exists a reductive group $M\subset \GL(V)$ defined over $\bQ$ (i.e.\ the generic Mumford--Tate group of the VHS) and a morphism of algebraic groups:
$$h\colon \bS\to M_\bR\subset \GL(V)_\bR$$
(where $\bS=\Res_{\bC/\bR}(\Gm)$ is the Deligne torus) such that: 
\begin{itemize}
\item[i)] $h$ defines a Hodge structure on $V$. 
\item[ii)] The representation $\rho$ factors through $M$ and $\rho(G)=M_{der}$. 
 \item[iii)] The induced morphism 
\begin{equation*}
\bar{h}\colon  \bS/\Gm \to M_{ad,\bR}=\bar G
\end{equation*}
 is conjugate to the cocharacter $\varphi\colon U(1)\to \bar G$ (i.e. $h$ is a lift of $\varphi$). 
 \end{itemize}

In the situation considered here, one has $\rho(G)\subset \SL(V)$ and thus  it suffices to restrict to the subgroup $\Hg=M\cap \SL(V)$ (thought of as the generic Hodge group of $\calV$). By abuse of notation, we still denote by $h\colon \U(1)\to \Hg_\bR$ the restriction $h_{\mid \U(1)}$,  where  $\U(1)\hookrightarrow \bS$ is the kernel of the norm map $\bS=\Res_{\bC/\bR}(\Gm)\to \Gm$. The Hodge decomposition on $V_\bC$  is simply the 
weight decomposition of $V_\bC$ with respect to $\U(1,\bC)\cong \bC^*$: $V^{p,q}$ corresponds to the eigenspace for the character $z^{p-q}$. Finally, note that $\Gm(\bR)\cap \U(1,\bR)=\{\pm 1\}$ as subgroups of $\bS(\bR)=\bC^*$, giving the following diagram:
\begin{equation}\label{diag1}
\xymatrix
{ 
 h\colon \bC^*\ \ \ \ar@{->}[r]\ar@{->}[d]^{2:1}&\Hg(\bC)\ar@{->}[d]\ar@{->}[r]&\GL(V_\bC)\\
\varphi\colon \bC^*\ \ \ \ar@{->}[r]&\bar G(\bC)\\
}.
\end{equation} 

\begin{remark}\label{rematate}
Since the representation $V$ is assumed irreducible, the lift of $\varphi$ to $M$ (versus $\Hg$) is equivalent to a Tate twist. However, with our definition of CY type (Definition  \ref{defcy}), no Tate twist is allowed.
\end{remark}

\subsubsection{Computation of the Hodge numbers} Modulo a factor of $2$ (explained by \eqref{diag1}), the computation of the weights of $h$ on $V_\bC$ can be done at the level of Lie algebras. Since $\Hg$ is reductive and $\rho(G)$ is the derived subgroup, we have 
\begin{equation}\label{liehodge}
\hg:=\Lie(\Hg)\cong \fg\oplus \fa,
\end{equation}
for some abelian Lie algebra $\fa$. 

Recall that a real irreducible representation $V_\bR$ is of three possible types: \textsl{real type}, \textsl{complex type}, or \textsl{quaternionic type} depending on the behavior of the complexification:
\begin{equation*}
V_\bC=\left\{
\begin{array}{ll}
V_{+}&\textrm{real type}\\
V_{+}\oplus V_{-},\ V_{+}\not\cong V_{-}&\textrm{complex type}\\
V_{+}\oplus V_{-},\ V_{+}\cong V_{-}&\textrm{quaternionic type}
\end{array}
\right.,
\end{equation*}
where $V_\pm$ are irreducible $G(\bC)$-representations and $V_\pm$  are dual representations (i.e.\ $V_+\spcheck\cong V_{-}$). Since 
\begin{equation}\label{comend}
\End_{\fg}(V_\bR)=\left\{
\begin{array}{ll}
\bR&\textrm{real type}\\
\bC&\textrm{complex type}\\
\bH&\textrm{quaternionic type}
\end{array}
\right.
\end{equation}
(e.g.\ \cite[p.\ 89]{mtbook}), we conclude that the abelian part $\fa$ is  trivial in the real case and at most one dimensional in the  complex or quaternionic case. Furthermore, in the complex or quaternionic case, when $V_\bC=V_{+}\oplus V_{-}$,  a generator of the one-dimensional $\fa$ corresponds to the obvious $\bC^*$-action: $t\in\bC^*$ acts as scaling by $t$ on $V_+$ and scaling by $t^{-1}$ on $V_{-}$. 

Since the projection of the differential of $h$ on the $\fg$ factor (see \eqref{liehodge}) is $H_0$, we conclude that in the real type case, the weights of $h$ on $V_\bC$ are: 
\begin{equation}\label{weightsreal}
\left\{ 2\varpi(H_0)\mid \varpi\in \sX(V_+)\right\},
\end{equation}
where  $\sX(V_+)$ denotes the weights of the irreducible $G(\bC)$-representation $V_+$ (N.B. $V_+\cong V_\bC$ in this case). In the complex or quaternionic case, if $\fa$ is $1$-dimensional, the weights can all be  shifted by a constant $c$. Since the weights of $V_+$ and $V_-$ are opposite and the shifts act  in opposite directions, we conclude that in the complex and quaternionic cases  the weights of $h$ on $V_\bC$ are: 
\begin{equation}\label{weightscx}
\left\{ \pm2(\varpi(H_0)-c)\mid \varpi\in \sX(V_+)\right\}
\end{equation}
for some constant $c$. In this situation, $c$ and $\varpi(H_0)$ can be (and typically are) rational numbers (with denominator dividing $2d$, where $d$ is the connection index of the root system $R$). At the same time, the difference $\varpi(H_0)-c$ is a half-integer.
This is explained by the fact that one might need to pass to a finite cover to lift $h$ to 
$$\tilde h\colon \bC^*\to G(\bC)\times\bC^*\xrightarrow{(\rho,\chi)}\GL(V_\bC)$$
so that it fits in the diagram  (extending \eqref{diag1}):
\begin{equation}\label{finitecover}
\xymatrix@!0{
& \bC^* \ar@{->}[rr]\ar@{->}'[d][dd]
& &G(\bC)\times \bC^* \ar@{->}'[d][dd]\ar@{->}[rrrd]^{(\rho,\chi)}
\\
\bC^* \ar@{<-}[ur]\ar@{->}[rr]\ar@{->}[dd]_{2:1}
& & \Hg(\bC) \ar@{<-}[ur]\ar@{->}[dd]\ar@{->}[rrrr]&&&&\GL(V_\bC)
\\
& \bC^* \ar@{->}'[r][rr]
& &  G(\bC)\ar@{->}[rrru]_{\rho}
\\
\bC^* \ar@{->}[rr]\ar@{<-}[ur]^>>{d:1}
& &\bar G(\bC) \ar@{<-}[ur]
}.
\end{equation}

\subsubsection{Half twists}
Gross \cite{bgross} has constructed an $\bR$-VHS of CY type on each of the Hermitian symmetric domains that are of tube type. Similarly, for the domains which are not of tube type, Sheng and Zuo \cite{sz} have constructed $\bC$-VHS of CY type. As we will see later, these examples can be understood in a uniform way. For now, we only note that the $\bC$-VHS cases can be modified to give rise to an $\bR$-VHS by the following recipe: Given an $\bR$-Hodge structure $V$ and a direct sum decomposition  $V_\bC=V_{+}\oplus V_{-}$ of $\bC$-Hodge structures such that  $\overline{V_{+}}=V_{-}$, one obtains $\bC$-Hodge structures on $V_{\pm}$ that are conjugate (i.e.\  $\overline{V_+^{p,q}}=V_{-}^{q,p}$). Conversely, given two conjugate $\bC$-Hodge structures $V_{+}$ and $V_{-}$, their direct sum $V_{+}\oplus V_{-}$ is naturally an $\bR$-Hodge structure. For example, as we shall see in Definition~\ref{weakCMdef},  given a Hodge structure $V$ of weight $n$ with weak CM by a CM field $E$ (and a choice of CM type for $E$), there is a natural (eigenspace) decomposition $V_\bC=V_{+}\oplus V_{-}$ such that  $\overline{V_{+}}=V_{-}$. In this situation, van Geemen \cite{halftwist0} noted that the Hodge structure on $V$ can be modified by a \textsl{half-twist} (see also \cite[\S1.4]{halftwist0} and \cite[Definition  V.B.1(v)]{mtbook}):

\begin{definition}
Suppose that $V_{\pm}$ are two conjugate $\bC$-Hodge structures as above and that $V$ is the corresponding $\bR$-Hodge structure. Given $a\in \Zee$, we define $V\left\{-\frac{a}{2}\right\}$ to be  the (not necessarily effective) Hodge structure of weight $n+a$ obtained by shifting the weights of $V_{\pm}$  (or equivalently twisting by a certain character) as follows:
\begin{equation}\label{halftwist}
V\left\{-\frac{a}{2}\right\}:=V_{+}\left\langle -\frac{a}{2}\right| \oplus V_{-}\left|-\frac{a}{2}\right\rangle,
\end{equation}
where for a  $\bC$-Hodge structure $W=\bigoplus_{p+q=k} W^{p,q}$, we define the left and right shifts by 
\begin{equation}\label{shift}
W\left\langle c\right|:=\bigoplus_{(p-2c)+q=k-2c} W^{p,q}, \textrm{ and resp. } W\left| c\right\rangle:=\bigoplus_{p+(q-2c)=k-2c} W^{p,q},
\end{equation}
i.e.\ $\left(W\left\langle c\right|\right)^{r,s}=W^{r+2c,s}$ and $\left(W\left| c\right\rangle\right)^{r,s}=W^{r,s+2c}$. The \textsl{half-twist} then corresponds to the case $a=1$. For a general $a\in \bZ$, we shall refer to $V\left\{-\frac{a}{2}\right\}$ as an \textsl{iterated half-twist}.
\end{definition}

\begin{remark}
One should not confuse $V\{c\}$ (or the shifts  $\langle\ |$ and $|\ \rangle$) with the Tate twist $V(c)$. In fact, applying a half-twist twice does not give a Tate twist, but applying a half-twist, followed by   complex conjugation (i.e.\ switch $V_+$ and $V_{-}$), followed by another half-twist, is the same as a Tate twist. Also note that the half-twist $\{\ \}$ and the Tate twist $(\ )$ are applied to $\bR$-Hodge structures, while the shifts $\langle\ |$ and $|\ \rangle$ are only applied to $\bC$-Hodge structures. For all four operations, the sign convention is the same as that for a Tate twist: the twist by $c\in \bZ$ (or $\frac{1}{2}\bZ$ respectively) decreases the weight by $2c$ for $\{\ \}$, $(\ )$, $\langle\ |$, and $|\ \rangle$.
\end{remark}

\begin{convention}\label{fractionalshift}
For a rational number $p/q$, the notations $W\left\langle \frac{p}{q}\right|$ and $W\left| \frac{p}{q}\right\rangle$ will have the same meaning as in \eqref{shift}, but should be understood in the context of Diagram \eqref{finitecover}. Namely, $\bC^*$ is  a one-parameter subgroup of  $G(\bC)\times_{\mu_q} \bC^*$, and the shift is by the character $\chi(t)=t^{-p}$ as in the diagram.
\end{convention}

\subsection{Classification of Hermitian VHS of CY type}\label{proofclassify}
We now specialize the discussion to the Calabi-Yau case. First, we note the following property for the cominuscule representation associated  to the Hermitian symmetric domain $\calD$.

\begin{lemma}\label{reptype}
Let $\calD=G(\bR)/K$ be a Hermitian symmetric domain of non-compact type, and $\varpi_i$ the associated cominuscule weight. As a $G_\bR$-representation the cominuscule representation $V_{\varpi_i}$ is of real type for the following cases  (see also table \ref{tabledomains} in \S\ref{listrep} for labeling): $\I_{n,n}$ $(A_{2n-1},\alpha_n)$, $\IV_{2n-1}$ $(B_n,\alpha_1)$, $\III_n$ $(C_n,\alpha_n)$, $\IV_{2n}$ $(D_n,\alpha_1)$, $\II_{2n}$ $(D_{2n},\alpha_{2n})$, $\EVII$ $(E_7,\alpha_7)$. Otherwise, $V_{\varpi_i}$ is of complex type. In particular, the quaternionic case does not arise.
\end{lemma}
\begin{proof}
If $V_{\varpi_i}$ is not self-dual, then the representation is of complex type. The dual representation has highest weight $\tau\varpi_i$, where $\tau$ is the opposition involution.  The condition  $\tau\varpi_i=\varpi_i$ holds exactly for the cases  listed  above. To complete the proof, if $\tau\varpi_i=\varpi_i$, one needs to decide if the representation is real or quaternionic. In our situation, it is standard that all these cases are real (e.g.\ \cite[\S2]{bgross}). Alternatively,  one easily checks the reality of the representation using the criterion given by \cite[Theorem  IV.E.4]{mtbook}. 
\end{proof}

\begin{remark}
The Hermitian symmetric domains for which the fundamental representation $V_{\varpi_i}$ is of real type are precisely the domains with a tube domain realization (compare \cite[\S1]{bgross}). We will refer to  these cases as \textsl{tube type} or \textsl{real type}, while the others will be referred as \textsl{non-tube} or \textsl{complex} cases.
\end{remark}

We can now state the following classification result, an expanded version of Theorem \ref{thm2intro} of the introduction:

\begin{theorem}\label{thmclassify}
For every irreducible Hermitian symmetric domain $\calD$, there exists a canonical $\bR$-variation  of Hodge structures $\sV$ of CY type parametrized by $\calD$. The VHS $\sV$ is uniquely determined by the requirement that its weight is  minimal. Specifically,
\begin{itemize}
\item[i)] if $\calD$ is of tube domain type,  $\sV$ is the variation constructed by Gross \cite{bgross}, and its weight is equal to the real rank of $\calD$;
\item[ii)] if $\calD$ is not of tube domain type, $\sV$ is obtained by taking the sum $(\mathscr W\langle c|)\oplus ({\mathscr W} \spcheck |c\rangle)$
 of the $\bC$-VHS  $\mathscr W$ constructed by Sheng--Zuo \cite{sz} with the dual VHS $ {\mathscr W}\spcheck$ and applying an appropriate twist. The minimal weight giving a VHS of CY type is equal the real rank of $\calD$ plus $1$.
\end{itemize}
 Any other irreducible $\bR$-VHS of CY type on $\calD$ can be obtained from the canonical $\calV$ by taking the unique irreducible factor of $\Sym^n\calV$ of CY type, or,  in the non-tube domain case,  by taking the unique irreducible factor of $\Sym^n\calV\left\{-\frac{a}{2}\right\}$ of CY type for appropriate  integers  $a$. 
\end{theorem}

The key ingredient of the theorem is the following lemma which shows that the relevant Hodge representations in the Hermitian CY case are the cominuscule representations.

\begin{lemma}\label{mainlemma} 
With notation  as in \S\ref{sectprelim}, let $\lambda$ be the highest weight of an irreducible factor $V_+$ of $V$. Possibly after replacing $V_+$ with $V_{-}$, we can assume that $\tau\lambda(H_0)\le \lambda(H_0)$. 
Then a necessary condition that the irreducible representation $\rho\colon G\to\GL(V)$ arises from a Hermitian VHS of CY type over $\calD$ is 
\begin{equation}\label{cycond}
\varpi(H_0)<\lambda(H_0) \textrm{ for all weights } \varpi\neq \lambda \textrm{ of } V_+.
\end{equation} 
 Furthermore, the condition \eqref{cycond} is equivalent to the condition that $\lambda$ is a multiple of the fundamental cominuscule weight $\varpi_i$ associated to the domain $\calD$. In particular, as a $G_\bR$-representation,  $V$ is   real if $\calD$ is of tube type and  complex if $\calD$ is not of tube type, and the quaternionic case does not arise.
\end{lemma}
\begin{proof}
Since all the weights of $V_+$ are obtained from $\lambda$ by subtracting positive roots,  it follows that
$$\max_{\varpi\in \sX(V_+)} \varpi(H_0)=\lambda(H_0).$$
Using the description of the weights of $h$ on $V_\bC$ (cf.\ \eqref{weightsreal} and \eqref{weightscx}), we see that $\dim_\bC V^{n,0}=1$ implies that the above maximum is attained only for the weight $\lambda$ (i.e. \eqref{cycond}). By applying the reflection in the root $\alpha_j$, and using \eqref{mucond}, we get 
$$s_{\alpha_j}(\lambda)(H_0)=\left(\lambda-\lambda(\check\alpha_j)\cdot \alpha_j\right)(H_0)=\lambda(H_0),$$
(where $\check\alpha_j$ is the corresponding coroot).  Since $s_{\alpha_j}(\lambda)\in \sX(V_+)$, from \eqref{cycond}, we conclude that $s_{\alpha_j}(\lambda)=\lambda$  for all $j\neq i$, i.e.\ $\lambda(\check\alpha_j)=0$ for all $j\neq i$, which is equivalent to $\lambda=n\varpi_i$. The last assertion follows from Lemma \ref{reptype}.
\end{proof}

\begin{proof}[Proof of Theorem \ref{thmclassify}.]
First consider the tube domain case. By Lemma \ref{mainlemma}, if $V$ is a Hermitian VHS of CY type, then there is an irreducible summand $V'$ (over $\Cee$) of $V$ occurring in  $V$ with highest weight $n\varpi_i$. But then $V'$ is a real representation, so that $V'=V$. The fact that $V$ actually arises follows from Gross \cite{bgross}. More precisely, for $n=1$, the construction  of the Hermitian VHS  $\calV$ over the tube domain $\calD$ is the content of \cite{bgross}. For $n>1$, the (not in general irreducible) VHS $\Sym^n\calV$ is of CY type, and there is an irreducible summand that will have highest weight $n\varpi_i$  and will remain of CY type. Finally, we have already noted that, in the real case, the Hodge group has to coincide with $G$ (see  \eqref{comend}). Thus, once the representation $\rho$ is fixed, the only possible variation in lifting $\varphi$ to $h$  is a Tate twist, which  in turn is  not allowed by Definition \ref{defcy} (see also Remark \ref{rematate}).  This completes the proof in the tube domain case.

In the non-tube domain case, a Hermitian VHS of CY type over $\calD$ gives a complex representation $V$ with $V_\bC=V_+\oplus V_{-}$. After reordering, we may assume that the highest weight for $V_+$ is equal to $n\varpi_i$. Note that $V_+$ is not real. Let $W$ be the representation of $G$ corresponding to $n=1$.  Sheng--Zuo \cite{sz} have noted that $W$ carries a $\bC$-Hodge structure, and hence the vector space $W\oplus W\spcheck$ will carry an $\bR$-Hodge structure. Typically, this structure will not be of CY type. From the description of the weights \eqref{weightscx}, it is immediate to see that after an iterated half-twist, one can arrange $V_+\langle c|\oplus V_{-}|c\rangle$ to be of CY type for all $n>0$. Specifically, for $\lambda=n\varpi_i$,
\begin{equation}\label{minc}
c=\frac{1}{2}\left(\lambda(H_0)-\tau\lambda(H_0)\right)-\frac{1}{2}
\end{equation}
will give a CY Hodge structure of minimal weight (compare \S\ref{listrep}). Any additional $-\frac{1}{2}$ half-twists will preserve the CY condition, but will increase the weight by $1$. The additional half-twists are explained by the fact that the (generic) Hodge group satisfies $\Hg(\bC)=G(\bC)\cdot \bC^*$ (in contrast to the real case, where $\Hg=G$). The minimal weight satisfying the CY condition is equal to the real rank of $\Hg$, and thus is one larger than the real rank of $G$. Finally, it is easy to see that there always   exists a compatible polarization (cf.\ \cite[\S IV.A (Step 4), p.\ 90]{mtbook}): in the complex case there exists an invariant Hermitian form on $V_+$, which gives both an alternating and a symmetric form on $V$ (the correct choice depending on the half-twist, or equivalently the weight).   
\end{proof}

\begin{remark}
If $\calD$ is a product $\calD_1\times\calD_2$ of Hermitian symmetric domains, then a Hermitian VHS $\calV$ on $\calD$ decomposes as a product $\sV_1\otimes\sV_2$. It is clear that $\calV$ is of CY type iff each $\calV_i$ is  of CY type.  
\end{remark}

\subsection{List of Hodge representations of CY type}\label{listrep}
Here we discuss in more detail the canonical Hermitian VHS of CY type given by Theorem \ref{thmclassify}.  Let $\calD$ be an irreducible Hermitian symmetric domain corresponding to $(R,\alpha_i)$, where $R$ is a root system and $\alpha_i$ is a special root. Let $V=V_{\varpi_i}$ be the  cominuscule representation giving the canonical VHS of CY type. 
 Let $\tau=-w_0$ denote the opposition involution. As before, we distinguish two cases: real (or tube) case if $\tau\alpha_i=\alpha_i$, and complex otherwise. We denote by $V_+$ the irreducible summand of the $G(\bC)$-representation $V_\bC$ of highest weight $\varpi_i$.  In the complex case, $V_\bC=V_+\oplus V_-$ and $V_-$ has highest weight $\tau\varpi_i$ (another cominuscule weight).

The set of roots $\Delta$ of $R$ are partitioned into compact roots $\Delta_c$
and non-compact roots $\Delta_{nc}$. If, as before (see \S\ref{proofclassify}), $\fg=\fk\oplus \fp$ is the Cartan decomposition and $\fh$ is a compact Cartan subalgebra, we have
\begin{eqnarray*}
\fk_\bC&=&\fh\oplus \bigoplus_{\alpha\in \Delta_c} \fg_\alpha=\bC\cdot H_0\oplus \fg'_\bC,\\
\fp_\bC&=&\bigoplus_{\alpha\in \Delta_{nc}}  \fg_\alpha,
\end{eqnarray*}
where $\fg'$ denotes the semi-simple part of the Lie algebra $\fk$. The Dynkin diagram of the root system $R'$ associated to $\fg'_\bC$ is obtained from that of $R$ by deleting the node corresponding to $\alpha_i$.  We are interested in the weight decomposition of $V_+$ with respect to the cocharacter $\varphi$ (corresponding to $H_0$). Using \eqref{mucond}, one immediately sees (compare also Lemma~\ref{mainlemma}, and  \cite[\S3]{moonenmt2}):
\begin{itemize}
\item[(1)] The weights of $\varphi$ on $W$ are: $\varpi_i(H_0), \varpi_i(H_0)-1,\dots, w_0\varpi_i(H_0)$. Thus, the minimal weight  of a Hodge structure on $V$ (not necessarily of CY type) will be $k:=\varpi_i(H_0)+\tau(\varpi_i)(H_0)\in \bZ_+$.  Note that $k$ coincides with the real rank of $G_\bR$ and can be computed as the sum of the coefficients of the roots $\alpha_i$ and $\tau\alpha_i$ in the weight $\varpi_i$. For example, for $(E_7,\alpha_7)$: $k=2\cdot\frac{3}{2}=3$, and for $(E_6,\alpha_1)$: $k=\frac{4}{3}+\frac{2}{3}=2$.
\item[(2)] The weight space corresponding to the weight $\varpi_i
(H_0)$ is $1$-dimensional (the CY condition). 
\item[(3)] The weight space corresponding to the weight $\varpi_i
(H_0)-1$ (corresponding to $V_+^{k-1,1}$) has dimension equal to half the number of non-compact roots (i.e.\ $\frac{1}{2}(|\Delta_R|-|\Delta_{R'}|)$). For example, for $E_7$, we get $h^{2,1}=\frac{1}{2}(126-72)=27$.
\end{itemize}

\smallskip
 
The following table gives a list of the Hermitian symmetric domains in terms of the root datum (the roots are labeled as in Bourbaki \cite{bourbaki}).
\begin{table}[htb]
\begin{tabular}{ll|ll|ll}
Label&$(R,\alpha_i)$&$G(\bR)$&$K$&$\bR$-rank\\
\hline
$\I_{p,q}$&$(A_{p+q-1},\alpha_p)$&$\SU(p,q)$&$\mathrm{S}(\U(p)\times \U(q))$&$\min(p,q)$\\
\hline
$\II_{n}$&$(D_{n},\alpha_{n})$&$\SO^*(2n)$&$\U(n)$&$\left[\frac{n}{2}\right]$\\
\hline
$\III_n$&$(C_n,\alpha_n)$&$\Sp(n,\bR)$&$\U(n)$&$n$\\
\hline
$\IV_{2n-1}$&$(B_n,\alpha_1)$&$\Spin(2,2n-1)$&$\Spin(2)\times_{\mu_2} \Spin(2n-1)$&$2$\\
$\IV_{2n}$&$(D_{n+1},\alpha_1)$&$\Spin(2,2n)$&$\Spin(2)\times_{\mu_2} \Spin(2n)$&$2$\\
\hline
$\EIII$&$(E_6,\alpha_1)$&$\E_{6,2}$&$\U(1)\times_{\mu_4}\Spin(10)$&$2$\\
$\EVII$&$(E_7,\alpha_7)$&$\E_{7,3}$&$\U(1)\times_{\mu_3}\E_6$&$3$
\end{tabular}
\vspace{0.1cm}
\caption{Hermitian symmetric domains of non-compact type}\label{tabledomains}
\end{table}

The real case (or equivalently $\calD$ is of tube domain type) is discussed in detail by Gross \cite{bgross}.  For completeness, we list the cases with some relevant information: 
\subsubsection{} $\I_{n,n}\ (A_{2n-1}, \alpha_n)$:  weight $2\varpi_n(H_0)=n$; $V_\bC=\bigwedge^n S$, where $S$ is the standard representation; for $n=3$, $h^{2,1}=9$.
\subsubsection{} $\IV_{2n-1}\ (B_n, \alpha_1)$: weight $2$; $V$ is the standard representation;  this is the classical case of $K3$ type.
\subsubsection{} $\III_n\ (C_n, \alpha_n)$: weight $n$; $V$ is an irreducible factor of $\bigwedge^n S$, where $S$ is the standard representation; for $n=3$, $h^{2,1}=6$; this case corresponds geometrically to the primitive middle cohomology of an abelian $n$-fold.
\subsubsection{} $\IV_{2n-2}\ (D_n, \alpha_1)$:  weight $2$; $V$ is the standard representation;  this is the classical case of $K3$ type.
\subsubsection{} $\II_{2n}\ (D_{2n}, \alpha_{2n})$: weight $n$; $V$ is a half-spin representation; for $n=3$, $h^{2,1}=15$.
\subsubsection{} $\EVII\ (E_7, \alpha_7)$: weight $3$; $V$ is the  cominuscule representation; $h^{2,1}=27$.

\medskip

The remaining cases are of complex type. The $\bC$-Hodge structure induced by (a lift of) $\varphi$ on $V_+$ was studied by Sheng--Zuo \cite{sz}. We are interested in   $\bR$-Hodge structures on $V$ of CY type. As explained above, these can be obtained by considering $V_+\langle c|\oplus V_-|c\rangle$ for an appropriate shift $c$. It is interesting to note that the minimal weight Hodge structure associated to the representation $V$ is not of CY type. The CY Hodge structure on $V$ is obtained by applying a half-twist to this minimal weight Hodge structure.  The relevant cases are:

\subsubsection{}  $\I_{p,q}\ (A_{p+q-1},\alpha_p)$ for $p<q$:  Note first $\tau \alpha_p=\alpha_{q}$, $P/Q\cong \bZ/(p+q)$, where $Q$ and $P$ are the root and weight lattices respectively.  We have 
$$\varpi_p(H_0)=\frac{pq}{p+q},\ \ 
\tau (\varpi_p)(H_0)=\frac{p^2}{p+q},$$
and thus the minimal weight for a Hodge structure on $V$ will be $k=\varpi_p(H_0)+\tau (\varpi_p)(H_0)=p$. To obtain a CY Hodge structure, the minimal weight will be $p+1$. Specifically,  the weights of the cocharacter $\varphi_\bC$ on $V_+$ are: 
$$\frac{pq}{(p+q)}, \frac{pq}{(p+q)}-1,\dots,-\frac{p^2}{p+q}$$
while those on the dual representation $V_-$ are:
$$\frac{p^2}{(p+q)}, \frac{p^2}{(p+q)}-1,\dots,-\frac{pq}{p+q}.$$
To obtain a (minimal weight) Hodge structure, one needs to shift the weights so that $\frac{pq}{(p+q)}-c=\frac{p^2}{(p+q)}+c$ (recall that the Hodge group $\Hg(\bC)=G(\bC)\cdot \bC^*$ in the complex case; see Convention~\ref{fractionalshift} for the meaning of the fractional twist; also compare to \eqref{minc}). We get that $V_+\left\langle\frac{p(q-p)}{2(p+q)}\right|\oplus V_-\left|\frac{p(q-p)}{2(p+q)}\right\rangle$ will carry a Hodge structure of weight $p$. By applying a half-twist, we obtain a CY Hodge structure of weight $p+1$: 
\begin{equation}\label{twistapq}
V_+\left\langle\frac{p(q-p)}{2(p+q)}-\frac{1}{2}\right|\oplus V_-\left|\frac{p(q-p)}{2(p+q)}-\frac{1}{2}\right\rangle.
\end{equation}
The Hodge numbers are easily computed by noting that $V_+=\bigwedge^p S$, where $S$ is the standard representation of $\SU(p,q)$. 

We are particularly interested in the weight $3$ CY case. Since we need $p+1\le 3$, we get $p=1$ or $p=2$. If $p=1$, the associated domain $I_{1,q}$ is the $q$-dimensional complex ball. The minimal weight for a Hermitian VHS is $1$ in this case, and the minimal weight for a Hermitian VHS of CY type is $2$ ($K3$ type). To obtain a weight $3$ VHS, we need to apply an additional half-twist. It is immediate to see that  $h^{2,1}=q$. This case is somewhat special among the complex cases, in the sense that $V_+=V^{3,0}\oplus V^{2,1}$ and $V_-=V^{1,2}\oplus V^{0,3}$, i.e.\ there is no ``mixing'' of $V_+$ and $V_-$ in $V^{2,1}$.  Also note that the VHS of weight $3$ is of maximal dimension, i.e.\ $\dim \calD=h^{2,1}$. If $p=2$, the minimal weight for a Hermitian VHS of CY type over $\I_{2,q}$ is $3=p+1$. Since $V_+=\bigwedge^2 S$, we get the dimensions of the weight spaces for $V_+$ (w.r.t. $\varphi$) to be $1$, $2q$, $\frac{q(q-1)}{2}$. Thus, 
$$h^{2,1}=h^{2,1}_++h^{2,1}_-=2q+\frac{q(q-1)}{2}=\frac{q(q+3)}{2}$$ for the resulting weight $3$ Hodge structure (where $h^{2,1}_\pm$ denotes the dimension of $(2,1)$ space on $V_{\pm}$ after the shift). Since the dimension of $\I_{2,q}$ is $2q$, we see that the Hermitian variation of Calabi--Yau type on $I_{2,q}$ is not maximal for $q>1$.

\subsubsection{} $(D_{2n-1},\alpha_{2n-1})$: $\tau \alpha_{2n-1}=\alpha_{2n-2}$, $P/Q\cong \bZ/4$. We have 
$$
 \varpi_{2n-1}(H_0)=\frac{2n-1}{4},\ \ 
\tau (\varpi_{2n-1})(H_0)=\frac{2n-3}{4}.$$
Thus the weights on $V_+$ are: $\frac{2n-1}{4},\frac{2n-1}{4}-1,\dots,-\frac{2n-3}{4}$. The minimal twist that gives a CY Hodge structure  is
$$V_+\left\langle-\frac{1}{4}\right|\oplus V_-\left|-\frac{1}{4}\right\rangle$$
and it has weight $n$. The representation $V_+$ is a half-spin representation. For $n=3$, $h^{2,1}=2^{2n-2}-1=15.$

\subsubsection{}  $(E_6,\alpha_1)$: $\tau \alpha_1=\alpha_6$, $P/Q\cong \bZ/3$. We have
\begin{equation*}
\varpi_{1}(H_0)=\frac{4}{3},\ \  \tau \varpi_1(H_0)=\frac{2}{3}.
\end{equation*}
The minimal twist that gives a CY Hodge structure  is
$$V_+\left\langle-\frac{1}{6}\right|\oplus  V_-\left|-\frac{1}{6}\right\rangle$$
of weight $3$. More explicitly, consider the direct sum $V_+\oplus V_-$ of the two (co)minuscule representations of $E_6$. There is a lift  $\tilde \varphi $ of the cocharacter $\varphi$ of $\bar G$  to the simply connected form $G$ so that it acts with weights (and dimensions of weight spaces):
$$\begin{matrix}
\textrm{weight}& 8&4&2&-2&-4&-8\\
\hline
V_+&1&&16&&10\\
V_- &&10&&16&&1\\
\hline
V_{\bC}&1&10&16&16&10&1
\end{matrix}$$
(N.B. the minimal lift of $\varphi$ to $G$ would act with weights $4,2,1,-1,-2,-4$. For our purposes,  we need to take a double cover of it; compare to diagram \eqref{diag1}.) We then consider the $1$-parameter subgroup given by
\begin{eqnarray*}
\bC^*&\to& G(\bC)\times \bC^*\\
t&\to&(\tilde {\varphi}(t), t)
\end{eqnarray*}
where $t\in \bC^*$ acts on $V_+\oplus V_-$ by the character $t$ on $V_+$ and $t^{-1}$ on $V_-$. The twisted weights (and dimensions of weight spaces) will be 
$$\begin{matrix}
\textrm{weight}& 3&1&-1&-3\\
\hline
V_{\bC}&1&26&26&1
\end{matrix}$$
i.e.\ we obtain a CY Hodge structure of weight $3$ with $h^{2,1}=26$. As in the other complex cases, it is possible to obtain a Hermitian VHS over the $\EIII$ domain of  weight $2=\rank_\bR \EIII$, but this is no longer of CY type (it has Hodge numbers $11,32,11$). 
Also note that $K=\U(1)\times_{\mu_4}\Spin(10)$. Then the dimensions of the $V^{p,q}$ (and $V_{\pm}^{p,q}$) spaces can be also computed by decomposing $V_+$ as a $\Spin(10)$-representation. Explicitly, we have 
$V_+\cong \bC\oplus W_{\varpi_1}\oplus W_{\varpi_5}$ 
as a $D_5$-representation (see also \cite[\S4.5]{sz}). Here the first summand is the standard representation, while the second is the half-spin representation of $\Spin^*(10)$. 

\medskip

\begin{notation} We encode a Hermitian VHS of CY type by $(R,\alpha_i;\lambda)\{c\}$, where   $(R,\alpha_i)$ determines the domain $\calD$, $\lambda$ is the highest weight of the irreducible complex representation $V_+$, and, in the complex case, $c$ denotes a half-twist (see \eqref{halftwist}). The possible fractional meaning of $c$ is explained by Convention~\ref{fractionalshift}.
\end{notation}

In summary, we obtain the following  Hermitian VHS of CY type for weight $3$. 

\begin{corollary}\label{thmclassify3}
The Hermitian VHS of CY type for  weight $3$ are:
\begin{itemize}
\item[i)] Four primitive real cases:  $\I_{3,3}$ $(A_5,\alpha_3;\varpi_3)$, $\III_3$ $(C_3,\alpha_3;\varpi_3)$, $\II_6$  $(D_6,\alpha_6;\varpi_6)$, and $\EVII$ $(E_7, \alpha_7; \varpi_7)$, corresponding to the weight $3$ cases in Gross \cite{bgross}.
\item[ii)] Re-embeddings of lower weight cases:  $\mathfrak{H}$ $(A_1,\alpha_1; 3\varpi_1)$. 
\item[iii)] Complex cases:  two infinite series: $\I_{1,n}$ $(A_{n},\alpha_1;\varpi_1)\left\{-\frac{n+3}{2(n+1)}\right\}$ and $\I_{2,n}$ $(A_{n+1},\alpha_2;\varpi_2)\left\{\frac{n-6}{2(n+2)}\right\}$, and two isolated case:  $\II_5$ $(D_5,\alpha_5; \varpi_5)\left\{-\frac{1}{4}\right\}$, and $\EIII$ $(E_6,\alpha_1; \varpi_1)\left\{-\frac{1}{6}\right\}$.

\item[iv)] Reducible cases: 
$\mathfrak{H}\times \IV_n$ (in particular $\mathfrak{H}\times \mathfrak{H}\times \mathfrak{H}$) or $\mathfrak{H}\times \I_{1,n}$. \end{itemize}
\end{corollary}

\begin{remark}\label{remmaxvar}
Note that the cases that give maximal horizontal subvarieties (i.e.\ $h^{2,1}=\dim\calD$) are those of (i), (ii), (iii) for $I_{1,n}$, and (iv) for $\mathfrak{H}\times \mathcal{D}_n$. Most of the remaining cases (all of complex type) can be embedded into maximal weight three Hermitian VHS over some other Hermitian domain (of real type) and be understood as Noether--Lefschetz subloci with weak complex multiplication. Specifically, if $\calD$ carries a VHS of CY type, and $\calD'\hookrightarrow \calD$ is a totally geodesic embedding, then, by restriction, $\calD'$ also carries a VHS of CY type. Satake \cite{satake} and Ihara \cite{ihara} have classified all the holomorphic, totally geodesic embeddings $\calD'\hookrightarrow \calD$ among Hermitian symmetric domains. In particular, we get the following  embeddings of the complex cases into maximal VHS:
\begin{itemize}
\item[a)] $\I_{2,6}\hookrightarrow \EVII$;
\item[b)] $\II_5\hookrightarrow \II_6$;
\item[c)] $\EIII\hookrightarrow \EVII$;
\item[d)] $\mathfrak{H}\times \I_{1,n}\hookrightarrow \mathfrak{H}\times \IV_{2n}$ (induced from  $\I_{1,n}\hookrightarrow \IV_{2n}$).
\end{itemize}
The case $\I_{2,n}$ ($n>6$) does not embed in a maximal weight three Hermitian VHS of CY type. 
\end{remark}

\subsection{Hermitian VHS of geometric origin}\label{geometricex} 
While typically the VHS of Calabi-Yau type that occur in the geometric context are not of Hermitian type, a number of interesting Hermitian VHS of CY type of geometric origin are known:

\subsubsection{Type $\III_n$} The middle cohomology of abelian $n$-folds gives a weight $n$ VHS of CY type parametrized by the Siegel upper half space $\mathfrak{H}_n$. 

\subsubsection{Type $\I_{n,n}$} It is well known that the domain $\I_{n,n}$ parameterizes abelian varieties $A$ of Weil type of dimension $g=2n$ (i.e. abelian varieties with weak CM multiplication by a degree $2$ CM field $E$ (see Definition \ref{weakCMdef}), and such that the induced unitary group is $\SU(n,n)\subset \Sp(2g)$). It is not hard to see that the Hodge structure $\bigwedge^n H^1(A)$ contains a Hodge substructure of Calabi-Yau type, giving a motivic realization for the Hermitian VHS of CY type associated to $\I_{n,n}$. For instance, this follows from our classification Theorem \ref{thmclassify} and some standard representation theory (see \S\ref{hodgestar} and the proof of Theorem \ref{classify3q} for a discussion of the key ingredients in a more general setup). A detailed study of the cases $n=2,3$ is done in \cite{lombardo} and \cite{filippini} respectively. 

\subsubsection{Type $\I_{1,1}\times \IV_n$} Borcea \cite{borcea} and Voisin \cite{voisincy} have constructed families of Calabi--Yau threefolds such that the associated VHS is parametrized by $\mathfrak{H}\times \calD$, where $\mathfrak{H}$ is the upper half plane and $\calD$ is a Type $\IV$ domain. The Calabi--Yau threefolds are obtained by resolving $(E\times S)/\langle \tau\rangle$, where $E$ is an elliptic curve, $S$ is a $K3$ surface and $\tau$ an involution (acting diagonally).
\subsubsection{Type $\I_{1,n}$} A slight modification of  the Borcea--Voisin construction, by considering higher order automorphisms, leads to VHS associated to Calabi--Yau threefolds that are parametrized by complex balls $\calB_n$ (e.g.\ Rohde \cite{rohde}, and Garbagnati--van Geemen \cite{vangeemen}). 
\begin{remark}
In addition to the Calabi-Yau examples mentioned above, the case $\I_{1,n}$ occurs in several other geometric examples. For instance, the embedding of $\I_{1,n}$ in a Siegel space, corresponding to the minimal weight Hermitian VHS on $\I_{1,n}$,    occurs in the Deligne--Mostow  uniformization of the moduli of points in $\bP^1$ by complex balls (\cite{delignemostow}). Kondo has realized most of these examples via embeddings of $\I_{1,n}$ into Type $\IV$ domains by using $K3$ surfaces, corresponding to the minimal weight Hermitian VHS on $\I_{1,n}$ of Calabi--Yau type (see \cite{dk}).  The half-twist construction explains (and was motivated by) these examples (see \cite{halftwist0}).
\end{remark}

\subsubsection{Type $\Sym^2 \IV_n$} It is interesting to note that   cases of type  $\Sym^n V$ occur in geometric situations. We thank K. O'Grady for informing us of the following example (unpublished). Let $X\subset \bP^5$ be a generic EPW sextic (cf.\ \cite{ogrady}).  It is known that $X$ is singular along a surface $S$ and that the resolution $\widetilde X$ is a Calabi--Yau fourfold. There exists a smooth double cover $Y\xrightarrow{2:1}X$ branched only along $S\subset X$ which is a  
hyperk\"ahler manifold (cf.\ \cite{ogrady}). It is well known that  $V=H^2_0(Y)$ is of $K3$ type, leading to a VHS  $\calV$ parametrized by a $20$-dimensional type  $\IV$ domain $\calD$. On the other hand, $H^4_0(Y)$ is essentially $\Sym^2 H_0^2(Y)$, and then an irreducible sub-Hodge structure of $H^4_0(\widetilde X)$ is isogeneous to $H^4_0(Y)$. Thus,  the   periods of the Calabi--Yau $4$-folds $\widetilde X$ are parametrized by a Hermitian symmetric domain $\calD$, and the horizontal embedding
$\calD\hookrightarrow \bD$  corresponds to the VHS $\Sym^2\calV$.

\section{Notes on the classification over $\bQ$ of Hermitian VHS of CY type} \label{cmsect}
Theorem \ref{thmclassify} is a classification of Hermitian VHS of CY type over $\bR$, or over $\bQ$ under the additional assumption \ref{conventionired} that the irreducible representation $V$ of $G$ is defined over $\bQ$ and remains irreducible over $\bR$. To understand the situation in general, and to put the real and complex cases in the proper context, we consider the endomorphism algebra of Hodge structures of CY type. Using this, we obtain a decomposition of $V$ into absolutely irreducible pieces, one of which is  of Calabi-Yau type (and thus classified in the previous section). However, to give a complete classification over $\bQ$ of the representations $V$ that might occur seems difficult and beyond the scope of this paper. Thus, we restrict ourselves to some general remarks and a partial classification in the weight $3$ case (Theorem \ref{classify3q}). The weight $2$ ($K3$ type) case (reviewed in \S\ref{k3mult}) is well understood by work of Zarhin \cite{zarhin}.

\subsection{The endomorphism algebra and the decomposition of Hodge structures of CY type}\label{endoalgssect} We recall that for a $\bQ$-Hodge structure $(V,h)$, the endomorphism algebra is defined as
$$E:=\End_{\Hg}(V)=\left\{f\colon V_{\bQ}\to V_{\bQ}\mid f \textrm{ is a $\bQ$-linear map s.t. } f_\bC(V^{p,q})\subseteq V^{p,q} \right\}.$$
A $\bQ$-linear map $f\colon V_{\bQ}\to V_{\bQ}$ preserves the Hodge filtration $\iff$ $f$ commutes with the action of the Hodge group $\Hg(V)$ of $V$ $\iff$ $f$ commutes with the action of the Mumford-Tate group $\MT(V)$ of $V$, 
justifying the notation $\End_{\Hg}(V)$. Under the assumption that $V$ is a simple Hodge structure (i.e.\ it contains no non-trivial sub-Hodge structures),  $E$ is a finite dimensional division algebra over $\bQ$. If $V$ is of CY type, by an argument essentially due to Zarhin \cite[Theorem  1.6(a), Theorem  1.5]{zarhin}, we obtain:
\begin{proposition}\label{thmendo}
Let $(V,h)$ be a simple $\bQ$-Hodge structure of CY type. Then the endomorphism algebra $E$ is a number field. If additionally $(V,h)$ is polarizable, then either  $E$ is a totally real number field (the real case) or  $E$ is a CM field, i.e.\ a purely imaginary extension of a totally real number field (the complex case). 
\end{proposition}
\begin{proof}
Since $V^{n,0}$ is preserved by a Hodge endomorphism, we get a non-trivial morphism of algebras $\epsilon\colon\End_{\Hg}(V) \to  \End_\bC(V^{n,0})$ via:
\begin{equation}\label{eqembed}
f\mapsto (f_\bC)_{\mid V^{n,0}}.
\end{equation} 
Since $V^{n,0}$ is $1$-dimensional and $E$ is a division algebra, we conclude $E$ is a subfield of $\bC\cong \End_\bC(V^{n,0})$ 
 and thus $E$ is a number field. In the polarized case, the polarization defines an involution (\textsl{the Rosati involution}) on the endomorphism algebra $E$ (e.g.\ \cite[(1.20)]{moonenmt2}). The proposition follows by restricting Albert's classification of involutive division algebras (e.g.\ \cite[(1.19)]{moonenmt2}) to the case of number fields. 
\end{proof}

We also note the following (cf. \cite[(1.24)]{moonenmt2}):
\begin{lemma}\label{lemrealtype}
With notation as above, if $E$ is a totally real field, $\Hg$ is semi-simple.  \qed
\end{lemma}

\begin{notation}\label{assumptiononv}
In what follows, we assume $(V,h,Q)$ is a simple weight $n$ polarized Hodge structure of CY type. Let $E$ be its field of endomorphisms. We denote by $\bar \ $ the complex conjugation on $E$ and by $E_0=\{x\in E\mid \bar x=x\}$ the totally real subfield of $E$. Let $d=[E_0:\bQ]$ and let $\{\sigma_1,\dots,\sigma_{d}\}$ be the embeddings of $E_0$ in $\bR$. Note that $E$ is endowed with a preferred embedding in $\bC$, $\epsilon\colon  E\hookrightarrow \bC$, given by \eqref{eqembed}. The restriction $\epsilon_{\mid E_0}$ gives a preferred embedding of $E_0$ into $\bR$, which we denote also by $\epsilon$ and we set $\sigma_1=\epsilon$. 
\end{notation}

Note that $V$ is naturally an $E_0$-vector space (via $(f,v)\in \End(V)\times V\mapsto f(v)\in V$) with $\dim_\bQ V=d\dim_{E_0} V$. Then, we have an eigenspace decomposition
\begin{equation}\label{decompositionr}
V_{\bR}=\bigoplus_{i=1}^d V_i,
\end{equation}
where $V_i:=\{ v\in V_\bR\mid f(v)=\sigma_i(f)\cdot v \textrm{ for all } f\in F\}$, or equivalently $V_i= V\otimes_{E_0,\sigma_i}\bR$. By construction, $V_i$ are $\bR$-Hodge structures (defined over $E_0$) and $V_1$ is  a weight $n$ Hodge structure of CY type. The other components $V_i$ have smaller Hodge level.
Similarly, if $E_0\neq E$ (the complex case), we have a refinement of \eqref{decompositionr}:
\begin{equation}\label{decompositionc}
V_{\bC}=\bigoplus_{i=1}^d (V'_i\oplus V_i''),
\end{equation}
with $V_i'\oplus V_i''=V_i\otimes_\bR\bC$ and $\overline{V_i'}=V_i''$. Note that $V_i$ (and $V_i'$) are representations of the Hodge group (and its derived subgroup).  We are interested in the geometric irreducibility of these representations. Specifically, we get: 
 
\begin{proposition}\label{propirred}
Let $V$ be a simple $\bQ$-Hodge structure of CY type. Let $G$ be derived subgroup of  $\Hg(V)$, and $E$ the endomorphism algebra. Consider the decompositions \eqref{decompositionr} and \eqref{decompositionc} as above. The following hold:
\begin{itemize}
\item[i)] In the real case ($E_0= E$),  $V_i$ are absolutely irreducible $G_\bR$-representa\-tions, i.e. $V_{i,\bC}$ is an irreducible $G_{\bC}$-representation;
\item[ii)] In the complex case ($E_0\subsetneqq E$),  $V'_i$ and $V''_i$ are irreducible $G_\bC$-representations. 
\end{itemize}
\end{proposition}
\begin{proof}
We make two basic remarks. First, let $k$ be a subfield of $\bC$,  
let $H$ be an algebraic group defined over $k$, and let $W$ an $H$-representation defined over $k$. Using the fact that $H(k)$  is dense in $H(\bC)$, it follows easily that 
$$\End_{H_F}(W_F)=\End_H(W)\otimes_k F$$
for any extension $k\subseteq F\subseteq \bC$.  
Secondly, $\dim_F \End_{H_F}(W_F)$ is at least equal to the number of irreducible $H_F$ factors (corresponding to scaling on each factor). 

Returning to our situation, in the real case ($E_0=E$), $\Hg$ is semi-simple, hence $\Hg = G$.  Moreover, $E_0\otimes_\bQ\bR \cong \R^d$ and
$$\dim_\bR \End_{\Hg_\bR}(V_\bR)=\dim_\bR E_0\otimes_\bQ\bR=d.$$
On the other hand, from \eqref{decompositionr}, $\dim_\bR \End_{\Hg_\bR}(V_\bR)\ge d$. Since the equality holds, we conclude that $V_i$ are irreducible representations with $\End_{\Hg_\bR}(V_i)=\bR$, showing that the $V_i$ are of real type.  In the complex case, $\Hg_\bC$ is an almost direct product of $G_\bC$ and a torus $(\bC^*)^k$. Since the representations of the torus are $1$-dimensional (corresponding to twisting by characters as in \eqref{representationtwist} below), the claim then follows by an argument  similar to that in the real case.
\end{proof}

In particular, the situation studied in Section \ref{sectclassify} is described as follows: 
\begin{corollary}\label{corcm}
Let $(V,h)$ be a polarizable $\bQ$-Hodge structure of CY type.  Let $\Hg$ be the Mumford--Tate group and $G=\Hg_{\textrm{der}}$ the derived subgroup (assumed to be non-trivial). Assume \ref{conventionired} holds (i.e.\ $V_\bR$ is irreducible as a $G_\bR$-representation), then either
\begin{itemize}
\item[i)] the real case: $E=\bQ$ and $V_\bR$ is a $G_\bR$-representation of real type, or
\item[ii)] the complex case: $E=\bQ[\sqrt{-d}]$ for some positive square-free integer $d$  and $V_\bC$ is a $G_\bR$-representation of complex type.
\end{itemize}
Conversely, let $V$, $G$, and $E$ be as above. Assume that $V$ is an irreducible $G$-representation and that $E$ is either $\bQ$ or $\bQ[\sqrt{-d}]$. Then $V_\bR$ is an irreducible $G_\bR$-representation. \qed
\end{corollary}

\begin{definition}\label{weakCMdef}
We recall that a Hodge structure $(V,h)$ is said to have \textsl{weak real multiplication}, respectively  \textsl{weak complex multiplication (CM)}, by $F$ if there is an inclusion $F\hookrightarrow E=\End_{\Hg}(V)$ for a totally real field, respectively CM field $F$ (see \cite[Definition  V.B.1(i)]{mtbook}). 
\end{definition} 

If $V$ has weak CM, then one can apply the half-twist construction of van Geemen \cite{halftwist0} (see \cite[\S V.B]{mtbook} for a fuller discussion). The previous corollary gives a more intrinsic description of the distinction between the real and complex case and explains the occurrence of half-twists in Theorem \ref{thmclassify}.

\subsection{Some remarks on the Hodge group and the structure of $V$}
Now suppose that $\rho\colon G\to \GL(V)$ is a Hodge representation defined over $\bQ$ as in Section \ref{sectclassify} (i.e. $\rho$ gives a Hermitian VHS of CY type).  However,  we do not assume that Convention \ref{conventionired} holds. Thus we assume that $V$ is  irreducible  over $\bQ$, but not  necessarily over $\R$. Without loss of generality, we assume $G$ is simply connected and almost simple.  Following \cite[\S10, p. 54]{milne}, there exists a geometrically almost simple group $G'$ defined over a totally real number field $F$ such that $G=\Res_{F/\bQ} G'$.
Thus,
\begin{equation}\label{decomposeg}
G_\bR=\prod_{j\in \calS_F(\bR)} G_j,
\end{equation}
where $\calS_F(\bR)=\Hom(F,\bR)$ is the set of real embedding of $F$ and $G_j=G'\otimes_{F,j}\bR$ are the twisted forms of $G'$. Since $\Gal(\bar \bQ/\bQ)$ permutes the factors $G_j$, all of the $G_{j,\bC}$ will have the same type (i.e. the same irreducible root system $R$). While some of the $G_j$ may be compact,  there is at least one non-compact factor. For each non-compact factor $G_j$, we have a cocharacter $\varphi_j\colon U(1)\to G_{j,\bR}$, which (due to the Hermitian assumption) corresponds to a special root of $R$ as in \S\ref{sectprelim}. For a compact factor $G_j$, the corresponding cocharacter $\varphi_j$ is trivial.
Note also that $G_j$ can be chosen arbitrarily, i.e. for a prescribed choice of real forms $G_j$ there exists a $\bQ$-form $G$ of type $G=\Res_{F/\bQ}G'$ with associated almost simple real factors $G_j$ (see \cite[Proposition 10.10]{milne}). Since $\rho(G)$ is the derived subgroup of the generic Hodge or Mumford-Tate group of the corresponding VHS $\calD$ of Hermitian type, the results of \S\ref{endoalgssect} apply to this situation, where $E$ denotes the endomorphism algebra of a Hodge structure at a general point of $\calD$.

\subsubsection{}Recall that the representation $\rho\colon G\to \GL(V)$ factors through the Hodge group $\Hg$ and that $\rho(G)=\Hg_{der}$. Thus, at the level of Lie algebras we have 
\begin{equation}\label{splithg}
\hg_\bR=\fa\oplus \fg_1\oplus \dots\oplus \fg_q,
\end{equation}
where $\fa$ is an abelian Lie algebra, and $\fg_j$ are the Lie algebras of $G_j$ (compare \eqref{liehodge}). An irreducible factor $W$ of the representation $V_\bC$ will be then of type 
\begin{equation}\label{representationtwist}
\chi_0\otimes W_1\otimes\dots\otimes W_q,
\end{equation}
where $W_j$ are irreducible $G_{j,\bC}$-representations and $\chi_0$ corresponds to twisting the weights of the Hodge structure on $W_1 \otimes \cdots \otimes W_q$ by a character. By Proposition \ref{propirred}, the irreducible components $W$ of $V_\bC$ are $V_{i,\bC}$ in the real case or $V'_{i}$ and $V_i''$ in the complex case. Thus, 
$$\Hg_\bR\subset \GL(V_1)\times \dots \times \GL(V_d),$$
and we get 
irreducible representations (defined over $E_0$)
$$\rho^i\colon G_\bR\to \GL(V_i).$$
Writing $G_\R$ as a product of  geometrically simple factors, we obtain absolutely irreducible representations
\begin{equation}\label{eqrhoij}
 \rho_j^i\colon G_j\to \GL(W_j^i)
 \end{equation}
 where $V_i=\bigotimes_j W_j^i$ and the $\rho^i=\bigotimes_j \rho_j^i\colon G\to \GL(V_i)$ are  the projections of $\rho\colon G\to \GL(V)$ on the $i^{\text{th}}$ factors of the decomposition \eqref{decompositionr}.  In the real case, the Galois group  $\Gal(\bar \bQ/\bQ)$ then permutes the  factors $\rho_j^i$, and a similar statement holds in the complex case. 
  
 \begin{remark}
 Note that by composing $ \rho_j^i$ with the cocharacter $\varphi_j$ associated to $G_j$, we get Hodge structures on  each $W_j^i$.  In particular, the compact factors $G_j$ give trivial Hodge structure of type $(0,0)$ (up to normalizing the weight) on $W_j^i$. 
 \end{remark}

\subsubsection{} As before we can view $V$ as an $E_0$-vector space. As $\bQ$-vector space we have $V=\Res_{E_0/\bQ} V$. Furthermore, in the polarized case, there is a symmetric or symplectic $E_0$-linear form $\tilde Q$ on $V$ such that the trace of $\tilde Q$ is the original polarization $Q$ on $V$.  If $E_0\subsetneq E$, then $\tilde Q$ can be taken to be a $E_0$- Hermitian form with respect to the extension $E_0\subset E$. It follows immediately that 
\begin{equation*}
\Hg\subseteq \left\{\begin{matrix}\Res_{E_0/\bQ}\Sp_{E_0}(V,\tilde Q)&\textrm{if  } E \textrm{ is totally real and odd weight}\\
\Res_{E_0/\bQ}\SO_{E_0}(V,\tilde Q)&\textrm{if  } E \textrm{ is totally real and even weight}\\\
\Res_{E_0/\bQ}\U_{E}(V,\tilde Q)& \textrm{if } E\textrm{ is a CM field }
\end{matrix}\right.
\end{equation*}
(compare \cite[(1.22)]{moonenmt2}). 

 A lower bound on the Hodge group can be obtained by considering the subgroup
 $$H'_{\bR}=(\Hg_\bR\cap \GL(V_1))\times\dots\times  (\Hg_\bR\cap \GL(V_d))\subseteq \Hg_\bR.$$
Since $\Hg$ and the representation $V$ are defined over $\bQ$, $H'$ is in fact an algebraic group over $\bQ$ and of type $\Res_{E_0/\bQ}(H)$ for some algebraic group $H$ over $E_0$ (compare \cite[(3.17)]{moonenmt2}). We conclude:
\begin{equation}\label{boundhodge}
\Res_{E_0/\bQ}(H)\subseteq\Hg\subseteq \left\{\begin{matrix}\Res_{E_0/\bQ}\Sp_{E_0}(V,\tilde Q)\\
\Res_{E_0/\bQ}\SO_{E_0}(V,\tilde Q)\\
\Res_{E_0/\bQ}\U_{E}(V,\tilde Q).
\end{matrix}\right.
\end{equation}
In general, both inclusions  can be strict.

\subsection{The case of $K3$ type (cf.\ \cite{zarhin}, \cite{vgk3})}\label{k3mult} For VHS of $K3$ type, both inclusions in \eqref{boundhodge} are equalities. First note that in the decomposition \eqref{decompositionr}, $V_1$ is a Hodge structure of $K3$ type and $V_i$ for $i\neq 1$ are Hodge structures of type $(1,1)$. Thus, the Hodge structure of $V$ is completely determined by that on $V_1$. It follows  $\Hg=\Res_{E_0/\bQ} H_1$, where $H_1=\Hg_{E_0}\cap \GL(V_1)$ (compare \cite[(3.17)]{moonenmt2}). The other equality follows from the classification of the Hermitian VHS of $K3$ type (given by $H_1\to \GL(V_1)$) as in Section \ref{sectclassify} (N.B. since the period domain $\bD$ of $K3$ type is already Hermitian symmetric, $H_1$ is automatically of Hermitian type). In conclusion,  the following holds in the real case (see \cite[Theorem  2.8]{vgk3}, \cite[Theorem  2.2.1]{zarhin}; see \cite[Theorem  2.3.1]{zarhin} for the complex case):

\begin{theorem}[Zarhin]
Let $(V,h,Q)$ be a $K3$-type Hodge structure with $E=\End_{\Hg}(V)=E_0$ a totally real field. Then, the Hodge group $\Hg$ satisfies:
$$\Hg=\Res_{E_0/\bQ}\SO(V,\tilde Q), \ \Hg(\bR)\cong \SO(2,m-2)\times \SO(m,\bR)^{d-1},$$
and then $\Hg(\bC)\cong \SO(m,\bC)^d$, where $\tilde Q$ is the unique symmetric $E_0$-bilinear form on $V$ such that $\mathrm{Nm}(\tilde Q)=Q$. The representations of these Lie groups on the $d\cdot m$-dimensional vector spaces $V_\bR$ and $V_\bC$ are the direct sum of the standard representations of the factors (N.B. in the notation of \eqref{eqrhoij}, only the representation $\rho_i^i$ are non-trivial, and they are permuted transitively by $\Gal(\bar \bQ/\bQ)$). 
\end{theorem}

\begin{remark}
It is not difficult to construct a Hermitian VHS of $K3$ type with the prescribed weak real multiplication by any totally real field $E_0$. Essentially, after fixing an embedding $\epsilon\colon E_0\hookrightarrow \bR$, one considers  an $E_0$-vector space $V$ of dimension $m$ with an $E_0$-symmetric bilinear form $\tilde Q$ of signature $(2,m-2)$ and proceeds as usual. The result will be a VHS of $K3$ type over the   Type $\IV$ domain  $$\calD=\left\{\omega\in \bP(V\otimes_{E_0,\epsilon}\bC)\mid \tilde Q(\omega,\omega)=0,\ \tilde Q(\omega,\bar\omega) >0\right\}.$$
The only difference to the usual situation  is that the Hodge structures involved are defined over $E_0$. By considering the action of  $\Res_{E_0/\bQ}\SO(V,\tilde Q)$ on $\Res_{E_0/\bQ}V$,  one obtains a VHS defined over $\bQ$. In order for this VHS to be of $K3$ type, one needs the  forms  on $V_\R$ defined by $\tilde Q$, but using the embeddings $\sigma \colon E_0\to \R$ with $\sigma \neq \epsilon$,  to be negative definite. For example, this can be achieved by using a quadratic form of type:
$$\tilde Q(x_1,\dots,x_n)=a_1x_1^2+a_2 x_2^2-x_3^2-\dots-x_{m}^2,$$
where  $a_1,a_2\in E_0$ such that $\epsilon(a_i)>0$ and $\sigma(a_i)<0$ for all other $\sigma\in\calS_{E_0}(\bR)$ (see \cite{vgk3}, esp. \cite[Lem. 3.2]{vgk3}, for further details).
\end{remark}

\subsection{The case of weight three CY type}\label{cyexampleq}
We close by making some remarks on the Calabi-Yau $3$-fold case. We restrict the discussion to the real case: $E=E_0$ is  a totally real field. The case $E=\bQ[\sqrt{-d}]$ was already discussed in Section \ref{sectclassify}.  The general complex case ($E$ is an arbitrary CM field) is a mixture of the real case and the CM type case, i.e. in addition to the study of $G_\bR$ representations $V_i$ one needs to keep track also of the shifts of the weights given by the (connected) center of the Hodge group. This leads to subtle Galois theoretic issues that go beyond the scope of this paper. A general reference is the discussion of 
 the CM type case (i.e. $E$ is CM with $[E:\bQ]=\dim_\bQ V$) in \cite[Chapter V]{mtbook} (esp. \cite[Proposition V.C.2]{mtbook}, which applies in the Calabi-Yau case).

It is easy to produce a Hermitian VHS of CY type with weak real multiplication by considering $V=W_1\otimes W_2$, where $W_1$ is a VHS of elliptic type and $W_2$ is a  VHS of $K3$ type with  real multiplication as in \S\ref{k3mult}.  Therefore, we focus on the primitive cases given by  Corollary \ref{thmclassify3} (i) and give the following partial classification.

\begin{theorem}\label{classify3q}
Let $G$ be an almost simple $\bQ$-algebraic group, and $\rho\colon G\to \GL(V)$ an irreducible Hodge representation giving a Hermitian VHS $\calV$ of Calabi-Yau $3$-fold type. Assume that the generic endomorphism algebra $\End_{\Hg}(V)$ is a totally real field $E_0\neq \bQ$. Then
\begin{itemize}
\item[(i)] As an $\bR$-VHS $\calV$ decomposes as a direct sum $\bigoplus_{i=1}^d\calV_i$ (with $d=[E_0:\bQ]$) such that $\calV_1$ is an irreducible Hermitian VHS of Calabi-Yau $3$-fold type (classified by Cor. \ref{thmclassify3}) and $\calV_i$ for $i>1$ are Tate twists of VHS of abelian variety type (i.e.\ weight one). 
\item[(ii)] Assume additionally that $\calV_1$ is primitive in the sense of \ref{thmclassify3}(i) (i.e. the associated domain $\calD_1$ is irreducible, of tube type, and rank $3$). Then the only two possibilities are 
\begin{itemize}
\item[a)] $\calV_1$ is of type $(A_5,\alpha_3;\varpi_3)$ and $\calV_i$ (for $i>1$) are of type $(A_5,\alpha_1;\varpi_3)$;
\item[b)] $\calV_1$ is of type $(D_6,\alpha_6;\varpi_6)$ and  $\calV_i$ (for $i>1$) are of type $(D_6,\alpha_1;\varpi_6)$.
\end{itemize}
Furthermore, both of these cases can be realized as $\bQ$-VHS $\calV$ as above.   
\end{itemize}
\end{theorem}
\begin{proof}
As discussed (Lemma \ref{lemrealtype}), in the real case we can focus on $V$ as a $G_\bR$-representation. We get a decomposition into absolutely  irreducible representations $V_\bR=V_1\oplus \dots \oplus V_d$ (Proposition \ref{propirred}) and a decomposition into simple geometric factors $G_\bR=G_1\times\dots\times G_q$ (cf. \eqref{decomposeg}) such that all  $G_j$ have the same type (over $\bC$). We make the following two easy remarks:
\begin{itemize}
\item[(1)] If $G_j$ occurs non-trivially in the representation $V_1$, then $G_j$ is of non-compact type and $\rho_j^1\colon G_j\to \GL(W_j^1)$ is of CY type and real, and thus classified in Section \ref{sectclassify}.  
\item[(2)] For all $i>1$, precisely one non-compact $G_j$ occurs non-trivially in $V_i$, since otherwise the Hodge level of $V_i$ will be larger than $1$  contradicting the weight $3$ Calabi-Yau assumption on $V$. The associated representations $\rho_j^i\colon G_j\to \GL(W_j^i)$ are of symplectic type (for $i>1$) as classified by Deligne \cite{dshimura}.
\end{itemize}
From (1) and the primitivity assumption, we get that only one group occurs nontrivially in the representation $V_1$ (i.e. only one $\rho_j^1$ is non-trivial). Since over $\bC$ the representations $V_i$ are isomorphic, it follows then that only one group $G_{j_i}$ occurs non-trivially in each $V_i$ (i.e. $\rho_{j_i}^i$ is nontrivial). From this and the two remarks above, it follows that each $G_j$ is of non-compact type. Furthermore, from (2) and the classification of symplectic representations, we conclude that $G_j$ has to be of classical type (e.g. \cite[\S10.8]{milne}); this excludes the $E_7$ case.  

Returning to (1),  the assumption in (ii)  gives the following three possibilities for the CY piece $V_1$ (with notations as in \ref{thmclassify3}):
 $(A_5,\alpha_3; \varpi_3)$, $(D_6,\alpha_6; \varpi_6)$,  $(C_3,\alpha_3; \varpi_3)$. The symplectic representations associated to these groups are: $(A_5,\alpha_3; \varpi_1)$  and its dual $(A_5,\alpha_3; \varpi_5)$, $(A_5,\alpha_1;\varpi_i)$, $(D_6,\alpha_6; \varpi_1)$, $(D_6,\alpha_1;\varpi_6)$, $(D_6,\alpha_1;\varpi_5)$, $(C_3,\alpha_3;\varpi_1)$ (e.g. \cite[\S10]{milne}).
For dimension reasons ($\dim V_i=\dim V_1$), the only two possibilities are 
\begin{itemize}
\item[a)] $V_1$ of type $(A_5,\alpha_3;\varpi_3)$, $V_i$ of type $(A_5, \alpha_1;\varpi_3)$ for $i>1$;
\item[b)] $V_1$ of type $(D_6,\alpha_6;\varpi_6)$, $V_i$ of type $(D_6,\alpha_1;\varpi_6)$
for $i>1$.
\end{itemize}

To conclude the proof, we only need to show that these two cases are realized by some Hodge representation $\rho\colon  G\to \GL(V)$ (over $\bQ$) as in the statement of the theorem. Both cases can be realized by starting with an appropriate  representation $\rho_1\colon G_1\to \GL(V_1)$ defined over a given totally real field $E_0$ and taking $G=\Res_{E_0/\bQ} G_1$ together with the induced representation on $V=\Res_{E_0/\bQ} V_1$. Then $G_\R = \prod_{i=1}^dG_{i, \R}$,   where $d=[E_0:\bQ]$, the distinct embeddings of $E_0$ in $\R$ are denoted by $\sigma_1, \dots, \sigma_d$,   the group $G_{i,\R}$ is the group over $\R$ induced by the embedding $\sigma_i$, and similarly $V\otimes_{\Q}\R = \bigoplus _{i=1}^d V_{i, \R}$.
The associated $\bQ$-VHS $\calV$ will have real multiplication by $E_0$ and will split over $\bR$  into a CY piece $\calV_1$ and a product of Tate twists of weight $1$ factors, say  $\calV_2, \dots, \calV_d$. 
The Mumford-Tate domain is $\calD_1\times \calD_2\times \cdots \times \calD_d$, a product of  different Hermitian symmetric domains, with $\calD_i$ parameterizing the Hermitian VHS $\calV_i$.
We will discuss here only the $A_5$ case. The  case of $D_6$ involves the detailed study of the interplay between the Clifford algebra and a suitable Hermitian form and it is discussed in \cite{FL2}.

In the $A_5$ case, we want $G_{1, \R}=\SU(3,3)$ and $G_{i, \R}=\SU(1,5)$ for $i> 1$. Over $\bC$, $G_{i,\bC}\cong \SL(6,\bC)$ for all $i$, and  $V_{i,\bC}\cong \bigwedge_{\bC}^3 U$ (corresponding to weight $\varpi_3$), where $U$ is the standard representation of $\SL(6,\bC)$. (Here we write $\bigwedge_{\bC}^3$ to emphasize that we take the wedge product of the \textbf{complex} vector space $U$.) A real representation  of $G_{1, \R}$, resp.\  $G_{i, \R}$ for $i>1$, is obtained by endowing $U$ with a Hermitian form  of signature $(3,3)$ resp.\  $(1,5)$. From  \cite[Thm. IV.E. 4]{mtbook}, it follows that the representation $\bigwedge_{\bC}^3 U$ for $G_{i, \R}$ is of real type, i.e. there exist real representations $V_{i, \R}$ of $G_{i, \R}$ with $\bigwedge_{\bC}^3 U\cong V_{i,\R}\otimes_\bR \bC$. 
Thus, we need to construct $G_1$ and the representation 
$$\rho_1\colon G_1\to \GL(V_1)$$
over the totally real field $E_0$, so that the induced representations $\rho_i\colon G_{i,\R} \to V_{i,\R}$ are  as claimed. Then  $\rho=\Res_{E_0/\bQ}(\rho_1)$ will be an irreducible representation over $\bQ$ with the desired properties. 

We fix an embedding $E_0\subset \bR$, and chose a quadratic imaginary extension $E_0\subset E$. By the approximation theorem, there exists a $\delta \in E_0$ such that $\sigma_1(\delta) > 0$ and $\sigma_i(\delta) < 0$ for $i> 1$. For example, in case $E_0 = \Q[\sqrt{d}]$ for a positive square-free integer $d$, we can simply take $\delta = \sqrt{d}$.
 Let $U$ be an $E$-vector space of dimension $6$, together with a Hermitian form $\psi_1$, where $\psi_1$ is given in a suitable basis by
$$\psi_1(z,w)=z_1\bar w_1+\delta\cdot z_2\bar w_2+\delta\cdot z_3\bar w_3-z_4\bar w_4-z_5\bar w_5-z_6\bar w_6.$$
Let   $\psi_i$ be the   form obtained by replacing  $\delta$ by $\sigma_i(\delta)$ above. 
Define $G_i=\SU(U,\psi_i)$,  and let $U_i=(U,\psi_i)$ (i.e.\ the standard representations of $G_i$). Clearly, $G_{1,\bR}\cong \SU(3,3)$ and $G_{i,\bR}\cong \SU(1,5)$ for $i > 1$. (Since the embeddings $\sigma_i \colon E_0\to \bR$ are given,  the groups $G_{i,\bR}$ are well defined.)    The representation  $\bigwedge^3_E U_1$ of $G_1$ is \emph{a priori} defined over $E$, where again we write $\bigwedge_{E}^3$ to emphasize that we take the wedge product of    $U_1$ viewed as  $E$-vector spaces.  By Corollary~\ref{ratcor} below, since $(-1)^3 \disc(\psi_1)=\delta^2$ is a square in $E_0$, and hence of the form $\Nm_{E/E_0}(c)$, where $\Nm_{E/E_0}(c)\colon E^* \to E_0^*$ is the norm,   $\bigwedge^3_E U_1$ is in fact defined over $E_0$: $\bigwedge^3_E U_1\cong V_1\otimes_{E_0} E$ for some $G_1$-representation $V_1$ defined over $E_0$. 
Clearly,  the different embeddings $\sigma_i$ replace the $\Cee$-vector space  $U_{1,\bR}$ by $U_{i,\bR}$ and the real vector space  $V_{1,\bR}$ by $V_{i,\bR}$. Hence the factors $G_{i,\bR}$ and the representations $\rho_i$  are as claimed, concluding the proof of Theorem~\ref{classify3q}.
\end{proof}

\begin{remark} By the classification of Hermitian forms due to Landherr (see for example \cite{shimurah}), if $\psi'$ is an $(E,E_0)$-Hermitian form of rank $6$ such that the signature of $\sigma_i(\psi')$ is $(3,3)$ for $i=1$ and $(1,5)$ for $i> 1$ and such that $(-1)^3 \disc(\psi')$ is a norm, then $\psi'$ is equivalent to the form $\psi_1$ defined above. Thus in some sense the above construction is the most general one possible starting with a real quadratic field $E_0$ and a Hermitian form defined over an imaginary quadratic extension of $E_0$.
\end{remark}

\subsection{Appendix to Section \ref{cmsect}}\label{hodgestar}
For $p+q=2n$, the representation $\bigwedge^n_\Cee U$ of the standard representation  $U$ of $\SU(p,q)$  has an extra endomorphism given by a variant of the Hodge star  operator.  Over $\bR$, this implies the fact that the representation of highest weight $\varpi_n$ of $\SU(p,q)$ is of real type if $p-q\equiv 0 \mod 4$ and of quaternionic type otherwise (compare \cite[Thm. IV.E.4]{mtbook}). For the proof of Theorem \ref{classify3q}, we need an analogous statement over number fields. For lack of a suitable reference, we sketch the details of the construction.  (See also \cite{lombardo} for a partial discussion in the number field case  and \cite[\S3.9]{gs} for the real case.)

Fix a subfield $E_0$ of $\bR$ (in practice,  either $E_0= \bR$ or $E_0$ is a totally real number field with a fixed embedding into $\R$), and let $E$ be an imaginary quadratic extension of $E_0$, with $\bar\ $ the complex conjugation ($E_0=\{z\in E\mid \bar z=z\}$). Consider an $E$-vector space $U$ of dimension $2n$ together a non-degenerate Hermitian form $\psi$. Let $G=\SU(U,\psi)$. Let $W$ be the $E$-vector space $\bigwedge_E^n U$. On $W$ we have a  pairing 
$$\wedge\colon  W\times W\to \bigwedge_E^{2n} U\cong E$$
given by the wedge product on $n$-forms. Fixing a generator $\omega \in \bigwedge_E^{2n} U$ (i.e.\ a ``volume form,''  or equivalently an identification $\bigwedge_E^{2n} U\cong E$) gives an $E$-linear identification:
$$\rho\colon W\to W^\vee.$$
Additionally, on $W$ we have an induced Hermitian form $\tilde \psi=\wedge^n\psi$, which can be viewed as an $E$-anti-linear identification 
$$\tau\colon W\to W^\vee.$$
One can check that $\tilde \psi$ is characterized by the condition that
$$\tilde \psi(v_1\wedge \cdots \wedge v_n,w_1\wedge \cdots \wedge w_n) = \det(\psi(v_i, w_j)).$$ 
The $E$-anti-linear (and hence $E_0$-linear) {\it Hodge star operator} is defined by
\begin{equation}\label{eqhodgestar}
\star= \rho^{-1}\circ \tau\colon W\to W.
\end{equation}
Equivalently, for all  $w_1,w_2\in W=\bigwedge_E^n U$, 
$$\tilde \psi(w_1,w_2)\cdot \omega=w_1\wedge \star w_2.$$
If we replace $\omega$ by $c\omega$ for some $c\in E$, we replace $\star$ by $c\star$ and hence we replace  $\star\star\colon W\to W$  by $\Nm_{E/E_0}(c )\star\star$.
For a non-degenerate Hermitian form $\psi$, one defines $\disc(\psi)\in E_0^*$ to be $\det M$, where $M$ is the Hermitian matrix (with entries in $E$) representing $\psi$ with respect to a  choice of basis on $U$. Since a change of basis transforms $M$ into $A\cdot M\cdot \bar{A}^t$,  the discriminant is intrinsically defined only as a class in $E_0^*/\Nm_{E/E_0}(E^*)$. For example,  for $E_0=\bR$ and signature $(p,q)$, the discriminant is $(-1)^q\in \bR^*/(\bR^*)^2\cong\{-1,1\}$. However, once the volume form $\omega$ has been fixed, by considering only bases $e_1, \dots, e_{2n}$ with  $e_1\wedge  \cdots \wedge e_{2n} = \omega$, we get a well defined $\disc(\psi)\in E_0$.   

\begin{lemma}\label{lemhodgestar}
The Hodge star operator $\star\colon W\to W$ commutes with the natural $G$-action on $W$ and satisfies:
$$\star \star =(-1)^n \disc(\psi)\cdot \id_W.$$
\end{lemma}
\begin{proof} The action of $G=\SU(U,\psi)$  preserves the volume form $\omega$, the pairing $W\times W\to \bigwedge_E^{2n} U$, and the Hermitian form $\tilde \psi=\wedge^n \psi$ and hence commutes with $\star$. 

We can check the formula for $\star \star$ by choosing a basis $\{e_1,\dots,e_{2n}\}$ for $U$ that diagonalizes the Hermitian form $\psi$, so that the matrix $M$ of $\psi$ with respect to this basis is $M=\mathrm {diag}(a_1,\dots,a_{2n})$ (with $a_i\in E_0$). Choose  $\omega=e_1\wedge\dots\wedge e_{2n}$. A basis for $W$ is given by $\{e_{I}:=e_{i_1}\wedge \dots \wedge e_{i_n}: |I|=n\}$. Clearly
$\star(e_{I})=\epsilon_{I,I'} \cdot a_I\cdot e_{I'}$,
where $I'=\{1,\dots, 2n\}- I$ is the complementary index, $a_I=a_{i_1}\dots a_{i_n}$, and $\epsilon_{I,I'}\in \{\pm 1\}$ is the signature of the permutation $(I,I')$. Thus, 
\begin{eqnarray*}
\star\star(e_{I})&=&(\epsilon_{I,I'}\cdot \epsilon_{I',I}) \cdot (a_I\cdot a_{I'})\cdot e_{I},\\
&=& (-1)^{n\cdot n} \disc(\psi) \cdot e_I,
\end{eqnarray*}
as claimed.
\end{proof}

The standard representation $U$ is only defined over $E$. Thus, the representation $W=\bigwedge^n_EU$ is also \emph{a priori} only defined over $E$. We can view it a representation over $E_0$, by considering $\Res_{E/E_0} W$,   i.e.\ by considering $W$ as a $E_0$-vector space. Note that 
$$\Res_{E/E_0} W \otimes _{E_0}E = W\otimes _{E_0}E \cong W \oplus \bar{W} \cong W \oplus W\spcheck \cong W\oplus W$$
as $E[G]$-modules. Clearly, $E\subseteq \End_{E_0[G]}(\Res_{E/E_0} W)$  (induced by the $E$-vector space structure on $W$). The above discussion says that $\star$ gives an extra endomorphism for the $G$-representation $\Res_{E/E_0} W$. Let $\mathcal{A}$ be the $E_0$-subalgebra of $\End_{E_0[G]} (\Res_{E/E_0} W)$ generated by $E$ and $*$.  Note that, for all $\alpha \in E$, $ \star \alpha= \bar{\alpha}\star$, so that $\mathcal{A}$ is a noncommutative $E_0$-algebra which is also an $E$-vector space. Clearly $\dim_E \mathcal{A}=2$ and hence $\dim_{E_0} \mathcal{A}=4$. Changing the volume form $\omega$ simply replaces $\star$ by $c\star$ for some $c\in E$, and hence $\mathcal{A}$ is canonically defined. In case $E = E_0(\sqrt{-e})$, an $E_0$-basis for $\mathcal{A}$ is given by the operators $\Id$, $\mathbf{i} = \sqrt{-e}\Id$, $\mathbf{j} = \star$, and $\mathbf{k} = \mathbf{i}\mathbf{j}$, with $\mathbf{i}^2 =-e\Id$, $\mathbf{j}^2=(-1)^n\disc(\psi)\Id$, and $\mathbf{i}\mathbf{j} = -\mathbf{j}\mathbf{i}$, hence $\mathbf{k}^2 = -\mathbf{i}^2\mathbf{j}^2$.

\begin{proposition}\label{repdecomposition} Let $W=\bigwedge_E^n U$, $G=\SU(U,\psi)$,  and $\mathcal{A}$ be as above.
\begin{enumerate}
\item[(i)] $\mathcal{A}= \End_{E_0[G]} (\Res_{E/E_0} W)$.
\item[(ii)] The algebra $\mathcal{A}$ is semisimple.  
\item[(iii)]  $\mathcal{A}$ is  a division algebra $\iff$ the representation $W$ cannot be defined over $E_0$, i.e.\ $\iff$ $\Res_{E/E_0} W$ is an irreducible $E_0[G]$-module.
\item[(iv)] $\mathcal{A}\cong  \mathbb{M}_2(E_0)$ is  a matrix algebra $\iff$ the representation $W$ can  be defined over $E_0$, i.e.\ $\iff$ $\Res_{E/E_0} W$ is a reducible $E_0[G]$-module, or equivalently if there exists a $G$-representation on an $E_0$-vector space $W_0$ such that $W\cong W_0 \otimes_{E_0}E$ as $E[G]$-modules.
\item[(v)] $\mathcal{A}\cong  \mathbb{M}_2(E_0)$ is  a matrix algebra $\iff$ $(-1)^n\disc(\psi)$ is a norm in $E_0$, i.e\ there exists a $c\in E$ such that $(-1)^n\disc(\psi) = \Nm_{E/E_0}(c)$.
\end{enumerate}
\end{proposition}
\begin{proof} Clearly, there is an inclusion $\mathcal{A}\subseteq  \End_{E_0[G]} (\Res_{E/E_0} W)$ and hence 
$$\mathcal{A}\otimes_{E_0}E\subseteq  \End_{E_0[G]} (\Res_{E/E_0} W)\otimes_{E_0}E \subseteq \End_{E[G]} (\Res_{E/E_0} W\otimes_{E_0}E).$$
Since $\Res_{E/E_0} W\otimes_{E_0}E \cong W\oplus W$ is a direct sum of two irreducible $E[G]$-modules, $\End_{E[G]} (\Res_{E/E_0} W\otimes_{E_0}E) \cong \mathbb{M}_2(E)$ is a matrix algebra of dimension $4$ over $E$. Since $\dim_{E_0}\mathcal{A} =4$, all of the above inclusions are equalities, proving (i), and (ii) follows since $\mathcal{A}\otimes_{E_0}E \cong \mathbb{M}_2(E)$ is semisimple. Clearly, if $\Res_{E/E_0} W$ is an irreducible $E_0[G]$-module, then $\mathcal{A}$ is a division algebra, and if $\Res_{E/E_0} W$ is a reducible $E_0[G]$-module, then it is necessarily of the form $W_0 \oplus W_0$, where $W_0\otimes_{E_0}E \cong W$ as $E[G]$-modules. This implies (iii) and (iv). 
Finally, to see (v), every element $a$ of $\mathcal{A}$ is of the form  $\alpha + \beta\star$,  with $\alpha, \beta\in E$, and $a=0$ $\iff$ $\alpha =\beta = 0$. Now
$$(\alpha + \beta\star)(\bar{\alpha} -  \beta\star) = \alpha\bar{\alpha} - (-1)^n\disc(\psi)\beta\bar{\beta}.$$
This expression is $0$ $\iff$ $\beta \neq 0$ and  $(-1)^n\disc(\psi) = \Nm_{E/E_0}(\alpha/\beta)$ or  $\alpha =\beta =0$. Thus, if $(-1)^n\disc(\psi)$ is not a norm, then $\alpha + \beta\star$ is invertible as long as  it is nonzero.  Conversely, if $(-1)^n\disc(\psi)$ is a norm, then we can construct a zero divisor in $\mathcal{A}$, so that $\mathcal{A}\cong  \mathbb{M}_2(E_0)$ is  a matrix algebra.
\end{proof} 

\begin{corollary}\label{ratcor} The representation $W$ can  be defined over $E_0$ $\iff$ $(-1)^n\disc(\psi)$ is a norm in $E_0$. \qed
\end{corollary}

\section{Explicit realization of horizontal subvarieties: real case}\label{sectequations}
For the remainder of this paper, we shall only be concerned with the weight three CY case. In this and the next section, our goal will be to give a concrete description of a class of   horizontal subvarieties $Y$ of $\check\bD$. In case $Y=\hat{Z}$ is globally a closed subvariety of $\check\bD$,  $Z =\hat{Z}\cap \bD$ will then be semi-algebraic in $\bD$. 

A local description of horizontal subvarieties $Y\subset \check\bD$ is well known (Bryant--Griffiths \cite{bg}) under a mild non-degeneracy condition satisfied in  case $Y=Z$ is the image of the period map for a family of Calabi--Yau threefolds. Namely, if $X$ is a Calabi--Yau threefold then the first order deformations of $X$ are unobstructed, by the Tian--Todorov theorem, and the Gauss--Manin connection induces an isomorphism $H^1(X; T_X) \to \Hom (H^0(X; \Omega^3_X), H^1(X; \Omega^2_X))$. Thus, roughly speaking, the period map is a local embedding of the  moduli space, and the image subvariety $Z\subset \bD$ satisfies: 
\begin{convention}\label{nondegcond} For every choice of local  coordinates $z_1, \dots, z_h$ on $Z$, if $\omega(z)$ is a local section of the Hodge bundle $F^3$, then  the derivatives $\partial \omega/\partial z_i$ span $F^2/F^3$. 
\end{convention}
We will assume in this section that the horizontal subvariety $Y\subset \check \bD$ satisfies this property and that there is a global choice of coordinates adapted to this situation, which we shall describe below. This global description is well-adapted to questions involving real or rational structures (for example, the existence of maximally unipotent monodromy) and is an analogue of  the Cayley transform, which is an unbounded realization of a Hermitian symmetric space.  In the next section, we shall describe the analogue of the Harish-Chandra embedding of a Hermitian symmetric space as a bounded symmetric domain.

\subsection{Local description of the horizontal subvarieties}\label{sectlocal}
 
 \begin{notation} The standard symplectic basis on a lattice $\Lambda$ of rank $2h+2$ will be written as $e_0, e_1, \dots, e_h, f_1, \dots , f_h, f_0$ with $\langle e_i, e_j\rangle = \langle f_i, f_j\rangle =0$ for all $i, j$ and $\langle e_i, f_j\rangle =\delta_{ij}$. In particular, the symplectic lattice $(\Zee^{2h+2}, \langle  \cdot, \cdot \rangle)$ comes with a fixed   filtration $W_\bullet$ defined by
\begin{gather*}
W_0 =W_1 = \Zee \cdot e_0 \subseteq W_2 = W_3 = \operatorname{span}\{e_0,e_1,  \dots, e_h\}\\
\subseteq W_4 = W_5 = \operatorname{span}\{e_0, e_1, \dots, e_h, f_1, \dots, f_h\}\subseteq W_6 = \Zee^{2h+2}.
\end{gather*}
Hence we can speak of a large radius limit, or equivalently of integral symplectic matrices $T$ preserving the filtration (which will in general be maximally unipotent). We will also relax the condition that the $e_i, f_i$ be integral and sometimes just assume that they are a fixed symplectic basis of a $k$-vector space $V$ with a symplectic form defined over $k$, where $k$ is a subfield of $\R$ or $\Cee$.
\end{notation}

Let $\mathbf{D}$ be the classifying space of Hodge structures of weight $3$ on $\Lambda_\Cee =\Lambda \otimes _\Zee\Cee$ satisfying $h^{3,0} = h^{0,3} = 1$, and hence $h^{2,1} = h^{1,2} = h$, and let 
$\check{\mathbf{D}}$ denote the compact dual of $\mathbf{D}$.
We begin with a local description of the horizontal subvarieties of $\check{\mathbf{D}}$, essentially going back to Bryant--Griffiths \cite{bg} (see also \cite{friedmancy}, \cite{voisinmirror}). More precisely, let $(Y,0)$ be the germ of an analytic subspace of $\check{\mathbf{D}}$ which satisfies Griffiths' transversality condition, as well as the non-degeneracy condition \ref{nondegcond}. 
In \S\ref{sectlocal} and \S\ref{globalconst}, the Hodge--Riemann inequalities will be irrelevant. We discuss the Hodge--Riemann inequalities in \S\ref{secthr}. To emphasize the distinction, we will write $Z$ for a locally closed horizontal subvariety (or the germ of such) of $\bD$, and $Y$ for a horizontal subvariety of $\check\bD$ (possibly $Y=\hat Z$, the Zariski closure of $Z$ in $\check \bD$).

First, following \cite{bg}, we note that in the Calabi-Yau threefold case the transversality conditions are closely related to the geometric notion of Legendrian manifold. 
\begin{definition} Let $V$ be a complex vector space endowed with a symplectic form $\langle \cdot, \cdot \rangle$.
A \textsl{Legendrian immersion} from a manifold $X$ to $\Pee V$ is an immersion $\iota \colon X\to \Pee V$ such that, for every $x\in X$, the subspace of $V$ corresponding to  $\iota_*(T_xX)$ is a Lagrangian subspace of $V$.
\end{definition}
By definition a Legendrian immersion induces a filtration of $V$ for each $x\in X$: given $x\in X$, set 
$$F^3 =\bC\cdot \tilde x \subseteq  F^2 =  \tilde T \subseteq  F^1 =(F^3)^\perp \subseteq  F^0 =V,$$
where $\tilde x\in \tilde T$ are affine lifts of $\iota(x)$ and $\iota_*(T_xX)$. If
 $\iota\colon X\to \bP V$ is a Legendrian immersion, then its first prolongation is an immersion 
$$\iota^{(1)}\colon X\to \check{\mathbf{D}}$$
whose image is (locally on $X$) an $h$-dimensional horizontal submanifold of $\check{\mathbf{D}}$. Conversely, if $Y$ is any $h$-dimensional immersed  horizontal submanifold of $\check{\mathbf{D}}$ satisfying the appropriate non-degeneracy condition and $\iota \colon Y \to \Pee V$ corresponds to taking the complex line $F^3 \subseteq V$,    then $Y\to \check{\mathbf{D}}$ is the first prolongation of $\iota$.
 
\smallskip
 
For an explicit construction, let $\omega(z)$ be a generator of  $F^3$, where $z=(z_1,\dots, z_h)$ is a set of local coordinates on $(Y,0)$. We  assume that $\omega$ is in the normalized form:
\begin{equation}\label{normalize}
\omega(z) = \psi e_0 + \sum_{i=1}^h\alpha_ie_i + \sum_{i=1}^hz_if_i + f_0.
\end{equation}
Here the $z_i$ are local coordinates on the subvariety $Y$ and $\psi$ and $\alpha_i$ are functions of the $z_i$. 

\begin{proposition}\label{localz}
With the above normalization \eqref{normalize}, the local description of a horizontal subvariety $Y=Y_\varphi$ is given by 
\begin{equation}\label{eqalld}
\omega = \left(\varphi - \frac12\sum_{i=1}^h z_i\frac{\partial \varphi}{\partial z_i}\right)e_0 + \frac12\sum_{i=1}^h  \frac{\partial \varphi}{\partial z_i} e_i + \sum_{i=1}^hz_if_i + f_0.
\end{equation}
for $\varphi$ a holomorphic function of the $z_i$, i.e.\ $\D\psi = \varphi - \frac12\sum_{i=1}^h z_i\frac{\partial \varphi}{\partial z_i}$ and $\D\alpha_i = \frac12\frac{\partial \varphi}{\partial z_i}$.
\end{proposition}
\begin{proof}
The relevant exterior differential system is given by
$$
0=\langle \omega, d\omega\rangle = \langle \omega, d\psi \cdot e_0 + \sum_{i=1}^hd\alpha_i\cdot e_i + \sum_{i=1}^hdz_i\cdot f_i \rangle = -d\psi - \sum_{i=1}^h z_id\alpha_i + \sum_{i=1}^h\alpha_i dz_i.
$$
It follows that $d\psi + \sum_{i=1}^h z_id\alpha_i - \sum_{i=1}^h\alpha_i dz_i=0$. 
Setting $t_0 = \psi + \sum_{i=1}^h z_i\alpha_i$ and $s_i=2\alpha_i$ gives
\begin{align*}
dt_0 -\sum_{i=1}^hs_idz_i &= d\psi + \sum_{i=1}^h z_id\alpha_i + \sum_{i=1}^h\alpha_i dz_i-2\sum_{i=1}^h\alpha_i dz_i\\
&= d\psi + \sum_{i=1}^h z_id\alpha_i - \sum_{i=1}^h\alpha_i dz_i.
\end{align*}
This is then solved by $t_0 = \varphi(z_1, \dots, z_h)$, an arbitrary function of the $z_i$, and $\D s_i =\frac{\partial \varphi}{\partial z_i}$. Explicitly, one gets the formula \eqref{eqalld}.
\end{proof}

Given the choice of $\omega$, there is a cubic form on the tangent bundle of a horizontal subvariety by the formula
$$\Xi (\xi_1, \xi_2, \xi_3) = \left\langle \nabla_{\xi_1}\nabla_{\xi_2}\nabla_{\xi_3}\omega, \omega\right\rangle.$$
 A  computation left to the reader shows  that the cubic form  is expressed as follows (see \cite[\S5]{friedmancy}).
\begin{proposition}\label{lcubic}
With notation as above, the cubic form for $Y_\varphi$  is given by
\begin{equation}\label{cubic}
\Xi \left(\frac{\partial}{\partial z_r},\frac{\partial}{\partial z_s},\frac{\partial}{\partial z_t}\right)=-\frac12\frac{\partial^3 \varphi}{\partial z_r\partial z_s\partial z_t}.
\end{equation}
In particular, $\Xi$ is constant iff $\varphi$ is a polynomial of degree at most $3$. \qed
\end{proposition}

\begin{remark}
Compare the above results with \cite[Lemma 3.4]{voisinmirror} which says $\D\frac{\partial \alpha_i}{\partial z_j}=\frac{\partial \alpha_j}{\partial z_i}$ in the notation of Equation~\eqref{normalize}. This guarantees the existence of a \textsl{potential} $\varphi$ (denoted $F$ in \cite[p. 42]{voisinmirror}). Then, Proposition \ref{lcubic} is just \cite[Prop. 3.3]{voisinmirror}.
\end{remark}

\subsection{The global case}\label{globalconst}
We  now assume that \eqref{eqalld} is actually a description of an affine open subset of a  horizontal subvariety of $\check\bD$. Specifically, we assume that $Y\subset \bD$ satisfies $Y\cong\bA^h$ and its embedding in $\bD$ is given by \eqref{eqalld} for  some  polynomial $\varphi$ in $z=(z_1,\dots,z_h)\in \bA^h$. Additionally, suppose that the variety $Y$ is invariant under a maximally unipotent $T\in \Sp(2h+2, \Zee)$ preserving the filtration $W_\bullet$ (and defining it in the sense of monodromy weight filtrations). It is easy to see that $\varphi$ then satisfies a difference equation with terms which are polynomials in the $z_i$ of degree at most two. For example, in case $h=1$, with $T$ given by  $Tf_0 = f_0 + af_1+ be_1+ce_0$, $Tf_1 = f_1+de_1+(b-ad)e_0$, $Te_1=e_1-ae_0$ and $Te_0=e_0$,  the maximal unipotency condition says that $ad\neq 0$, and the function $\varphi$ satisfies the difference equation: 
\begin{equation}\label{eqdiff}
\varphi (z+a) -\varphi(z) = dz^2 + 2bz +(ab+c).
\end{equation}
This is only possible in general if the polynomial $\varphi$ has degree at most $3$. A similar statement, under mild nondegeneracy assumptions or by considering instead a system of $h$ commuting difference equations, is true in case $h>1$ as well. Thus, it is natural to make the following assumption on $\varphi$: 
\begin{convention}
The function $\varphi$ of \eqref{eqalld} is a \textbf{homogeneous} polynomial of degree $3$.
\end{convention}
In this case, by Euler's theorem, we can write
\begin{equation}
\omega =   - \frac12\varphi e_0 + \frac12\sum_{i=1}^h  \frac{\partial \varphi}{\partial z_i} e_i + \sum_{i=1}^hz_if_i + f_0.
\end{equation}
However, to avoid the factors of $\frac12$, we replace  $\frac12\varphi$ by $\varphi$, and assume that $Y$ is described via
\begin{equation}\label{eqalld2}
\omega =   - \varphi e_0 + \sum_{i=1}^h  \frac{\partial \varphi}{\partial z_i} e_i + \sum_{i=1}^hz_if_i + f_0,
\end{equation}
where again $\varphi$ is a homogeneous polynomial of degree $3$ in $z_1, \dots, z_h$.
Conversely, if $Y$ is so defined, then a calculation shows that $Y$ is invariant under the (symplectic) transformation $T_v$ defined by
\begin{align}\label{defTv}
T_vf_0 &= f_0 + \sum_i v_if_i + \sum_i\frac{\partial \varphi}{\partial z_i}(v)e_i + (-\varphi(v))e_0;\\
T_vf_i &= f_i + \sum_j\frac{\partial^2 \varphi}{\partial z_i\partial z_j}(v)e_j + \left(-\frac{\partial \varphi}{\partial z_i}(v)\right)e_0 \qquad (i\neq 0);\notag\\
T_ve_i &= e_i - v_ie_0 \qquad (i\neq 0);\notag\\
T_ve_0 &= e_0,\notag
\end{align}
because $T_v\omega(z) = \omega(z+v)$.  Note that $T_v$ is rational if $\varphi$ has rational coefficients and $v\in \Q^h$, and similarly $T_v$ is real if $\varphi$ has real coefficients and $v\in \R^h$.  Thus there is an action on $Y$ of an abelian unipotent group $U$ isomorphic to $\mathbb{G}_a^h$. There is also a (non-symplectic) action of $\mathbb{G}_m$ on $Y$: given $\lambda\in \mathbb{G}_m$, the automorphism $S_\lambda$ of $\Pee^{2h+1}$  defined by $S_\lambda(e_0) = \lambda^3e_0$, $S_\lambda(e_i) = \lambda^2e_i$, $S_\lambda(f_i) = \lambda f_i$, and $S_\lambda(f_0) = f_0$ satisfies $S_\lambda\omega(z) = \omega(\lambda z)$.  For $\varphi$ general, these are in fact all of the projective automorphisms of $Y$, as we shall see below.

\begin{remark} 
\begin{itemize}
\item[(1)] If instead of assuming that the subvariety $Y$ is invariant under a system of $h$ commuting difference equations, we assume that $Y$ is invariant under an $h$-dimensional unipotent abelian subgroup of the complex symplectic group satisfying the appropriate conditions, then one can always choose a complex symplectic basis in which $\varphi$ is homogeneous. This argument is essentially given in the course of the proof of Theorem~\ref{hermcubiccase} below.
\item[(2)]  With $T_v$ defined as in Equation~\eqref{defTv}, consider $ T_v-\Id$, which differs from $N_v=\log T_v$ by an invertible matrix commuting with $T_v$ and $N_v$, and let $C\subseteq \Pee^{h-1}$ be the cubic hypersurface $V(\varphi)$ defined by the homogeneous polynomial $\varphi$. Then $N_v^2 =0, N_v \neq 0$ $\iff$ $v\neq 0$ and $v$ defines a point $\bar{v}$ of $\Pee^{h-1}$ lying on $\operatorname{Sing}C$, $N_v^3 =0, N_v^2 \neq 0$ $\iff$ the point $\bar{v}$ of $\Pee^{h-1}$ lies on $C-\operatorname{Sing}C$, and $N_v^3 \neq 0$ $\iff$ the point $\bar{v}$ of $\Pee^{h-1}$ does not lie on $C$.
\end{itemize}
\end{remark}

Equation  \eqref{eqalld2} can be interpreted as a map $F\colon Y(\cong\bA^h) \to  \bA^{2h+2}$ defined by
$F(z) = \omega(z)$.
 Since the period is only defined  up to scaling, we homogenize and obtain a rational map
$\Pee^h \dasharrow \Pee^{2h+1}$ via:
$$F(z_1, \dots, z_h, t) =  -\varphi(z)e_0 + t\sum_{i=1}^h\frac{\partial \varphi}{\partial z_i}e_i + t^2\sum_{i=1}^h z_if_i + t^3f_0.$$
Thus $F$ is a dominant rational map from $\Pee^h$ to $\overline{Y}$, the closure of $Y$ in $\bP^{2h+1}$, which fails to be defined at the codimension $2$ subvariety $t= \varphi = 0$, the cubic hypersurface $C = V(\varphi)$ defined by the homogeneous cubic polynomial $\varphi$ in  the hyperplane $H \cong \Pee^{h-1}\subseteq \Pee^h$ defined by $t=0$ and with homogeneous coordinates $z_1, \dots, z_h$. Note that $F$ is regular on the affine open $\mathbb{A}^h$ defined by $t\neq 0$, and that the projection of $\Pee^{2h+1}$ onto $\Pee^h$ defined by taking the coordinates corresponding to $f_1, \dots, f_h, f_0$ induces an isomorphism on the corresponding affine open subsets $\mathbb{A}^h$. Also, if $h=1$, $F$ embeds $\Pee^1$ in $\Pee^3$ as a rational normal cubic.

\smallskip

Clearly the rational map $F$ corresponds to a subseries of the complete linear series $|\scrO_{\Pee^h}(3) -C|$ of cubics on $\Pee^h$ passing through $C$. Let $X$ be the blowup of $\Pee^h$ along $C$, with exceptional divisor $E$, ruled over $C$ via $\rho\colon E \to C$. The proper transform $H'$ of $H$ is then exceptional, i.e.\ the normal bundle $\scrO_{H'}(H')$ corresponds to $\scrO_{\Pee^{h-1}} (-2)$. Let $\overline{X}$ be the normal variety which is the contraction of $H'$ and let $\overline{E}\subseteq \overline{X}$ be the image of $E$, i.e.\ the  contraction of $E$ along $H'\cap E$. Clearly $E\cong \Pee(\scrO_{C}(1) \oplus \scrO_{C}(3)) \cong \Pee(\scrO_{C}\oplus \scrO_{C}(2))$ and $H'\cap E$ is the negative section.

\begin{proposition} With notation as above, suppose that $C$ is smooth.
\begin{itemize}
\item[(i)] The complete linear system $|\scrO_{\Pee^h}(3) -C|$ defines a base point free linear system on $X$. The associated morphism $\varphi$ blows down $H'$ and embeds $\overline{E}$ as the cone over the Veronese image of $C$ under $\scrO_{H}(2)$.
\item[(ii)] The subsystem of $|\scrO_{\Pee^h}(3) -C|$ induced by $F$, i.e.\ the linear system $\Sigma$ of cubics on $\Pee^h$ spanned by $\varphi(z)$, $t\partial \varphi/\partial z_i$, $t^2z_i$, and $t^3$, defines a base point free finite birational morphism $\overline{X}\to \overline{Y}\subseteq \Pee^{2h+1}$ which is an embedding on the affine open subset  $\mathbb{A}^h = \Pee^h -H \cong \overline{X} - \overline{E}=Y$.
\end{itemize}
\end{proposition}
\begin{proof} (i) is straightforward. For (ii), one can check by hand that $\Sigma$ is base point free. It then automatically induces a finite morphism on $\overline{X}$ since the induced linear system is  a subsystem of a very ample linear system, and it is birational since it is an embedding on the open subset $\mathbb{A}^h$.
\end{proof}

\begin{theorem}\label{thmstabilizersmooth} 
Let $h>1$. If $C$ is smooth, then the identity component of $\operatorname{Aut}\overline{Y}\subset \PGL(2h+2)$ is isomorphic to the identity component of the group of automorphisms of $\Pee^h$ fixing the hyperplane $H$ and acting as the identity on $H$, and hence is isomorphic to the semidirect product of $U\cong \mathbb{G}_a^h$ and $\mathbb{G}_m$.
\end{theorem}
\begin{proof} Since $\overline{X}\to \overline{Y}$ is a finite birational morphism, it identifies  $\overline{X}$ with the normalization of  $\overline{Y}$. Hence every automorphism $f$ of $\overline{Y}$ lifts to an automorphism of $\overline{X}$ which fixes the unique singular point and thus induces an automorphism of $X$, also denoted $f$, with $f^{-1}(H') = H'$. Since we are only considering the identity component, it follows that $f^*$ is the identity on $\Pic X$. Hence $f^{-1}(E)$ is an effective divisor linearly equivalent to $E$, and so $f^{-1}(E) = E$ since $f^{-1}(E)$ must contain every $\Pee^1$ fiber of the morphism $\rho\colon E\to C$. It follows that $f$ induces and is induced by an automorphism of $\Pee^h$, which we continue to denote by $f$ and now view as an element of $\PGL(h+1)$, such that $f(H) = H$ and $f(C) = C$. As $C$ is a cubic in $H\cong \Pee^{h-1}$ with $h>1$, there are only finitely many automorphisms of $H$ fixing $C$. Thus the restriction of $f$ to $H$ is the identity by connectedness. After normalizing $f$ by a scalar, we can then assume that $f(z_1, \dots, z_h, t) = (z_1 + v_1t, \dots, z_h + v_ht, \lambda t)$ as claimed.
\end{proof}

\begin{remark} The proof of Theorem~\ref{thmstabilizersmooth} can be interpreted as saying that, in case $C$ is smooth, the Legendrian submanifold $F(Y)\subseteq  \bA^{2h+1}\subseteq \Pee^{2h+1}$ defined by
$F(z) = \omega(z)$ cannot be completed to a (projective) Legendrian submanifold of $\Pee^{2h+1}$; in fact, the normalization of the closure $\overline{Y}$ has a unique singular point.
\end{remark}

\subsection{The Hodge--Riemann bilinear relations}\label{secthr}
We now discuss the inequalities imposed on $C$ by the Riemann-Hodge bilinear relations. Some related results, from the point of view of the mirror manifold, have been  given by Trenner--Wilson and Trenner  \cite{trennerw}, \cite{trenner}.

We continue to use the normalization
$$\omega(z) = -\varphi(z)e_0 + \sum_{i=1}^h\frac{\partial \varphi}{\partial z_i}e_i + \sum_{i=1}^h z_if_i + f_0,$$
and assume throughout this  subsection that $\varphi$ has \textbf{real} coefficients.

\begin{theorem}\label{thmhr} 
Let $z= (z_1, \dots, z_h)\in \Cee^h$ and let $y = (y_1, \dots , y_h)\in \R^h$, where $y_i = \operatorname{Im} z_i$. Then the open set of $z\in \Cee^h$ where the Hodge--Riemann inequalities are satisfied is given by
 the set of $z\in \Cee^h$ such that $\varphi(y) < 0$ and the signature of the quadratic form $\D \left (\frac{\partial ^2\varphi}{\partial z_i \partial z_j}\right)(y)$ is $(h-1,1)$.
 \end{theorem}
\begin{proof} We shall just outline the argument and omit many of the calculations. The Hodge--Riemann inequalities are the statement that $\sqrt{-1}\langle \omega , \bar{\omega}\rangle > 0$ and that the   form on $V^{2,1}$ given by $H(\psi) = \sqrt{-1}\langle \psi , \bar{\psi}\rangle $ is negative definite. Then the fact that $z$ must satisfy $\varphi(y) < 0$ follows from:

\begin{equation}  \langle \omega , \bar{\omega}\rangle = 8\sqrt{-1}\varphi(y).  
\end{equation}

Next, to compute the sign on $V^{2,1}$, define
$$\omega_i =\frac{\partial \omega}{\partial z_i} = -\frac{\partial \varphi}{\partial z_i}e_0 + \sum _{j=1}^h\left(\frac{\partial ^2\varphi}{\partial z_i \partial z_j}\right)e_j + f_i.$$
Then a computation gives
\begin{equation} 
 \langle \omega_i, \bar{\omega}_j \rangle = 2\sqrt{-1}\frac{\partial ^2 \varphi}{\partial z_i \partial z_j}(y) .  
\end{equation}
We must modify $\omega_i$ by a multiple of $\omega$ to make it orthogonal to $\bar{\omega}$, and hence an element of $V^{2,1}$. To find the correct multiple,  use:
\begin{equation}  \langle \omega_i, \bar{\omega}  \rangle = - \frac{\partial \varphi}{\partial z_i}(z-\bar{z}) = - \frac{\partial \varphi}{\partial z_i}(2\sqrt{-1}y) = 4\frac{\partial \varphi}{\partial z_i}(y) .  
\end{equation}
Using the skew symmetry and reality of the  pairing, we have:
$$\langle \omega, \bar\omega_i\rangle = \overline{\langle \bar\omega,  \omega_i\rangle} = -\overline{\langle \omega_i, \bar\omega\rangle}=- 4\frac{\partial \varphi}{\partial z_i}(y).$$
Now let us choose $a_i$ such that, with $\psi_i = a_i\omega + \omega_i$, we have $\langle \psi_i, \bar\omega \rangle = 0$, i.e.\ $\psi_i\in V^{2,1}$. Writing this as
$0 = \langle \psi_i, \bar\omega \rangle = a_i\langle \omega, \bar\omega \rangle + \langle \omega_i, \bar{\omega}  \rangle$,
we see that
$$a_i = -\frac{\langle \omega_i, \bar{\omega}  \rangle}{\langle \omega, \bar\omega \rangle}.$$
We now work out the signature of the intersection matrix $(\langle \psi_i, \bar\psi_j \rangle)$. We shall use the shorthand $\D \varphi_i = \frac{\partial \varphi}{\partial z_i}$ and $\D\varphi_{ij} = \frac{\partial ^2\varphi}{\partial z_i \partial z_j}$. Then a calculation gives
\begin{equation}
 \langle \psi_i, \bar\psi_j \rangle = \frac{2\sqrt{-1}}{\varphi(y)}(-\varphi_i(y)\varphi_j(y) + \varphi(y)\varphi_{ij}(y)) .  
\end{equation}
We could also write this as
$$\langle \psi_i, \bar\psi_j \rangle = 2\sqrt{-1} \varphi(y)\frac{\partial ^2}{\partial z_i\partial z_j}\log \varphi (y).$$
Now, we want the Hermitian matrix $\D\sqrt{-1}(\langle \psi_i, \bar\psi_j \rangle) =-\frac{2}{\varphi(y)}(-\varphi_i(y)\varphi_j(y) + \varphi_{ij}(y))$ to be negative definite, and hence, since we have already assume $\varphi(y) < 0$, we want the real symmetric matrix $(-\varphi_i(y)\varphi_j(y) + \varphi(y)\varphi_{ij}(y))$ to be negative definite.

Let $B(x,y) = \sum_{i=1}^hx_iy_i$ be the standard inner product on $\R^h$ and let $\Upsilon$ and $\Upsilon_j$ be the   vectors defined by
$$\Upsilon(y) = (\varphi_1(y), \dots, \varphi_h(y)); \qquad \Upsilon_j(y) = (\varphi_{1j}(y), \dots, \varphi_{hj}(y)).$$
Then, by Euler's theorem, we have
$$B(y, \Upsilon(y))  = 3\varphi(y);\qquad 
B(y, \Upsilon_j(y))  = 2\varphi_j(y).
$$
In particular, we see that $y$ is not orthogonal to $\Upsilon(y)$ with respect to $B$, and hence that every vector in $\R^n$ can be uniquely written as $ty+ v$ for some $t\in \R$, where $v\in \Upsilon^\perp$ (the perpendicular space for the standard inner product).
Also, for a vector $(v_1, \dots, v_h) \in \R^h$,
$${}^tv(\varphi_i(y)\varphi_j(y))v =  (B(v, \Upsilon(y)) )^2.$$
Hence, if $v\in \Upsilon^\perp$, then 
$${}^tv(-\varphi_i(y)\varphi_j(y) + \varphi(y)\varphi_{ij}(y))v = {}^tv(\varphi(y)\varphi_{ij}(y))v.$$
Now  ${}^ty(\varphi(y)\varphi_{ij})y = 6\varphi^2(y) >0$, and 
\begin{gather*}
{}^ty(-\varphi_i(y)\varphi_j(y) + \varphi(y)\varphi_{ij}(y))y = - (B(y, \Upsilon(y)) )^2+ \varphi(y)  B(y, 2\Upsilon(y))  \\
= -9\varphi^2(y) + 6\varphi^2(y)= -3\varphi^2(y) < 0.
\end{gather*}

 Note that, if $v\in \Upsilon^\perp$, ${}^tv(\varphi_{ij})y = 2B(v, \Upsilon(y)) =0$, and likewise ${}^tv(\varphi_i\varphi_j )y =  3\varphi(y)B(v, \Upsilon(y)) =0$. Hence $\Upsilon^\perp$ is contained in the  orthogonal space to $y$ corresponding to the quadratic form $(\varphi_{ij})$ (or to $(\varphi(y)\varphi_{ij})$) and to $(-\varphi_i(y)\varphi_j(y))$. Thus, applying the form $(-\varphi_i(y)\varphi_j(y) + \varphi(y)\varphi_{ij}(y))$ to a vector of the form $ty+ v$ with $t\in \R$ and $v\in \Upsilon^\perp$ gives
$${}^t(ty+ v)(-\varphi_i(y)\varphi_j(y) + \varphi(y)\varphi_{ij}(y))(ty+v) = -3t^2\varphi^2(y) + {}^tv(\varphi(y)\varphi_{ij}(y))v.$$ 
Now, if the signature of the matrix $(\varphi(y)\varphi_{ij}(y))$ is $(1, h-1)$, then $(\varphi(y)\varphi_{ij}(y))$ is negative definite on  $\Upsilon^\perp$ and hence 
$${}^t(ty+ v)(-\varphi_i(y)\varphi_j(y) + \varphi(y)\varphi_{ij}(y))(ty+v)\leq 0,$$
 with equality $\iff$ $t=v=0$. Conversely, if $(-\varphi_i(y)\varphi_j(y) + \varphi(y)\varphi_{ij}(y))$ is negative definite, then since the induced quadratic form agrees with that for $(\varphi(y)\varphi_{ij}(y))$ on $\Upsilon^\perp$, $(\varphi(y)\varphi_{ij}(y))$ has at least $h-1$ negative eigenvalues, and since it has at least one positive eigenvalue corresponding to $y$, its signature is $(1, h-1)$.
\end{proof}

\begin{remark} Let $\overline{Y}$ be the subvariety of $\Pee^{2h+1}$ defined by (the homogenization of) Equation~\eqref{eqalld2} and let $Z$ be the open subset of $\overline{Y}$ defined by the Hodge--Riemann inequalities. Clearly, in case $\varphi =0$, $Z=\emptyset$. In all other cases it is nonempty: in fact, if $L$ is a line through the origin in $\Cee^h$ defined over $\R$ and such that $\varphi|L \neq 0$, and $f\colon L \to \overline{Y}$ is the natural morphism, then $f^{-1}(Z)$ is an upper half plane $\fH$ embedded in its compact dual, which is isomorphic to $\bP^1$. 
\end{remark}

\section{The complex case}
\label{sectequations2}
The description of Equation~\eqref{eqalld2} does not suffice to handle the case where $\varphi =0$ (the unit ball case, see Remark \ref{remmaxvar}). Also, in certain situations we would like to relax the non-degeneracy condition that the derivatives $\partial \omega/\partial z_i$ span all of $F^2/F^3$.  To describe the relevant horizontal subvarieties which we shall encounter, we fix the following notation: 
\begin{notation}
Let $\Cee^{2h+2}$ have a complex basis $\varepsilon_0, \dots, \varepsilon_h, \delta_0, \dots, \delta_h$ and  a real structure such that $\bar{\varepsilon}_i =\delta_i$. Let $\langle  \cdot, \cdot \rangle$ be the unique  symplectic form  on $\Cee^{2h+2}$ such that $\langle  \varepsilon_i, \varepsilon_j \rangle = \langle  \delta_i, \delta_j \rangle =0$ for all $i$ and $j$ and $\langle  \varepsilon_0, \delta_0 \rangle =2\sqrt{-1}$, $\langle  \varepsilon_i, \delta_i \rangle =-2\sqrt{-1}$ for $1\leq i\leq a$, and $\langle  \varepsilon_i, \delta_i \rangle =2\sqrt{-1}$ for $a+1\leq i\leq h$ (where we allow the possibility that $a=h$ and $b=0$). Note that the $\varepsilon_i$ span a maximal  isotropic complex subspace of $\Cee^{2h+2}$, as do the $\delta_i$.
\end{notation}

More invariantly, we suppose that $V_+$ is an $(h+1)$-dimensional complex vector space with a non-degenerate Hermitian inner product $[\cdot, \cdot]$, which thus gives a complex anti-linear isomorphism $f$ from $V_+$ to $V_-$, where $V_-$ is the dual of $V_+$. The isomorphism $f$ then defines a real structure on $V_\bC:=V_+\oplus V_-$ (with $V_\bC=V\otimes_\bR\bC$); for example, if $v\in V_+$, $\operatorname{Re}v = \frac12(v+ f(v))$ and $\operatorname{Im}v = \frac1{2\sqrt{-1}}(v- f(v))= \operatorname{Re}(-\sqrt{-1}v)$. If $\varepsilon_0, \dots, \varepsilon_h$ is a diagonal basis for $V_+$ with respect to the form $[\cdot, \cdot]$, i.e.\ $[\varepsilon_i, \varepsilon_j]= (-1)^{a_i}\delta_{ij}$ with $a_i=0$ or $1$, and we set $\delta_i =\bar{\varepsilon}_i$, then $\delta_i =(-1)^{a_i}\varepsilon_i^*$, where $\varepsilon_0^*, \dots, \varepsilon_h^*$ is the dual basis for $V_-\cong V_+\spcheck$. On $V_\bC$, there is the natural symplectic form 
$$\langle  v_1+\xi_1, v_2 +\xi_2 \rangle_0 = \xi_2(v_1) - \xi_1(v_2).$$
Applying the form to a pair  of real vectors gives
$$\left\langle  \frac12(v+ f(v)), \frac12(w+ f(w)) \right\rangle_0 =  \frac14(f(w)(v) - f(v)(w)) = -\frac{\sqrt{-1}}{2}\operatorname{Im} [w,v].$$
Thus $\langle  \cdot, \cdot \rangle = 2\sqrt{-1} \langle  \cdot, \cdot \rangle_0$ is a real symplectic form on $V$, with the property that
$$\langle  \varepsilon_i, \delta_i \rangle =\langle  \varepsilon_i, (-1)^{a_i}\varepsilon_i^* \rangle = (-1)^{a_i}2\sqrt{-1}.$$
Conversely, with notation as at the beginning of this section, we define $V_+ =   \operatorname{span}\{\delta_0, \dots, \delta_h\}$ and $V_- =   \operatorname{span}\{\varepsilon_0, \dots, \varepsilon_h\}$ and use the symplectic pairing to define an identification of $V_-$ with the dual of $V_+$, and define a complex anti-linear isomorphism from $V_+$ to $V_-$ by taking the Hermitian form
$$[v, w] = \frac{1}{2\sqrt{-1}}\langle \bar{v}, w\rangle.$$

\subsection{The non-degenerate case} 
In this case, we let $a=h$ and $b=0$ and work out the analogue of the preceding section using the complex basis $\varepsilon_0, \dots, \varepsilon_h, \delta_0, \dots, \delta_h$. Writing
$$\omega(z) = \Psi(z)\varepsilon_0 +\sum_{i=1}^hA_i(z)\varepsilon_i + \sum _{i=1}^hz_i\delta_i + \delta_0,$$
the only difference between this picture and that of Equation~\eqref{eqalld} is the sign change between $\langle \varepsilon_0, \delta_0\rangle$ and $\langle \varepsilon_i, \delta_i\rangle$ for $i>0$, giving: there exists a function $\Phi$ such that 
\begin{equation}\label{eqallddegen1}
\omega(z) = \left(\Phi - \frac12\sum_{i=1}^hz_i\frac{\partial \Phi}{\partial z_i}\right)\varepsilon_0 -\frac12\sum_{i=1}^h\frac{\partial \Phi}{\partial z_i}\varepsilon_i + \sum _{i=1}^hz_i\delta_i + \delta_0.
\end{equation}
This  determines a filtration $F^\bullet$ of $\Cee^{2h+2}$, by taking 
$$F^3(z) = \Cee \cdot \omega(z)\subseteq F^2(z) = \operatorname{span}\left\{\omega(z), \frac{\partial \omega}{\partial z_1} , \dots ,  \frac{\partial \omega}{\partial z_h}\right\},$$
and then setting $F^1(z)  = (F^3(z))^\perp$ and $F^0 =V_\bC\cong \Cee^{2h+2}$. As in Section \ref{sectequations}, this describes the germ $(Y,0)$ of a horizontal subvariety of $\check \bD$. Also, as in \S\ref{globalconst}, if $\Phi$ are polynomials on $\bA^h$, one obtains a horizontal subvariety $Y\cong \bA^h$ of $\check \bD$.

\begin{remark}
Note that, in case $\Phi$ is a homogeneous cubic polynomial, then $\Psi = -\frac12\Phi$ as before. Normalizing again to eliminate the factor of $1/2$ then gives
\begin{equation}\label{eqallddegen2}
\omega(z) = -\Phi\varepsilon_0 -\sum_{i=1}^h\frac{\partial \Phi}{\partial z_i}\varepsilon_i + \sum _{i=1}^hz_i\delta_i + \delta_0.
\end{equation}
If $\Phi$ in Equation~\ref{eqallddegen2} is $0$, we are in the case of the unit ball, as we shall see below. If $\Phi$ is a homogeneous linear or quadratic polynomial, the complex variation of Hodge structure defined by Equation~\ref{eqallddegen2} reduces to the case of the unit ball via an appropriate symplectic transformation.

For a general choice of the cubic $\Phi$, it follows from Theorem~\ref{thmstabilizersmooth} that the Zariski closure $\hat{Z}=\overline Y$ of a horizontal subvariety $Y$ cannot be simultaneously expressed in the form Equation~\eqref{eqallddegen2} and the form Equation~\eqref{eqalld} for two different choices of bases $\varepsilon_0, \dots, \varepsilon_h, \delta_0, \dots, \delta_h$ and $e_0, \dots, e_h, f_0, \dots, f_h$. However, as we shall see in the next section, this is always possible in the Hermitian symmetric tube domain case.
\end{remark}

If $\Phi$ and $\partial \Phi/\partial z_i$ both vanish at the origin, then the Hodge structure at $0$ is  determined by $F^3 =\Cee\cdot\delta_0\subseteq F^2 =  \operatorname{span}\{\delta_0, \delta_1, \dots, \delta_h\}=V_+$, and hence $V^{3,0}(0) = \Cee\cdot\delta_0$, $V^{2,1}(0)=  \operatorname{span}\{\delta_1, \dots, \delta_h\}$, $V^{1,2}(0)=  \operatorname{span}\{\varepsilon_1, \dots, \varepsilon_h\}$, and $V^{0,3}(0) = \Cee \cdot \varepsilon_0$. Thus the filtration automatically  satisfies the Hodge--Riemann inequalities, which was the reason for our choice of signs. In particular, the open subset $Z$ of $\hat{Z}$ is always nonempty in this case. We will not write out the Hodge--Riemann inequalities in general, although this is straightforward to do, except to note that it is easy to check directly that, in case $\Phi$ is identically $0$, the inequalities amount to:
$$\sum_{i=1}^h |z_i|^2 < 1,$$
i.e.\ the horizontal subvariety in question is the $h$-dimensional complex unit ball. 

\subsection{A degenerate case} 
Here we are  interested in the case $a< h$ and hence $b>0$. For $z = (z_1, \dots, z_a)$, we define a holomorphically varying line in $\Pee^{2h+1}$, the analogue in some sense of Equation~\eqref{eqalld}, via
\begin{equation}\label{eqallddegen3}
\omega(z) = \sum_{k=a+1}^hq_k(z)\delta_k + \sum_{i=1}^az_i\delta_i + \delta_0. 
\end{equation}
Taking derivatives, we have, for $1\leq i\leq a$,
$$\omega_i(z) = \frac{\partial \omega}{\partial z_i} = \sum_{k=a+1}^h \frac{\partial q_k}{\partial z_i}\delta_k + \delta_i.$$
To complete this to a Hodge filtration, we define
$$F^2 = \operatorname{span}\{\omega(z), \omega_1(z), \dots, \omega_a(z)\} \oplus \operatorname{span}\{\omega(z), \omega_1(z), \dots, \omega_a(z)\}^\perp,$$
where $\operatorname{span}\{\omega(z), \omega_1(z), \dots, \omega_a(z)\}^\perp$ denotes the orthogonal space in $W\spcheck$. Explicitly, for $a+1\leq k\leq h$,  let
$$\omega_k(z) = \varepsilon_k + \sum _{i=1}^a\frac{\partial q_k}{\partial z_i}\varepsilon_i + s_k\varepsilon_0,$$
where we set
$$s_k = \sum_{i=1}^az_i \frac{\partial q_k}{\partial z_i}-q_k.$$ 
As before, holomorphic $q_k$ describe the germ of a horizontal subvariety $(Y,0)$, and polynomial $q_k$ correspond to $Y$ being affine. Note in particular that if $q_k$ is a homogeneous quadratic polynomial for all $k$, then $s_k = q_k$.

Then $\omega_{a+1}(z), \dots, \omega_h(z)$ is a basis for $\operatorname{span}\{\omega(z), \omega_1(z), \dots, \omega_a(z)\}^\perp$.
It follows that $\omega(z), \omega_1(z), \dots , \omega_h(z)$  span an isotropic subspace $F^2$ of $V_\bC\cong \Cee^{2h+2}$. Setting $F^1 = (\omega)^\perp$ (under the symplectic form $\langle  \cdot, \cdot \rangle$) then defines a weight three complex variation of Hodge structure $F^3 \subseteq F^2\subseteq F^1 \subseteq F^0 =V_\bC$ of CY type.

\begin{remark}\label{extremark} In case $b> 0$, the above variation of Hodge structure is never maximal, and in fact can be  embedded in an $h$-dimensional complex variation of Hodge structure given by a slight variant of Equation~\eqref{eqallddegen1}. The simplest such choice is as follows: for $z= (z_1, \dots, z_h)$, setting
\begin{align*}
\Phi_0(z) &= -2\sum_{k=a+1}^hz_kq_k (z_1, \dots, z_a);\\
\Psi_0(z) &= \Phi_0(z) - \frac12\sum_{i=1}^hz_i\frac{\partial \Phi_0}{\partial z_i};\\
A_i(z) &= \frac12\frac{\partial \Phi_0}{\partial z_i} = \begin{cases} -\sum_{k=a+1}^hz_k\frac{\partial q_k}{\partial z_i}, &\text{if $i\leq a$;}\\
-q_i, &\text{if $i\geq a+1$,}
\end{cases}
\end{align*}
and   defining
\begin{equation}\label{eqdegenext}
\omega(z) = \Psi_0(z)\varepsilon_0 - \sum_{i=1}^aA_i(z)\varepsilon_i - \sum_{k=a+1}^h A_k(z)\delta_k + \sum_{i=1}^az_i\delta_i + \sum_{k=a+1}^h  z_k\varepsilon_k + \delta_0,
\end{equation}
then we get a variation of Hodge structure of the form of Equation~\eqref{eqallddegen1}  which specializes to Equation~\eqref{eqallddegen3} when $z_{a+1} =\dots = z_h =0$. Here, we are free to replace $\Phi_0$ by $\Phi = \Phi_0 + \Phi_1$, where $\Phi_1$ is any holomorphic function in $z_1, \dots, z_h$ which vanishes to order at least two in $z_{a+1}, \dots, z_h$. 

If the $q_k$ are all homogeneous quadratic polynomials (as will be the case in our application), then $\Phi_0$ is a homogeneous cubic polynomial and $\Psi_0 = - \frac12 \Phi_0$. Given the direct sum decomposition $\Cee^h = \Cee^a \oplus \Cee^b$, where we write $z = (z_1, \dots, z_h) = (z', z'')$, say, then the $q_k$ are invariantly given as a linear map $Q \colon (\Cee^a) \otimes \Cee^a) \to (\Cee^b)\spcheck$, where $(\Cee^b)\spcheck$ is the dual vector space to $\Cee^b$, and (up to the factor $-2$) $\Phi_0(z) = \langle z'', Q(z'\otimes z')\rangle$ using the pairing $\langle \cdot, \cdot \rangle$ on $\Cee^b\otimes (\Cee^b)\spcheck$.
In this case, if we modify $\Phi_0$ by adding a term $\Phi_1$ as above, it is natural to require that $\Phi_1$ is also a homogeneous cubic polynomial vanishing to order at least two in $z_{a+1}, \dots, z_h$. Then $\Phi = \Phi_0 + \Phi_1$ is a homogeneous cubic polynomial whose zero locus contains the linear space $z_{a+1}= \dots= z_h=0$.

If $\D q_k(0) =\frac{\partial q_k}{\partial z_i} =0$ for all $i$ and $k$, then the Hodge filtration at $0$ is determined by $F^3= V^{3,0}(0) =\Cee\cdot\delta_0\subseteq  F^2= \operatorname{span}\{\delta_0, \delta_1, \dots, \delta_a, \varepsilon_{a+1}, \dots, \varepsilon_h\}$. Again, by our choice of signs, this satisfies the Hodge--Riemann inequalities.  Note that the Hodge--Riemann inequalities for the degenerate case given by Equation~\eqref{eqallddegen2} follow from those for the larger family given by Equation~\eqref{eqdegenext}. 
\end{remark}

\section{The weight three Hermitian symmetric case}\label{sect5}
Our goal now is to show that the Hermitian symmetric weight three Calabi--Yau examples can be described as in the previous two sections, and to discuss rationality issues. We begin by recalling some standard facts about Hermitian symmetric spaces (see also \S\ref{proofclassify}). If $G$ is a simple real algebraic group  with maximal compact subgroup $K$ such that $\mathcal{D}= G(\R)/K$ is a Hermitian symmetric space, then the real Lie algebra $\mathfrak{g} =\mathfrak{k}\oplus \mathfrak{p}$ decomposes into the $+1$ and $-1$ eigenspaces of the Cartan involution. Moreover, the complexification $\mathfrak{p}_\Cee$ of $\mathfrak{p}$ has a direct sum decomposition $\mathfrak{p}_\Cee= \mathfrak{p}_+ \oplus \mathfrak{p}_-$, where  each subspace $\mathfrak{p}_\pm$ is an abelian Lie algebra corresponding to a unipotent subgroup $P_\pm$ of $G(\Cee)$. Then $\mathfrak{k}_\Cee\oplus \mathfrak{p}_+$ is the Lie algebra of a parabolic subgroup $P_0(\Cee) = K(\Cee)P_+$ defined over $\Cee$.  
Note that $G(\Cee)/P_0(\Cee)$ is the compact dual of $\mathcal{D}$, in the sense that $P_0(\Cee)\cap G(\R) = K$ and thus there is an inclusion $\calD = G(\R)/K \subseteq \check\calD = G(\Cee)/P_0(\Cee)$.
By the Borel and Harish-Chandra embedding theorem, $G(\R)\subseteq P_-K(\Cee) P_+= P_-P_0(\Cee)$ and $\mathcal{D}$ is embedded in the open subset $P_-P_0(\Cee)/P_0(\Cee)$ of $\check\calD$. Clearly, $P_+$ is the unipotent radical of $P_0(\Cee)$, $K(\Cee)$ is the Levi subgroup,  and   $K(\Cee) P_-$ is the opposite parabolic, with unipotent radical $P_-$. 

In terms of Hodge structures associated to a representation of a reductive form of $G$, having chosen a reference Hodge structure $F^\bullet_0$, $K$ is the stabilizer of  $F^\bullet_0$ in $G(\bR)$ and $P_0(\bC)$ is the stabilizer of $F^\bullet_0$ in $G(\bC)$.
 Hence every Hodge structure is in the $P_-$-orbit of $F^\bullet_0$.

\begin{lemma} The parabolic subgroup $P_0(\Cee)$ is conjugate in $G(\Cee)$ to a parabolic subgroup $P(\Cee)$, where $P$ is any maximal real parabolic subgroup corresponding to a zero-dimensional boundary component.
\end{lemma}
\begin{proof} By a theorem of Kor\'anyi--Wolf \cite[Theorem  5.9]{korwolf} (and the remark at the end of Section 5 in op. cit.), the subgroup $P=(\operatorname{Ad}(c))^{-1}P_0(\Cee)\cap G(\R)$ is a real parabolic subgroup of $G$, where $c= c_r$ is the Cayley transform (which is all that we shall need in this section). The more precise identification of $P$ as the maximal parabolic subgroup corresponding to a zero-dimensional boundary component is given in \cite[Corollary  6.9]{wolfkor}.
\end{proof}

Let  $P$ be a  maximal real parabolic subgroup corresponding to a $0$-dimensional boundary component of $\mathcal{D}$. Then  its unipotent radical $U$ satisfies: the complexification $U(\Cee)$ is conjugate in $G(\Cee)$ to $P_-$. Note that, if $k$ is a subfield of $\R$ such that $G$ and $P$ are defined over $k$, then $U$ is also defined over $k$.

\begin{remark}\label{uniremark} Suppose that $T \in G(\Zee)$ is a monodromy matrix corresponding to a holomorphic map $\Delta^* \to \calD/G(\Zee)$. Then $T$ is conjugate in $G(\R)$ to an element of $U(F)$, where $U(F)$ is the center of the unipotent radical of the real parabolic subgroup corresponding to an appropriate rational boundary component $F$ (see for example  \cite[Theorem, p.\ 279]{amrt} or \cite[Satz 12]{schmid2}). One can show (cf.\ \cite[Theorem 3, p.\ 240]{amrt}) that, for such a $T$, the element $T$ is conjugate in $G(\R)$ to an element of $U$, the unipotent radical of a  maximal real parabolic subgroup corresponding to a $0$-dimensional boundary component, and hence in $G(\Cee)$ to an element of $P_-$.
\end{remark}

Finally, we shall use the following:

\begin{lemma}\label{wtthreelemma} Let $\iota \colon \check\bD \to \Pee^{2h+1}$ be the morphism defined by taking the complex line $F^3$ and let $\calD$ be a Hermitian subvariety of $\bD$.  Then $\iota$ induces an embedding $\calD \to \Pee^{2h+1}$.
\end{lemma}
\begin{proof}
Given $o_1 ,  o_2\in \calD$, with $x_1$, $x_2$   the corresponding points of $\Pee^{2h+1}$, suppose that $x_1 =x_2$. View the maximal compact subgroup $K$ as the stabilizer of $o_1$. Then $P_0$ is the stabilizer of the point $x_1\in \Pee^{2h+1}$ (as one can easily check; see also Remark~
\ref{remminuscule}). There exists a $g\in G(\R)$ such that $go_1 = o_2$, and hence $g\cdot x_1 = x_2 =x_1$. It follows that $g\in G(\R)\cap P_0 = K$, and hence $o_1= o_2$. 
\end{proof}

\subsection{Tube domain case}\label{ssectube}
In the tube domain case, we can give a very explicit description of the complex Lie algebra $\mathfrak{g}_\Cee$ and of the minuscule representation $V_\bC$ of $G(\Cee)$ corresponding to the weight three variation of Hodge structure (see Lemma \ref{mainlemma} and Corollary \ref{thmclassify3}). This discussion is closely related to some of the work of Landsberg and Manivel (esp. \cite{lm1} and \cite{lm4}). 

In the tube domain case, $\mathfrak{k}_\Cee \cong  \mathfrak{g}'_\bC\oplus \mathfrak{z}$, where $\mathfrak{g}_\bC'$ is a semisimple (complex) Lie algebra, $\mathfrak{z}$ is the one-dimensional center of  $\mathfrak{k}_\Cee$, and, as a  $\mathfrak{k}_\Cee$-module, $\mathfrak{g}_\Cee \cong \mathfrak{k}_\Cee \oplus W \oplus W\spcheck$, where $W$ is an irreducible representation of $\mathfrak{k}_\Cee$,   $W\spcheck$ is its dual, and  $W= \mathfrak{p}_-$, $W\spcheck = \mathfrak{p}_+$ in the notation above. There is also a $\mathfrak{g}_\bC'$-invariant linear map $B\colon \operatorname{Sym}^2 W \to W\spcheck$ such that the associated cubic tensor $C(w_1, w_2, w_3) = \langle B(w_1, w_2), w_3\rangle$ is   nonzero, $\mathfrak{g}_\bC'$-invariant, and symmetric, where $\langle \cdot, \cdot  \rangle$ is the evaluation pairing $W\otimes W\spcheck \to \Cee$.  The minuscule representation $V_\bC$ of $\mathfrak{g}_\Cee$ splits as a $\mathfrak{k}_\Cee$-module as
\begin{equation}\label{decomposek}
V_\bC \cong \underline{\Cee} \oplus W \oplus W\spcheck \oplus \underline{\Cee},
\end{equation}
where $\mathfrak{g}_\bC'$ acts in the standard way on $W$ and $W\spcheck$ and trivially on the two factors $\underline{\Cee}$, and the center $\mathfrak{z}=\bC\cdot H_0$ (in the notation of \S\ref{proofclassify}) acts diagonally with weights $-3/2,  -1/2, 1/2, 3/2$ respectively (compare to \S\ref{listrep}). We write an element of $V_\bC$ as a $4$-tuple $\mathbf{x}=(t, v, \xi, s)$, where $t,s\in \Cee$, $v\in W$, and $\xi \in W\spcheck$. Note that $V_\bC$ has a natural real structure, and complex conjugation exchanges $W$ and $W\spcheck$ and the two factors $\underline{\Cee}$. In terms of Hodge structures, with $K$ the stabilizer of a reference Hodge structure, this says that the first factor $\underline{\Cee}=V^{3,0}$, $W=V^{2,1}$, $W\spcheck=V^{1,2}$, and the last factor $\underline{\Cee}=V^{0,3}$. An interesting example here is the $E_7$ case: $V$ is the unique minuscule $E_7$ representation, the semi-simple part of $K$ is $E_6$, and $W$ and $W\spcheck$ are the two minuscule representations of $E_6$.

We are interested in the action of the space $W$ on $V$. We will normalize the symplectic form on $V$ (which is unique up to a scalar) to be given by
$$\left\langle (t_1, v_1, \xi_1, s_1), (t_2, v_2, \xi_2, s_2)\right\rangle = s_1t_2 - s_2t_1 + \langle \xi_1, v_2\rangle - \langle \xi_2, v_1\rangle .$$
The $W$-action is given by $w\in W \mapsto N_w$, where
$$N_w (t, v, \xi, s) = (0, tw,  B(w,v), -\langle \xi, w\rangle).$$
Here the fact that $[N_{w_1}, N_{w_2}] = 0$ and that $N_w$ preserves the symplectic form, i.e.\ that $\langle N_w\mathbf{x}_1, \mathbf{x}_2 \rangle = -\langle \mathbf{x}_1, N_w\mathbf{x}_2 \rangle = \langle N_w\mathbf{x}_2, \mathbf{x}_1 \rangle$ follow from the symmetry of $C(w_1, w_2, w_3) = \langle B(w_1, w_2), w_3\rangle$. Note that
\begin{align*}
N_w^2(t, v, \xi, s)) &= (0,0,  tB(w,w), -\langle B(w,v),w\rangle), \\
  N_w^3(t, v, \xi, s)) &= (0,0,0, -t\langle B(w,w),w\rangle),
  \end{align*}
 and hence $(\exp N_w)(t, v, \xi, s)$ is given by
$$ \left(t, v+tw, \xi + B(w,v) + \frac{t}{2}B(w,w), s-\langle \xi, w\rangle -\frac12\langle B(w,v), w\rangle - \frac{t}{6}\langle B(w,w), w\rangle\right).$$
In particular,
$$(\exp N_w)(1,0,0,0) = \left(1, w, \frac{1}{2}B(w,w), - \frac{1}{6} C(w,w , w)\right).$$
This is in the form of Equation~\eqref{eqallddegen2} with $\Phi(z)  = - \frac{1}{6} C(w,w , w)$, for an appropriate  complex basis $\{\varepsilon_0, \varepsilon_1, \dots, \varepsilon_h, \delta_1, \dots, \delta_h, \delta_0\}$ and the cubic form $\varphi$ is only defined over $\Cee$. (Note that in Equation~\eqref{eqallddegen2} there is a minus sign in front of the factor $\frac12B(w,w)$ since, up to the common factor of $\sqrt{-1}$,  $\varepsilon_1, \dots, \varepsilon_h$ are the negatives of the dual basis for $\delta_1, \dots, \delta_h$.)  We are more interested in the variant corresponding to Equation~\eqref{eqalld2} where everything is defined over $\R$ or over $\Q$, as follows:

\begin{theorem}\label{hermcubiccase} Let $k$ be a subfield of $\R$. Suppose that $G$ is defined over $k$ and that  the maximal real parabolic subgroup $P$ is also defined over $k$, so that its unipotent radical $U$ is defined over $k$ as well. Then there exists a symplectic $k$-basis $e_0, e_1, \dots, e_h, f_1, \dots, f_h, f_0$ of $V$  and a cubic polynomial $\varphi(z_1, \dots, z_h)$ defined over $k$, such that $\mathcal{D}$ is biholomorphic to the locally closed  subset of $\Pee^{2h+1}$ given by the set of lines of the form 
$$\omega =   - \varphi e_0 + \sum_{i=1}^h  \frac{\partial \varphi}{\partial z_i} e_i + \sum_{i=1}^hz_if_i + f_0$$
and which satisfy the Hodge--Riemann inequalities (see Theorem \ref{thmhr}), and the action of $U$ is given by the unipotent subgroup corresponding to $z\mapsto z+v$.
\end{theorem}
\begin{proof} The abelian unipotent subgroup $U$ defines a filtration $V^\bullet$ of $V$, defined over $k$: if $\mathfrak{u}$ is the Lie algebra of $U$, then $V^0 = V$, $V^1 = \mathfrak{u}(V)$, $V^2 = \mathfrak{u} (V^1)$, and $V^3 = \mathfrak{u} (V^2)$. Of course, the filtration $V^\bullet$ is isomorphic over $\Cee$ to the filtration defined by the action of $W$ on $V$ described above. In particular, $\dim V^0/V^1 = \dim V^3 = 1$ and $\dim V^1/V^2 = \dim V^2/V^3 = h$, and the symplectic form pairs $V^0/V^1$ with $V^3$ and $V^1/V^2$ with $V^2/V^3$. The group $P$ preserves the filtration $V^\bullet$, and in fact $P$ is the stabilizer in $G$ of $V^3$. For $w\in \mathfrak{u}$, if $N_w$ is the corresponding nilpotent endomorphism of $V$, if $f_0$ is a nonzero element of $V^0/V^1$ and $e_0\in V^3$ is the dual element, then $w\mapsto N_w(f_0)$ defines a $k$-linear isomorphism $\mathfrak{u} \to V^1/V^2$ and hence identifies $V^2/V^3$ with $\mathfrak{u}\spcheck$ over $k$. Write  $N_{w_1}N_{w_2}N_{w_3}(f_0)= C_0(w_1,w_2,w_3)e_0$, where $C_0$ is a symmetric trilinear form defined over $k$. Similarly, the linear map $w\mapsto N_wN_v(e_0)\bmod V^3$ defines a symmetric bilinear form $B_0\colon \mathfrak{u}\otimes\mathfrak{u} \to \mathfrak{u}\spcheck$, and $C_0(w_1,w_2,w_3) = \langle B_0(w_1,w_2),,w_3\rangle$ from the definition. Note that, over $\Cee$, $C_0$ and $B_0$ are projectively equivalent to $C$ and $B$ as defined above.

Since $U(\Cee)$ is conjugate to $P_-$, there exists a point $o\in G(\Cee)/P(\Cee)$ such that $U(\Cee)\cdot o$ is a nonempty   Zariski open subset of  $G(\Cee)/P(\Cee)$. Since $G(k)$ is Zariski dense in $G(\Cee)$, $U(\Cee)\cdot o$ contains a $k$-rational point. Hence we can assume that $o$ itself is $k$-rational and thus that   $U(\Cee)\cdot o$  is an affine space defined over $k$. Since $G(\Cee)/P(\Cee)$ is a Legendrian subvariety of the space $\check \bD$ of isotropic filtrations of $V$ in the usual sense, by taking the first subspace $\Cee\cdot \omega$ of the filtration we have a morphism from $U(\Cee)\cdot o$ to the $k$-rational variety $U(\Cee)\cdot x$, where $x \in \Pee^{2h+2}(k)$ is the line spanned by $\omega$. Write $x = \Cee \cdot f_0 $ where $f_0\in V(k)$. We now define subspaces of $V$, defined over $k$, as follows:
\begin{align*}
V_0 &= \Cee \cdot f_0;\\
V_1 &= \mathfrak{u}(V_0);\\
V_2 &= V^2\cap (f_0)^\perp;\\
V_3 &= V^3 = \Cee \cdot e_0,
\end{align*}
where $\langle e_0, f_0 \rangle = 1$ and $e_0$ is uniquely defined up to scalar. By construction, $\mathfrak{u}$ is the tangent space to $U(\Cee)\cdot o$ at every point, so by horizontality $V_0\oplus V_1$ is an isotropic subspace of $V$. By definition, $\mathfrak{u}(V_0) = V_1$ and in fact the choice of $f_0$ identifies $V_1$ with $\mathfrak{u}$. Moreover, 
$$\langle \mathfrak{u}(V_1), f_0 \rangle = \langle \mathfrak{u}(f_0), V_1 \rangle =\{0\}$$
since $V_1$ is isotropic. Hence $\mathfrak{u}(V_1) \subseteq V^2\cap (f_0)^\perp = V_2$. Finally, $\mathfrak{u}(V_2) \subseteq V^3=V_3$. Clearly, $V^2 = V_2\oplus V_3$ and $V^2$ is isotropic since $V^2 = \operatorname{Im}\mathfrak{u}^2$. It is then easy to see that $V =  V_0\oplus V_1\oplus  V_2\oplus V_3$.

Writing $N_w$ in terms of the direct sum decomposition and the identifications above gives
$$N_w(t,v, \xi, s) = (0, tw, B_0(v,w),-\langle \xi, w\rangle),$$
where the last entry follows from the fact that $N_w$ preserves the symplectic form and the identification of $V_1$ with $\mathfrak{u}$ and of $V_2$ with $\mathfrak{u}\spcheck$, via
$$\langle N_w(\xi), f_0\rangle = \langle N_w(f_0), \xi\rangle = \langle w, \xi\rangle = -\langle \xi, w\rangle.$$
Exponentiating the action of $N_w$ shows that the affine space $\mathbb{A} = U(\Cee)\cdot x$ is described via Equation~\ref{eqalld2}. Moreover, the Zariski closure $\overline{\mathbb{A}}$ is equal to the image of $\check\calD$ and hence certainly contains the image of $\calD$. We must show that $\calD$ is contained in $\mathbb{A}$, not just in its closure. Note that, using the notation of Equation~\ref{eqalld2} and taking homogeneous coordinates $x_0, x_1, \dots, x_h, y_1, \dots, y_h, y_0$ on $\Pee^{2h+1}$ corresponding to the basis $e_0, e_1, \dots, e_h, f_1, \dots , f_h, f_0$, the projective closure $\overline{\mathbb{A}}$ is contained in the variety defined by the homogeneous equations $y_0^2x_0 = -\varphi(y_1, \dots, y_h)$ and $\D y_0x_i = \frac{\partial \varphi}{\partial z_i}(y_1, \dots, y_h)$ for $i= 1, \dots, h$. Hence, the intersection of $\overline{\mathbb{A}}$ with the affine open subset $y_0 \neq 0$, or equivalently with $\{x \in \Pee^{2h+1}: \langle x , e_0\rangle \neq 0\}$, is exactly $\mathbb{A}$. So it suffices to show that the image of $\calD$ in $\Pee^{2h+1}$ is contained in the affine open subset $\{x \in \Pee^{2h+1}: \langle x , e_0\rangle \neq 0\}$. Equivalently, we must show that there does not exist an $x$ in the image of $\calD$ such that $\langle x , e_0\rangle  =0$.

To see this last statement, suppose that such an $x$ did exist. Let $B$ be any minimal real parabolic subgroup of $G$ contained in $P$ and let $K_0$ by any maximal compact subgroup of $G(\R)$ for which the Iwasawa decomposition $G(\R) = B(\R)K_0$ holds. Then $G(\R)= P(\R)K_0$ as well. It follows that $P(\R)$ acts transitively on $\calD$ and preserves the condition that 
$\langle x , e_0\rangle  =0$. But then the image of $\calD$ would be contained in $\{x \in \Pee^{2h+1}: \langle x , e_0\rangle = 0\}$. This contradicts the fact that $\calD$ is open in $\check\calD$.
\end{proof}

\begin{remark} 
 If $\varphi$ is an arbitrary cubic polynomial defined over a subfield $k$ of $\R$ and $\hat{Z}$ is (the Zariski closure of the locus) defined by Equation~\ref{eqalld2}, then $\operatorname{Aut}_{\Sp}(\hat{Z})$, the group of symplectic automorphisms of $\Pee^{2h+1}$ preserving $\hat{Z}$, is an affine group scheme defined over $k$. In particular, in the Hermitian symmetric case, if $\varphi$ is defined over $k$, then $G$ can be defined over $k$ as well.
\end{remark}

The compact duals in the weight three tube domain cases are connected to homogeneous Legendrian submanifolds via the following theorem (e.g.\ \cite{mukai}, \cite{lm4}).  

\begin{theorem}\label{classifylegendre}
The homogeneous Legendrian varieties $X=G(\bC)/P(\bC)$ are as follows:
\begin{itemize}
\item[(i)] a linear embedding $\bP^{h}\subseteq \bP^{2h+1}$ for $G(\bC)=\SL(h+3,\bC)$.
\item[(ii)] the Segre embedding $\bP^1\times Q_{h-1}\subseteq \bP^{2h+1}$ (with $Q_h$  a quadric) for $G(\bC)=\SL(2,\bC)\times \SO(n,\bC)$;
\item[(iii)] the twisted cubic $\bP^1\subseteq \bP^3$ for  $G(\bC)=\SL(2, \Cee)$;
\item[(iv)] the subexceptional series corresponding to $A_5$, $C_3$, $D_6$, and $E_7$ respectively (these are  discussed in detail in \cite{lm1}).
\end{itemize}
\end{theorem}

 Note also that the homogeneous Legendrian varieties correspond precisely to the maximal Hermitian VHS of Calabi--Yau threefold type (see  Corollary~\ref{thmclassify3} and Remark~\ref{remmaxvar}).
Specifically, the four examples of Gross \cite{bgross} (cf.\ Corollary~\ref{thmclassify3}(i)) correspond to   Case (iv) in the above theorem. The remaining cases arise as follows: Case (iii) corresponds to $\Sym^3V$, where $V$ is the canonical weight one variation of elliptic curve type over the upper half plane $\mathfrak{H}$ (see Corollary~\ref{thmclassify3}(ii)). Case (ii) corresponds to $\mathfrak{H}\times \mathcal{D}$, where $\mathcal{D}$ is a Type $\IV_n$ symmetric space corresponding to  variations of $K3$ type (see Corollary~\ref{thmclassify3}(iv)). Case (i) corresponds to the unit ball as discussed elsewhere in this paper (e.g.\ Corollary~\ref{thmclassify3}(iii)). 
\begin{remark}
Some examples of non-homogeneous Legendrian varieties have been given by Landsberg--Manivel \cite{lm4} and Buczy\'nski \cite{buchyp}. While these examples give interesting horizontal subvarieties of period domains $\bD$ of CY threefold type, they will be stabilized by small groups and thus will not occur as images of period maps (compare with Theorem \ref{thmalgebraic} and Theorem \ref{thmstabilizersmooth}).
\end{remark}

\begin{remark}\label{secants} There is a further connection between the compact duals and Severi varieties (see \cite{mukai}, \cite{baily}, \cite{lm1}). By a theorem of Zak  (see \cite{lazarsfeld}) there are precisely $4$ smooth projective varieties $X\subseteq \bP^n$ with $3\dim X=2(n-2)$ such that the secant variety $\mathrm{Sec}(X)$ is not (as expected) $\bP^n$. In fact, $\mathrm{Sec}(X)\subseteq \bP^n$ is an irreducible cubic hypersurface with large symmetry group. The $4$ examples (of dimensions $2,4,8,16$) are: 

\begin{itemize} 
\item[(1)] {\it The Veronese surface:} $\bP^2\hookrightarrow \bP^5$;
\item[(2)] {\it The Segre embedding:} $\bP^2\times\bP^2\hookrightarrow \bP^8$;
\item[(3)] {\it The Pl\"ucker embedding:} $G(2,6)\hookrightarrow \bP^{14}$;
\item[(4)] {\it The exceptional case:} $X\hookrightarrow \bP^{26}$, the orbit of the highest weight vector for a $27$-dimensional representation of $E_6$.
\end{itemize} 
Here the cases (1)--(4) correspond to the $4$ tube domains. In all cases, $X$ is a compact Hermitian symmetric space for the complex form $G'_\Cee$ of the derived group of $K$, the secant variety $\mathrm{Sec}(X)$ is the cubic $C =V(\varphi) \subseteq \Pee^{h-1}$, and $X$ is the singular locus of $\mathrm{Sec}(X)$.
\end{remark}

 \subsection{The case where the domain is not of tube type}
We begin with the analogue of the Landsberg--Manivel picture in this case. Here, if $\mathfrak{g}_\Cee$ is the complex Lie algebra of $G$ and we write $\mathfrak{k}_\Cee \cong  \mathfrak{g}_\bC'\oplus \mathfrak{z}$ as before, then as in the previous case $\mathfrak{g}_\Cee \cong \mathfrak{k}_\Cee \oplus W_+ \oplus W_-$, where $W_\pm$  are irreducible representations of $\mathfrak{k}_\Cee$  and $W_-$ is the dual of $W_+$; as before  $W_\mp= \mathfrak{p}_\pm$.
  The minuscule representation $V_+$ of $\mathfrak{g}_\Cee$ splits as a $\mathfrak{k}_\Cee$-module as
$$V_+ \cong \underline{\Cee} \oplus W_+ \oplus W_0,$$
and hence by duality (using $V_{-}=V_+\spcheck$ as in Section \ref{sectclassify})
$$V_- \cong \underline{\Cee} \oplus W_- \oplus W_0\spcheck.$$
Here $\mathfrak{g}_\bC'$ acts in the standard way on $W_\pm$, $W_0$, and $W_0\spcheck$ and trivially on the two factors $\underline{\Cee}$, and the center $\mathfrak{z}$ acts diagonally with weights as described in Section \ref{sectclassify} (esp. \S\ref{listrep}). As before, $V_\bC=V_+\oplus V_-$ has a natural real structure. 

The main new feature in the non-tube case is that there exists a $G'_\Cee$-invariant bilinear form $B\colon W_+\otimes W_+ \to W_0$ or dually $B^*\colon W_0\spcheck \otimes W_+\to W_-$. For example, in case $G$ is of type $E_6$, then $\fg'_\bC$ is of type $D_5$,  $W_\pm$ are the two $16$-dimensional spin representations of $\mathfrak{g}_\bC'$, $W_0 \cong W_0\spcheck$ is the standard representation of dimension $10$, and $B^*$ is Clifford multiplication. The unit ball case ($G = \SU(1,n)$) is the special case where $W_0 = 0$, $B=0$, and $W_+$ is the standard representation of $\U(n)$.

In the above notation, the natural nilpotent action of $\mathfrak{p}_-  =W_+$ on $V$ is given by
$$N_w(s,v,e) = (0, sw, B(v,w)),$$
with $N_w^2\neq 0$, $N_w^3 =0$.
The dual action on $V_- \cong W_0\spcheck \oplus W_- \oplus \underline{\Cee}$ is then
$$N_w(e,\xi, t) = (0, B^*(e,w), \langle \xi, w\rangle).$$
Hence, on $V$,  $N_w^2(s,v,e) = (0, 0, sB(w,w))$ and 
$$\exp N_w(s,v,e)   = (s, v+sw, e+ B(v,w) + \frac{s}2B(w,w)).$$
In particular
$$\exp N_w(1,0,0)   = (1,  w,  \frac12B(w,w)),$$
and this is the explicit complex  description of the horizontal subvariety $Z$ in this notation (as opposed to that of Equation~\eqref{eqallddegen3}).

In summary, we have shown the following:

\begin{theorem}\label{thmexplicit2} 
Let $V_\bC=V_+\oplus V_-$ be a weight three Hermitian variation of Hodge structure of complex type. Then there exists a complex basis of $V_+$ and $V_-$ such that the variation of Hodge structure  is described by Equation~\eqref{eqallddegen3}, where the $q_k$ are homogenous quadratic polynomials.
\end{theorem}

\begin{remark} There are also rationality statements in case $G$ and $P$ are defined over a subfield $k$ of $\R$ and $V_{\pm}$  are defined over an imaginary quadratic extension $K$ of $k$.
\end{remark}

\begin{remark} In  the above notation, we have
\begin{enumerate}
  \item $N_w \neq 0$ $\iff$ $w\neq 0$.
  \item If $w\neq 0$, then $\operatorname{Ker}N_w = \{(0, v, e): B(v,w) = 0\}$.
  \item $N_w^2 \neq 0$ $\iff$ $B(w,w)\neq 0$.
\end{enumerate}
\end{remark}

\subsection{An example} As already mentioned (see Remark \ref{remmaxvar}, Remark~\ref{extremark}, and Equation~\eqref{eqdegenext}), most of the complex cases can be embedded into maximal Hermitian VHS (of real type). Thus, we consider the following situation:
$$\calD'\hookrightarrow \calD\hookrightarrow \bD,$$
where $\calD'$ and $\calD$ are Hermitian symmetric domains with a holomorphic equivariant embedding $\calD'\subset \calD$ (as classified by Satake \cite{satake} and Ihara \cite{ihara}), and  $\calD\hookrightarrow \bD$ is a maximal horizontal embedding in a period domain $\bD$ of Calabi--Yau threefold type. We assume $\calD$ is of tube type, while $\calD'$ is not. Clearly, $\calD'\hookrightarrow \bD$ is a horizontal embedding.  It is interesting to compare (via restriction from $\calD$ to $\calD'$) the description of the embedding of $\calD$ given by Theorem \ref{hermcubiccase} with that of $\calD'$ given by Theorem \ref{classifylegendre}. We will discuss this below for the two most interesting cases in the Satake--Ihara classification, both cases corresponding to maximal subdomains in the exceptional domain $\EVII$.

\subsubsection{$\EIII\hookrightarrow \EVII$}\label{e6e7}
 As already noted, the minuscule $27$-dimensional representation for $\E_6$ decomposes as a $\Spin(10)$-module into $\underline{\Cee} \oplus W_+ \oplus W_0$, where $W_+$ is one of the half spin representations, of dimension $16$,  and $W_0$ is the standard representation, of dimension $10$. The procedure of Remark~\ref{extremark} defines a $26$-dimensional variation of Hodge structure associated to a cubic form $\Phi_0$, which can be thought of intrinsically as induced from the trilinear form $W_+\otimes W_0 \otimes W_- \to \Cee$ given by Clifford multiplication. On the other hand, the description of the $\EVII$ domain is given by Equation \eqref{eqalld2} with $\Phi$ the Cartan cubic, the unique $E_6$-invariant cubic polynomial on $W$, where $W$ is a minuscule $E_6$-representation (compare \eqref{decomposek}).  The precise relationship between the equations realizing the $\EIII$ domain as 
 a horizontal subvariety of CY type (based on Clifford multiplication) and those 
for the $\EVII$ domain (based on the Cartan cubic) is then as follows:
\begin{proposition}
With notation as above, $\Phi_0$ is the restriction of the Cartan cubic $\Phi$ to a hyperplane. 
\end{proposition}
\begin{proof} Perhaps the most natural way to do so is via del Pezzo surfaces: view the root system for $E_6$ as the primitive cohomology of $X_6$, the blowup of $\Pee^2$  at $6$ general points, and let $h$ and $e_1, \dots, e_6$ be the classes of the pullbacks of the hyperplane class in $\Pee^2$ and the $6$ exceptional divisors, and view the root system for $D_5$ as lying in the span of $h, e_1, \dots, e_5$. The weights of one of the minuscule representations of $E_6$ correspond to lines on a cubic surface, and three weight spaces pair nontrivially under the cubic $\iff$ the lines sum to a hyperplane section. For $D_5$, the weights for $W_+$ can be taken to correspond to lines on a degree $4$ del Pezzo surface $X_5$, those for $W_0$ correspond to pencils of conics, and those for $W_-$ to linear systems of twisted cubics on $X_5$, with a line and a twisted cubic pairing nontrivially $\iff$ their sum is a hyperplane section. Then viewing Clifford multiplication as a pairing $W_+ \otimes W_+ \to W_0$, two weight spaces corresponding to lines $\ell_1$ and $\ell_2$ pair nontrivially $\iff$ $\ell_1 \cdot \ell_2 = 1$, and in this case the corresponding weight space is the pencil of conics $|\ell_1 + \ell_2|$; equivalently, a line  and a conic pair nontrivially $\iff$ their sum is a hyperplane section. To relate this to the Cartan cubic, note that there are $11$ lines on $X_6$ which do not correspond to lines on $X_5$. Ten of these correspond to pencils of conics on $X_5$: if $|C|$ is such a pencil, there is a unique element $C_0\in |C|$ passing through the point corresponding to the $6^{\text{th}}$ blowup and the new line is the proper transform of $C_0$. The remaining line is $e_6$. In this way, we can identify the restriction of the Cartan cubic to the sum of all of the weight spaces except for that spanned by $e_6$ (at least over $\Cee$, and being somewhat careless about possible scalings of the factors) with $\Phi_0$.
\end{proof}

\subsubsection{$\I_{2,6}\hookrightarrow \EVII$} If we look instead at the subspace of the $\EVII$ Hermitian symmetric space corresponding to $\SU(2,6)$ (i.e.\ of type $\I_{2,6}$), then we see a somewhat different picture. The homomorphism $\SU(2,6) \to  \E_{7,3}$ induces a homomorphism on the maximal compact subgroups, and hence on the semisimple factors $\SU(2) \times \SU(6) \to \E_6$. We work over $\Cee$, and use the notation of del Pezzo surfaces as in \S\ref{e6e7} to describe the roots for $E_6$ and $E_7$.  The simple roots for $E_6$ are $e_1-e_2, \dots, e_5-e_6, h-e_1-e_2-e_3$, and the highest root is $\tilde\alpha = 2h-\sum_{i=1}^6e_i$. Using the fact that (up to Weyl equivalence) there is a unique embedding of $A_7$ in $E_7$, one can check that the embedding of $A_1+ A_5$ in $E_6$ can be assumed to be of the following form: take $e_1-e_2, \dots, e_5-e_6$ for the simple roots of the $A_5$ term and $-\tilde \alpha$ as the simple root for the $A_1$ term.
The $27$-dimensional minuscule representation $W$ for $\E_6$ with highest weight $e_6$ then becomes a representation for $\SU(2) \times \SU(6)$. To describe it, let $V_0$ be the standard representation for $\SU(2,6)$, so that $V_0$ decomposes as a representation for the maximal compact subgroup $\mathrm{S}(\U(2) \times \U(6))$ as $W_1 \oplus W_2$, where $\dim W_1 = 2$ and $\dim W_2 = 6$, in the obvious way. Then the representation $V =\bigwedge ^2V_0$ decomposes as $\underline{\Cee} \oplus (W_1\otimes W_2) \oplus \bigwedge ^2W_2\spcheck$, where $\dim W_1\otimes W_2 =12$ and $\dim \bigwedge ^2W_2\spcheck = 15$. Thus $W$, viewed as a representation of $\SU(2) \times \SU(6)$, is isomorphic to $(W_1\otimes W_2) \oplus \bigwedge ^2W_2\spcheck$. The recipe of Remark~\ref{extremark} constructs the cubic form $\Phi_0$ on $W= (W_1\otimes W_2) \oplus \bigwedge ^2W_2\spcheck$  as follows: given $x_1, x_2 \in W_1\otimes W_2$, write $x_1\wedge x_2$ for the (symmetric) pairing $(W_1\otimes W_2) \otimes (W_1\otimes W_2) \to (\bigwedge ^2W_1)\otimes (\bigwedge ^2W_2)\cong \bigwedge ^2W_2$, and define the trilinear form
$C_0(x_1,x_1, \xi) = \xi(x_1\wedge x_2)$. Then $C_0$ corresponds to a cubic form $\Phi_0$ on the $27$-dimensional vector space $(W_1\otimes W_2) \oplus \bigwedge ^2W_2\spcheck$. However, $\Phi_0$ is \textbf{not} the Cartan cubic. Instead we must modify $\Phi_0$ (as discussed at the end of  Remark~\ref{extremark}), by adding a term $\Phi_1$, where $\Phi_1$ is the homogeneous cubic corresponding to the symmetric trilinear form $C_1$, only nonzero on the $\bigwedge ^2W_2\spcheck$ summand, defined by
$$C_1(\xi_1, \xi_2, \xi_3) = \xi_1 \wedge  \xi_2\wedge \xi_3 \in \bigwedge ^6W_2\spcheck \cong \Cee.$$ 
Using the explicit description of the weight spaces and the forms involved, one can then check that, over $\Cee$ and with some care as to the scalings of the weight spaces, the cubic $\Phi_0 + \Phi_1$ is in fact the Cartan cubic.

\section{Concluding remarks}\label{sectconclude}
In this section, we continue to restrict to the weight three tube domain case. Our goal is to identify interesting Hodge theoretical loci in terms of the Hermitian symmetric space $\mathcal{D}$ or the projective geometry of the cubic $\varphi$.

\subsection{The intermediate Jacobian locus} The most interesting locus of Hodge tensors from a geometric point of view is that where the intermediate Jacobian $JV = (V^{3,0} \oplus V^{2,1})/\Lambda$ is isogenous to a product $J_1\times J_2$ where $J_1$ is a polarized abelian variety. Equivalently, there is a symplectic direct sum decomposition over $\Q$: $V=V_1\oplus V_2$, where $V_1$ and $V_2$ are nonzero sub-Hodge structures of $V$ with $V_1^{3,0} = 0$. We shall refer to the set of all such points in $\calD$ as the \textsl{intermediate Jacobian locus}.

\begin{lemma}
Let $\calD\hookrightarrow \bD$ be a horizontal subvariety of a period domain $\bD$ which is of Hermitian type (cf.\ Definition \ref{defherm}). Then the locus $Z\subset \calD$ of points for which the associated Hodge structure is decomposable is a union  of Hermitian symmetric sub-domains each embedded holomorphically and equivariantly into $\calD$  (and thus horizontally into $\bD$).
\end{lemma}
\begin{proof}
Since the decomposability of the Hodge structure can be expressed in terms of the existence of special Hodge tensors, $Z$ will be a union of Noether--Lefschetz loci in the sense of \cite[\S II.C]{mtbook}.  Each component of a Noether--Lefschetz locus coincides with a component of some Mumford--Tate domain $D_{M'}^\circ$ (\ \cite[Theorem  II.C.1]{mtbook}). Clearly,  the embedding $D_{M'}^\circ\subset \calD$ is holomorphic and horizontal and it corresponds to a specialization of Mumford--Tate groups, i.e.\ if $M$ is the generic Mumford--Tate corresponding to $\calD$, then $M'\subseteq M$ and $M'$ is the Mumford--Tate  group of some special point (with respect to the rational structure) in $\calD$. Since $\calD\hookrightarrow \bD$ is a horizontal embedding, $D_{M'}^\circ\hookrightarrow \bD$ is also horizontal. Thus, as already argued in the proof  of Theorem \ref{thmalgebraic}, it follows (from \cite{dshimura}, \cite[Theorem  7.9]{milne}) that  $D_{M'}^\circ$ is in fact a Hermitian symmetric domain with a holomorphic, horizontal and equivariant embedding in $\bD$. 
\end{proof}

According to the lemma, the possible intermediate Jacobian loci in our situation, i.e.\ Hermitian VHS $\calV$ of CY type over $\calD$,  are given by subdomains $\calD'\hookrightarrow \calD$. 
  Conversely, suppose that $\mathcal{D}' \subseteq \mathcal{D}$ is a positive dimensional subdomain, corresponding to an inclusion of $\Q$-algebraic groups $G'\subseteq G$. If the restriction of the representation $\rho\colon G \to \Sp(V)$ is a reducible representation of $G'$ over $\Q$, then the variation of Hodge structure over $\mathcal{D}'$ splits over $\Q$ and the corresponding intermediate Jacobians acquire abelian variety factors. The possibilities for $\mathcal{D}' \subseteq \mathcal{D}$ have been tabulated by Satake \cite{satake}  and Ihara \cite{ihara} as  we have already  mentioned  (e.g.\ Remark \ref{remmaxvar}). An interesting example in this set-up is the case of the embedding of the exceptional domains $\EIII$ (associated to $E_6$) into $\EVII$ (associated to $E_7$) discussed in \S\ref{e6e7}. In this example the restriction of the VHS of weight $3$ CY type over $\EVII$  with Hodge numbers $(1,27,27,1)$ will decompose over $\EIII$ into a VHS of CY type with Hodge numbers $(1,26,26,1)$ and a Tate twist of a VHS of elliptic curve type. In fact, one can see that both VHS will have weak CM by the same imaginary quadratic field $\bQ[\sqrt{-d}]$ (compare to Section \ref{cmsect}, especially Corollary \ref{corcm}). Note that it is possible, in the above notation, for the restriction of the representation $\rho$ to $G'$ to remain irreducible. For example, the restriction of the VHS over the $\EVII$ domain to the subdomain of type $\I_{2,6}$ remains irreducible over $\R$, and so the $\I_{2,6}$ subdomain is not part of the intermediate Jacobian locus. Nonetheless, this case seems somewhat atypical.

While all of the intermediate Jacobian loci in $\calD\hookrightarrow \bD$ of Hermitian type are given by subdomains $\calD'$,  to classify all the possibilities, it does not suffices to consider the groups $G$ and $G'$ (notation as in the preceding paragraph). Instead one has to consider the full (generic) Mumford--Tate groups $M$ and $M'$ and to analyze the possible specializations $M'\subset M$ (recall that $G$ is the derived subgroup of $M$). Even when starting with a Hermitian VHS satisfying our convention \ref{conventionired}, its restriction  to intermediate Jacobian loci (or more generally Noether--Lefschetz loci) typically will not in general satisfy \ref{conventionired}.  For instance, note that one of the simplest types of Noether--Lefschetz sublocus is that cut out by endomorphisms or equivalently Hodge tensors of height $2$ (see \cite[p.\ 52-53]{mtbook}). This leads to the consideration of Hodge structures with weak real or complex multiplication as discussed in Section \ref{cmsect}, which in turn leads to involved Galois theoretic arguments (see \cite[Chapter VI, e.g.\ \S VI.D]{mtbook}) whose analysis would take us too far afield. Here we will only make some remarks from the perspective of equations defining the horizontal subvariety and apply this to the $1$-dimensional case (i.e.\ $h^{2,1}=1$).  The $1$-dimensional case can be understood also from a group theoretic perspective via the results of Green--Griffiths--Kerr (esp. \cite[Theorem  VII.F.1]{mtbook}). 

\begin{remark}
Note that the consideration of Mumford--Tate groups allows one to view also points $z_0\in \bD$ as being of Hermitian type (in the sense considered in this paper). Namely, we say that a point $z_0\in \bD$ is of Hermitian type iff the Mumford--Tate group of $z_0$ is abelian, i.e.\ a torus (since it is connected),  and thus $G$ will be trivial. Equivalently, the Hodge structure associated to $z_0$ will be of CM type (see \cite[\S V.B]{mtbook}). 
\end{remark}
\subsubsection{The equations cutting out the Noether--Lefschetz locus of type $E\times J_2$} Let us consider the very concrete case where the abelian variety factor $J_1$ is an elliptic curve $E$. In this case, given two vectors $\alpha, \beta \in V(\Q)$ such that $\langle \alpha, \beta \rangle \neq 0$, the condition that $\alpha$ and $\beta$ span a two-dimensional sub-Hodge structure of $V$ is the condition that $\alpha \wedge \beta$ is orthogonal to $F^4\bigwedge^2V$ for the symmetric pairing $\langle \cdot, \cdot \rangle$ on $\bigwedge^2V$ induced by the symplectic form on $V$. Note that
$$F^4\bigwedge^2V = (V^{3,0}\otimes V^{2,1}) \oplus (V^{3,0}\otimes V^{1,2})\oplus \bigwedge ^2 V^{2,1}.$$
We will also want the condition that the non-zero Hodge numbers occurring in $\operatorname{span}\{\alpha, \beta\}$ are $h^{2,1}$ and $h^{1,2}$ (rather than $h^{3,0}$ and $h^{0,3}$).

\begin{lemma}\label{lemmaij} 
Let $\omega$ be a generator for $V^{3,0}$. Then the following conditions are equivalent:
\begin{enumerate}
\item[\rm(i)] $\alpha \wedge \beta$ is orthogonal to $F^4\bigwedge^2V$ and the non-zero Hodge numbers occurring in $\operatorname{span}\{\alpha, \beta\}$ are $h^{2,1}$ and $h^{1,2}$.
\item[\rm(ii)]  $\langle \alpha, \omega \rangle = \langle \beta, \omega \rangle =0$ and $\alpha \wedge \beta$ is orthogonal to $\bigwedge ^2 V^{2,1}$. 
 \item[\rm(iii)]   $\langle \alpha, \omega \rangle = \langle \beta, \omega \rangle =0$ and $V^{2,1} \cap (\Cee \cdot \alpha + \Cee \cdot \beta) \neq \{0\}$. 
 \end{enumerate}
\end{lemma} 
\begin{proof} Clearly, if $\langle \alpha, \omega \rangle = \langle \beta, \omega \rangle =0$, then $\alpha \wedge \beta$ is orthogonal to $(V^{3,0}\otimes V^{2,1}) \oplus (V^{3,0}\otimes V^{1,2})$, so that if $\alpha \wedge \beta$ is orthogonal to $\bigwedge ^2 V^{2,1}$, it is orthogonal to $F^4\bigwedge^2V$ as well. Conversely, if $\alpha \wedge \beta$ is orthogonal to $F^4\bigwedge^2V$, it is orthogonal to $\bigwedge ^2 V^{2,1}$. Moreover, the condition that the non-zero Hodge numbers occurring in $\operatorname{span}\{\alpha, \beta\}$ are $h^{2,1}$ and $h^{1,2}$ implies that    $\langle \alpha, \omega \rangle = \langle \beta, \omega \rangle =0$. This shows the equivalence of the first two conditions, and the equivalence of the second and third is also straightforward.
\end{proof}

\subsubsection{The $1$-dimensional case} To illustrate the above discussion, we consider the $1$-dimensional case, i.e.\  $h=\dim V^{2,1} =1$. Let $\calV$ be a irreducible Hermitian VHS  of weight $3$ and CY type over the upper half-plane $\fH$. 
In this set-up, there are two distinct horizontal embeddings of $\fH$ into the period domain $\bD$, one corresponding to $\Sym^3 W$, where $W$ is the Hodge structure of an elliptic curve (Corollary \ref{thmclassify3}(ii)), and one corresponding to the complex ball (Corollary \ref{thmclassify3}(iii), case $\I_{1,n}$ for $n=1$). We will discuss only the first case here (the second case is similar but simpler). Thus, we assume $\fH\hookrightarrow \bD$ via   $\Sym^3$ of the standard representation of $\SL(2)$. Then one can show that

\begin{proposition}\label{claimnl1} With notation as above. The special points of $\fH$ correspond precisely to the case when the Hodge structure $V$ acquires weak CM by an imaginary quadratic field $\bQ[\sqrt{-d}]$ (see \S\ref{cmsect}).  
\end{proposition}

\begin{remark}
From the point of view of Mumford--Tate domains, it is easy to see which of the points of $\fH\hookrightarrow \bD$ are special in the sense discussed above. Specifically, since the generic Hodge group in this situation is $\SL(2)$, the only possibility for Hodge groups at special points of $\fH$ is $\U(1)$.  Then the proposition follows from \cite[Theorem  VII.F.1]{mtbook} (cases (v) and (xii)). To illustrate our approach, we give another proof  based on  the explicit description of the embedding $\fH\hookrightarrow \bD$ given by \eqref{eqalld}. 
\end{remark}
\begin{proof}[Proof of \ref{claimnl1}] The condition of Lemma \ref{lemmaij} that $\alpha \wedge \beta$ is orthogonal to $\bigwedge ^2 V^{2,1}$ is vacuous. In this case, we can write the variation of the generator of $V^{3,0}$ as $\omega(z) = -cz^3e_0 + 3cz^2 e_1 + zf_1 + f_0$ for $z$ in the upper half plane, where we assume that $c\in \Q$, $c\neq 0$. Setting 
\begin{align*}
\alpha &= \alpha_0e_0 + \alpha_1e_1 + \alpha_1'f_1 + \alpha_0' f_0;\\
\beta &= \beta_0e_0 + \beta_1e_1 + \beta_1'f_1 + \beta_0' f_0
\end{align*}
leads to two equations
\begin{align*}
 -c\alpha_0'z^3 + 3c\alpha_1'z^2  - \alpha_1 z - \alpha_0 &=0;\\
 -c\beta_0'z^3 + 3c\beta_1'z^2  - \beta_1 z - \beta_0 &=0,
\end{align*}
where the coefficients are rational, and hence after possibly eliminating the constant term to an equation of the form $zP(z) = 0$, where $P(z)$ is a quadratic polynomial with rational coefficients. Thus $z$ lies in an imaginary quadratic field. Conversely, if $P(t)$ is an irreducible quadratic polynomial with rational coefficients and $z$ is a root of $P$ in the upper half plane, choose two rational numbers $r_1,r_2$ and consider the cubic polynomials $Q_1(t) = (t+r_1)P(t)$ and $Q_2(t) = (t+r_2)P(t)$. The coefficients of $Q_1, Q_2$ determine rational  vectors $\alpha, \beta$ such that $\langle \alpha, \omega \rangle = \langle \beta, \omega \rangle =0$, and a computation shows that, up to a factor of $\frac13$, $\langle \alpha, \beta \rangle = r_1-r_2$. Hence, if $r_1$ and $r_2$ are distinct, we produce a two dimensional subspace of $V(\Q)$ with the desired properties. Here, the choice of $r_1$ and $r_2$ is essentially equivalent to the choice of a $\Q$-basis for the span of $\alpha$ and $\beta$. 
\end{proof}

\begin{remark} In the non-algebraic case of Equation~\ref{eqalld},   $\omega(z)$ is  given by 
$$
\omega(z) = \left(\varphi(z) - \frac12  z \frac{d \varphi}{d z}\right)e_0 + \frac12\frac{d \varphi}{d z } e_1 +  z f_1 + f_0,
$$
where $\varphi$ is only assumed to be a holomorphic function of $z$. The equations $\langle \alpha, \omega \rangle = \langle \beta, \omega \rangle =0$ become two equations of the form
\begin{align*}
c_1\varphi(z) + \ell_1(z)\frac{d \varphi}{d z} + m_1(z) &=0;\\
c_2\varphi(z) + \ell_2(z)\frac{d \varphi}{d z} + m_2(z) &=0,
\end{align*}
where the $c_i\in \Q$ and the $\ell_i, m_i$ are degree one polynomials of $z$ with coefficients in $\Q$. In particular, both $\varphi(z)$ and $d\varphi(z)/dz$ lie in $\Q(z)$ (and are of a very special form there). In case $\varphi$ is a cubic and $z$ is imaginary quadratic, the existence of one equation of the above type leads to a second such. It would be interesting to find examples of transcendental functions $\varphi$ and points $z$ which have analogous properties.
\end{remark}

\begin{remark}
Green--Griffiths--Kerr \cite[\S VII.F]{mtbook} (especially Theorem VII.F.1) discuss in detail the Noether--Lefschetz loci in the case of weight $3$ CY Hodge structures with $h^{2,1}=1$. Specifically, the cases (iii) and (v) of \cite[Theorem  VII.F.1]{mtbook} correspond  to two irreducible case of Hermitian VHS with $h^{2,1}=1$ in our classification: the ball case and the $\Sym^3$ case respectively. The cases (viii)--(xii) of \cite[Theorem  VII.F.1]{mtbook} classify the possible Mumford--Tate groups for Hodge structures that decompose as $V=V_1\oplus V_2$. Since we are considering Noether--Lefschetz subloci of Hermitian symmetric domains, the only relevant cases for us are (xi) and (xii). The most interesting situation, the case (v) specializing to the case (xii), was discussed above (Proposition \ref{claimnl1}). For some geometric situations that lead to these two cases (essentially some special cases of the Borcea--Voisin construction) we refer the reader to \cite[p. 201-202]{mtbook}.
\end{remark}

\subsection{The Baily-Borel compactification and monodromy}
We keep the notation of \S\ref{ssectube}, so that $V =\underline{\Cee}\oplus W\oplus W\spcheck \oplus \underline{\Cee}$ and $B\colon \operatorname{Sym}^2 W \to W\spcheck$ is the bilinear form. We begin by analyzing the following general situation: $\sV$ is a weight three $\Q$-VHS of CY type over $\Delta^*$ with general fiber $V$ is a complex vector space with a rational structure and a nondegenerate symplectic form defined over $\Q$.  Let  $N= \log T$, where $T\in G(\Zee)$ is the monodromy matrix, which we assume to be unipotent. In particular, $N$ is a rational nilpotent matrix preserving the symplectic form such that $N^4= 0$,. Let $F^\bullet$ be the limiting Hodge filtration and  $W_\bullet$ be the corresponding monodromy weight filtration. Thus $(V, F^\bullet, W_\bullet)$ is a mixed Hodge structure satisfying the usual properties: $N\colon W_k/W_{k-1} \to W_{k-2}/W_{k-3}$ is a morphism of Hodge structures of type $(-1, -1)$, and $N^k\colon W_{3+k}/W_{3+k-1} \to W_{3-k}/W_{3-k-1}$ is an isomorphism. Note that $N^4=0$. 
\begin{definition}
We shall say that the limiting mixed Hodge structure is \textsl{of Type IV} if $N^3\neq 0$, \textsl{of Type III} if $N^3 =0$, $N^2 \neq 0$, and \textsl{of Type II} if $N^2 =0$, $N\neq 0$. (Type I is the condition that $N=0$.) 
\end{definition}
We discuss the various possibilities below, first in general and then for the Hermitian symmetric tube domain case.  Here, we shall use the fact (Remark~\ref{uniremark}) that, in the notation of Section 5,  every nilpotent $N$ of the type we are considering is conjugate in $G(\Cee)$ to an element of   $\mathfrak{p}_-$ (or equivalently that every unipotent $T\in G(\R)$ is conjugate in $G(\Cee)$ to an element of  $P_-$) to analyze the  possibilities for the monodromy weight filtration over $\Cee$ and relate them to the form $B$. Of course, a more careful analysis would proceed via the fine structure of Hermitian symmetric spaces and would give information about the real structure. For  the three classical cases, and especially for  $\Sp(6, \R)$ and $\SU(3,3)$,  one can work out all of the statements below directly.

\medskip
\noindent \textbf{Type IV:} $N^4= 0, N^3\neq 0$ (or \textsl{maximal unipotent monodromy}). The Calabi--Yau condition implies that $F^3 \to W_6/W_5$ is an isomorphism and that $W_5/W_4 = 0$. In particular, $N$ has a unique Jordan block of length $4$ and no Jordan block of the length $3$. The remaining Jordan blocks must be of length $2$ or $1$. In terms of the weight filtration, this says that $W_4/W_3$ is a Hodge structure of pure type $(2,2)$ and  $N\colon W_4/W_3 \to W_2/W_1$ is an isomorphism of type $(-1, -1)$, so that $W_2/W_1$ is pure of type $(1,1)$. This leaves open the possibility that $W_3/W_2$ is nonzero, whose nonzero summands are of Hodge type $(2,1)$ and $(1,2)$, corresponding to length $1$ Jordan blocks of $N$. It is easy to see that there are no length $1$ Jordan blocks $\iff$ $W_3=W_2$ $\iff$ $N$ induces an isomorphism from $W_4/W_2$ to $W_2/W_0$.

In the Hermitian case, we can assume that $N=N_w$ corresponds to the action of $w$ on $V$. Using the explicit description of $N_w$, we see that Type IV corresponds to $C(w) = \langle B(w,w), w\rangle \neq 0$, and the condition that there are no length $1$ Jordan blocks is equivalent to:

\begin{condition} If $C(w) \neq 0$, then the linear map $B_w$ defined by $v\in W\mapsto B(w,v) \in W\spcheck$ is an isomorphism from $W$ to $W\spcheck$.
\end{condition}

This follows easily, in the $\EVII$ case,  from the fact that the only $\E_6$-invariant polynomials on the minuscule representation $W$ are polynomials in the Cartan cubic (here denoted $C$).

\medskip
\noindent \textbf{Type III:} $N^3= 0, N^2\neq 0$. In this case, $V= W_6=W_5$ and there is a length $3$ Jordan block, so that $W_5/W_4 \neq 0$. The Calabi--Yau condition then implies that $\dim W_5/W_4 = 2$ and that up to a Tate twist $W_5/W_4$ is of elliptic curve type (the nonzero summands are one-dimensional of type $(3,2)$ and $(2,3)$). Of course, a similar statement is true for $W_1/W_0 = W_1$, and, since $N^2\colon W_5/W_4 \to W_1$ is an isomorphism, $N$ induces an injection from $W_5/W_4$ to $W_3/W_2$. Again, the remaining Jordan blocks must be of length $2$ or $1$ and the Hodge structure on $W_4/W_3$ is of pure type $(2,2)$ and that on $W_2/W_1$ is pure of type $(1,1)$. Again, \textit{a priori} there could be length $1$ Jordan blocks of $N$ corresponding to the possibility that $W_3/W_2$ is strictly larger than the $2$-dimensional image of $N$. 

In the tube domain case, the condition $N^3= 0, N^2\neq 0$ is equivalent to $C(w) =0$ but $B(w,w)\neq 0$. The condition that there are no length $1$ Jordan blocks is equivalent to:

\begin{condition} Let $L$ be the hyperplane $\Ker B(w,w)\subseteq W$, let $L^*$ be the hyperplane $w^\perp =\{ \xi \in W\spcheck: \langle \xi, w\rangle = 0\}$, and let $B_w\colon L\to L^*$ be the linear map $B_w(v) = B(w,v)$. (Note that $w\in L$ and that, if $v\in L$, then $B_w(v)(w) = \langle B(w,v), w\rangle = \langle B(w,w), v\rangle =0$ by definition, so that $B_w(v)\in L^*$.) Then $B_w$ is an isomorphism from $L$ to $L^*$.
\end{condition}

In the $\EVII$ case, this also presumably reduces to a known fact about the Cartan cubic.

\medskip
\noindent \textbf{Type II:} $N^2= 0, N\neq 0$. In this case, $W_4 = V$, $W_3 = \Ker N$, and $W_2 =\operatorname{Im} N$. Of course, this holds in the Type I case also. There is no Jordan block of length $1$ $\iff$ $W_3=W_2$ $\iff$ $\Ker N = \operatorname{Im} N$. In the cases of interest to us, we will in fact have $W_3\neq W_2$, so the mixed Hodge structure will not be of Hodge--Tate type.

In the tube domain case, $N^2= 0, N\neq 0$ $\iff$ $B(w,w)= 0$ but $w\neq 0$. In this case, 
\begin{align*}
\operatorname{Im} N &= \{(0, tw, B(w,v), s): t, s\in \Cee, v\in W\};\\
\Ker N &= \{(0, v, \xi, s): v\in \Ker B_w, \xi \in w^{\perp}\}.
\end{align*}
Thus $\Ker N/\operatorname{Im} N \cong (\Ker B_w/\Cee\cdot w) \oplus (w^{\perp}/\operatorname{Im}B_w)$, where as before $B_w(v) = B(w,v)$, we have repeatedly used $B(w,w)= 0$, and the two summands are dual to each other. Note that, by the discussion of Remark~\ref{secants}, the image in $\Pee^{h-1}$ of the set of  $w\in W$ such that $B(w,w) =0$ is a single orbit under $K(\Cee)$, and hence the dimensions of the spaces $W_3$ and  $W_2$ are independent of the choice of $w$. 

The above does not specify the shape of the mixed Hodge structure. In general, we make the following definition:

\begin{definition}\label{defpl} Let $(V, F^\bullet, W_\bullet)$ be a weight three limiting  mixed Hodge structure  of Calabi--Yau type ($\dim F^3 = 1$) such that $N^2=0$, $N\neq 0$. We say that $(V, F^\bullet, W_\bullet)$ is \textsl{of Picard--Lefschetz type} if $W_4/W_3$ is of pure weight $(2,2)$, or equivalently if $W_3/W_2$ is of Calabi--Yau type. Note that, if $(V, F^\bullet, W_\bullet)$ is not of Picard--Lefschetz type, then $W_4/W_3$ is (up to a Tate twist) of $K3$ type. 
\end{definition}

In the case of interest to us, Picard--Lefschetz type does not arise (cf.\ the remark by Deligne in \S10 of \cite{bgross}):

\begin{lemma} In the Hermitian symmetric case, with $N^2=0$ but $N\neq 0$, the resulting limiting mixed Hodge structure is never of Picard--Lefschetz type.
\end{lemma}
\begin{proof} (Sketch.) Clearly $(V, F^\bullet, W_\bullet)$ is   of Picard--Lefschetz type $\iff$ $N\omega = 0$, where $F^3 = \Cee \cdot\omega$, $\iff$ $T\omega =\omega$. On the other hand, up to conjugacy, there is a boundary component $F$ such that $T\in U(F)$ and $\omega$ is induced by an element of $\calD(F) = U(F)(\Cee)\cdot \calD$ (cf.\ \cite{amrt}, \S7 of Chapter III). 
Hence $\omega = S\omega_0$ for some $S\in U(F)(\Cee)$, where   $\omega_0$ is induced by an element of $\calD$. Since $U(F)$ is abelian, it follows that $T\omega_0 = \omega_0$. But then $\omega_0$ would  fixed by a nontrivial unipotent element of $G(\R)$, contradicting the fact that its stabilizer in $G(\R)$ is compact.
\end{proof}

To tie the above picture in with the general theory of Hermitian symmetric spaces, Kor\'anyi and Wolf have classified the boundary components of the Hermitian symmetric spaces \cite{wolfkor}. In the four cases of irreducible tube domains of real rank three, the boundary components are almost completely specified by the conditions that they be of tube type and of real rank $0$, $1$, $2$. In particular, the real rank $0$ boundary components are points, the real rank $1$ boundary components are copies of the upper half plane, and the real rank two boundary components are the rank two Hermitian symmetric spaces associated to $\SO(2,10)$ in case $G$ is of type $\E_7$, $\SO^*(8) = \SO(2, 6)$ in case $G$ is of type $D_5$,  $\SU(2,2)=\SO(2,4)$ in case $G$ is of type $A_5$, and of type $C_2 = B_2$, i.e.\ $\Sp(4, \R) \cong \SO(2,3)$, in case $G$ is of type $C_3$ (all equalities mod the center). 

Putting this together, we see that, in the Type IV case, the mixed Hodge structure is of Hodge--Tate type. In the Type III case,  $W_2 = W_2/W_0$ is   an extension of a weight one Hodge structure by a pure weight two Hodge structure and $W_6/W_3 = W_5/W_3$ is a Tate twist of the dual of $W_2$. In this case, presumably all of the information of the mixed Hodge structure is contained in $W_2$. In the Type II case, the Hodge structure on $W_2$ is of $K3$ type, the one on $W_4/W_3$ is a Tate twist of $W_2$, and the one on $W_3/W_2$ is, up to a Tate twist, the Kuga--Satake construction applied to the weight two Hodge structure on $W_2$, because the representation of $\SO(2,k)$ corresponding to $W_3/W_2$ is essentially a spin or half spin representation. It seems likely that, in all cases, the essential information of the mixed Hodge structure is contained in the two step extension $W_3$, and that this extension can be described quite explicitly.

\begin{remark} In the complex case, the weight three variation of Hodge structure is obtained by reassembling a weight one variation (in the unit ball case) or a weight two variation (in the remaining cases). Thus, we must have $N^2 =0$ in the unit ball case and $N^3 = 0$ in the remaining cases. Of course, it is easy to see directly in the case of $\SU(1,n)$, or more generally $\SU(p,q)$, that, if $P$ is a maximal real parabolic subgroup corresponding to a zero-dimensional boundary component, $U$ is the unipotent radical of $P$, and $N\in U$, then $N^2=0$, where we view $N$ as acting on $\Cee^{p+q}$ in the standard representation (compare Rohde \cite{rohde3}, who notes that in the unit ball case the monodromy is never maximal unipotent). 

One can also analyze the resulting limiting mixed Hodge structures along the lines of the tube domain case. For example, in the unit ball case, viewing the complex representation $V \cong \underline{\Cee}\oplus W$ with dual $W\spcheck \oplus \underline{\Cee}$, if $N = N_w$ is a nontrivial monodromy matrix, then $W_1$ and $W_4/W_3$ are  of rank two and pure types $(1,1)$ $(2,2)$, respectively, and $W_3/W_2 \cong W/\Cee\cdot w \oplus w^\perp $, but the Hodge structure on $W_3/W_2$ is constant. Similar but slightly more involved statements hold in the remaining cases; note here that the real rank one boundary components are unit balls.
\end{remark}

\bibliography{refcy}

\begin{thebibliography}{AMRT75}

\bibitem[AMRT75]{amrt}
A.~Ash, D.~Mumford, M.~Rapoport, and Y.~Tai.
\newblock {\em Smooth compactification of locally symmetric varieties}.
\newblock Math. Sci. Press, Brookline, Mass., 1975.

\bibitem[And92]{andre}
Y.~Andr{\'e}.
\newblock Mumford-{T}ate groups of mixed {H}odge structures and the theorem of
  the fixed part.
\newblock {\em Compositio Math.}, 82(1):1--24, 1992.

\bibitem[Bai00]{baily}
W.~L. Baily, Jr.
\newblock Exceptional moduli problems. {II}.
\newblock {\em Asian J. Math.}, 4(1):1--9, 2000.
\newblock Kodaira's issue.

\bibitem[BG83]{bg}
R.~L. Bryant and P.~A. Griffiths.
\newblock Some observations on the infinitesimal period relations for regular
  threefolds with trivial canonical bundle.
\newblock In {\em Arithmetic and geometry, {V}ol. {II}}, volume~36 of {\em
  Progr. Math.}, pages 77--102. Birkh\"auser Boston, Boston, MA, 1983.

\bibitem[Bor97]{borcea}
C.~Borcea.
\newblock {$K3$} surfaces with involution and mirror pairs of {C}alabi-{Y}au
  manifolds.
\newblock In {\em Mirror symmetry, {II}}, volume~1 of {\em AMS/IP Stud. Adv.
  Math.}, pages 717--743. Amer. Math. Soc., Providence, RI, 1997.

\bibitem[Bou02]{bourbaki}
N.~Bourbaki.
\newblock {\em Lie groups and {L}ie algebras. {C}hapters 4--6}.
\newblock Elements of Mathematics (Berlin). Springer-Verlag, Berlin, 2002.

\bibitem[Buc08]{buchyp}
J.~Buczy{\'n}ski.
\newblock Hyperplane sections of {L}egendrian subvarieties.
\newblock {\em Math. Res. Lett.}, 15(4):623--629, 2008.

\bibitem[CF12]{filippini}
S.L. Cacciatori and S.~A. Filippini.
\newblock The ${E}^3/\mathbb {Z}_3$ orbifold, mirror symmetry, and {H}odge
  structures of {C}alabi-{Y}au type.
\newblock arXiv:1201.5057, 2012.

\bibitem[Del79]{dshimura}
P.~Deligne.
\newblock Vari\'et\'es de {S}himura: interpr\'etation modulaire, et techniques
  de construction de mod\`eles canoniques.
\newblock In {\em Automorphic forms, representations and {$L$}-functions},
  Proc. Sympos. Pure Math., XXXIII, pages 247--289. AMS, Providence, R.I.,
  1979.

\bibitem[DK07]{dk}
I.~V. Dolgachev and S.~Kond{\=o}.
\newblock Moduli of {$K3$} surfaces and complex ball quotients.
\newblock In {\em Arithmetic and geometry around hypergeometric functions},
  volume 260 of {\em Progr. Math.}, pages 43--100. Birkh\"auser, Basel, 2007.

\bibitem[DM86]{delignemostow}
P.~Deligne and G.~D. Mostow.
\newblock Monodromy of hypergeometric functions and nonlattice integral
  monodromy.
\newblock {\em Inst. Hautes \'Etudes Sci. Publ. Math.}, 63:5--89, 1986.

\bibitem[FL13]{FL2}
R.~Friedman and R.~Laza.
\newblock On some {H}ermitian variations of {H}odge structure of {C}alabi-{Y}au
  type with real multiplication.
\newblock arXiv:1301.2582, 2013.

\bibitem[Fri91]{friedmancy}
R.~Friedman.
\newblock On threefolds with trivial canonical bundle.
\newblock In {\em Complex geometry and {L}ie theory ({S}undance, {UT}, 1989)},
  volume~53 of {\em Proc. Sympos. Pure Math.}, pages 103--134. Amer. Math.
  Soc., Providence, RI, 1991.

\bibitem[GGK12]{mtbook}
M.~Green, P.~A. Griffiths, and M.~Kerr.
\newblock {\em Mumford-{T}ate groups and domains}, volume 183 of {\em Annals of
  Mathematics Studies}.
\newblock Princeton University Press, Princeton, NJ, 2012.

\bibitem[Gro94]{bgross}
B.~H. Gross.
\newblock A remark on tube domains.
\newblock {\em Math. Res. Lett.}, 1(1):1--9, 1994.

\bibitem[GS90]{gs}
V.~Guillemin and S.~Sternberg.
\newblock {\em Variations on a theme by {K}epler}, volume~42 of {\em American
  Mathematical Society Colloquium Publications}.
\newblock American Mathematical Society, Providence, RI, 1990.

\bibitem[GSvSZ13]{zuo}
R.~Gerkmann, M.~Sheng, D.~van Straten, and K.~Zuo.
\newblock On the monodromy of the moduli space of {C}alabi-{Y}au threefolds
  coming from eight planes in $\mathbb{P}^3$.
\newblock {\em Math. Ann.}, 355(1):187--214, 2013.

\bibitem[GvG10]{vangeemen}
A.~Garbagnati and B.~van Geemen.
\newblock Examples of {C}alabi-{Y}au threefolds parametrised by {S}himura
  varieties.
\newblock {\em Rend. Semin. Mat. Univ. Politec. Torino}, 68(3):271--287, 2010.

\bibitem[Iha67]{ihara}
S.~Ihara.
\newblock Holomorphic imbeddings of symmetric domains.
\newblock {\em J. Math. Soc. Japan}, 19:261--302, 1967.

\bibitem[Kna02]{knapp}
A.~W. Knapp.
\newblock {\em Lie groups beyond an introduction}, volume 140 of {\em Progress
  in Mathematics}.
\newblock Birkh\"auser Boston Inc., Boston, MA, second edition, 2002.

\bibitem[KP12]{kollar}
J.~Koll{\'a}r and J.~Pardon.
\newblock Algebraic varieties with semialgebraic universal cover.
\newblock {\em J. Topol.}, 5(1):199--212, 2012.

\bibitem[KW65]{korwolf}
A.~Kor{\'a}nyi and J.~A. Wolf.
\newblock Realization of hermitian symmetric spaces as generalized half-planes.
\newblock {\em Ann. of Math. (2)}, 81:265--288, 1965.

\bibitem[LM01]{lm1}
J.~M. Landsberg and L.~Manivel.
\newblock The projective geometry of {F}reudenthal's magic square.
\newblock {\em J. Algebra}, 239(2):477--512, 2001.

\bibitem[LM07]{lm4}
J.~M. Landsberg and L.~Manivel.
\newblock Legendrian varieties.
\newblock {\em Asian J. Math.}, 11(3):341--359, 2007.

\bibitem[Lom01]{lombardo}
G.~Lombardo.
\newblock Abelian varieties of {W}eil type and {K}uga-{S}atake varieties.
\newblock {\em Tohoku Math. J. (2)}, 53(3):453--466, 2001.

\bibitem[Loo91]{looijengaletter}
E.~Looijenga.
\newblock Exceptional motives.
\newblock (unpublished) letters to I. Dolgachev, 1991.

\bibitem[LVdV84]{lazarsfeld}
R.~Lazarsfeld and A.~Van~de Ven.
\newblock {\em Topics in the geometry of projective space}, volume~4 of {\em
  DMV Seminar}.
\newblock Birkh\"auser Verlag, Basel, 1984.
\newblock Recent work of F. L. Zak, With an addendum by Zak.

\bibitem[Mil11]{milne}
J.~S. Milne.
\newblock Shimura varieties and moduli.
\newblock To appear in ``Handbook of Moduli'' (arXiv:1105.0887), 2011.

\bibitem[Moo99]{moonenmt2}
B.~Moonen.
\newblock Notes on {M}umford-{T}ate groups.
\newblock Preprint available at
  http://staff.science.uva.nl/~bmoonen/NotesMT.pdf, 1999.

\bibitem[Muk98]{mukai}
S.~Mukai.
\newblock Simple {L}ie algebra and {L}egendre variety.
\newblock Preprint available at
  http://www.kurims.kyoto-u.ac.jp/~mukai/paper/warwick15.pdf, 1998.

\bibitem[O'G06]{ogrady}
K.~G. O'Grady.
\newblock Irreducible symplectic 4-folds and {E}isenbud-{P}opescu-{W}alter
  sextics.
\newblock {\em Duke Math. J.}, 134(1):99--137, 2006.

\bibitem[Roh09]{rohde}
J.~C. Rohde.
\newblock {\em Cyclic coverings, {C}alabi-{Y}au manifolds and complex
  multiplication}, volume 1975 of {\em Lecture Notes in Mathematics}.
\newblock Springer-Verlag, Berlin, 2009.

\bibitem[Roh10]{rohde3}
J.~C. Rohde.
\newblock Maximal automorphisms of {C}alabi-{Y}au manifolds versus maximally
  unipotent monodromy.
\newblock {\em Manuscripta Math.}, 131(3-4):459--474, 2010.

\bibitem[Sat65]{satake}
I.~Satake.
\newblock Holomorphic imbeddings of symmetric domains into a {S}iegel space.
\newblock {\em Amer. J. Math.}, 87:425--461, 1965.

\bibitem[Sch74]{schmid2}
W.~Schmid.
\newblock Abbildungen in arithmetische {Q}uotienten hermitesch symmetrischer
  {R}\"aume.
\newblock In {\em Classification of algebraic varieties and compact complex
  manifolds}, pages 243--258. Lecture Notes in Math., Vol. 412. Springer,
  Berlin, 1974.

\bibitem[Shi08]{shimurah}
G.~Shimura.
\newblock Arithmetic of {H}ermitian forms.
\newblock {\em Doc. Math.}, 13:739--774, 2008.

\bibitem[SYY97]{yoshida}
T.~Sasaki, K.~Yamaguchi, and M.~Yoshida.
\newblock On the rigidity of differential systems modelled on {H}ermitian
  symmetric spaces and disproofs of a conjecture concerning modular
  interpretations of configuration spaces.
\newblock In {\em C{R}-geometry and overdetermined systems ({O}saka, 1994)},
  volume~25 of {\em Adv. Stud. Pure Math.}, pages 318--354. Math. Soc. Japan,
  Tokyo, 1997.

\bibitem[SZ10]{sz}
M.~Sheng and K.~Zuo.
\newblock Polarized variation of {H}odge structures of {C}alabi--{Y}au type and
  characteristic subvarieties over bounded symmetric domains.
\newblock {\em Math. Ann.}, 348(1):211--236, 2010.

\bibitem[Tre10]{trenner}
T.~Trenner.
\newblock A curvature formula for the complexified index cone of a cubic form.
\newblock arXiv:1007.2737, 2010.

\bibitem[TW11]{trennerw}
T.~Trenner and P.~M.~H. Wilson.
\newblock Asymptotic curvature of moduli spaces for {C}alabi-{Y}au threefolds.
\newblock {\em J. Geom. Anal.}, 21(2):409--428, 2011.

\bibitem[UY11]{ullmo}
E.~Ullmo and A.~Yafaev.
\newblock A characterization of special subvarieties.
\newblock {\em Mathematika}, 57(2):263--273, 2011.

\bibitem[vG01]{halftwist0}
B.~van Geemen.
\newblock Half twists of {H}odge structures of {CM}-type.
\newblock {\em J. Math. Soc. Japan}, 53(4):813--833, 2001.

\bibitem[vG08]{vgk3}
B.~van Geemen.
\newblock Real multiplication on {$K3$} surfaces and {K}uga-{S}atake varieties.
\newblock {\em Michigan Math. J.}, 56(2):375--399, 2008.

\bibitem[Voi93]{voisincy}
C.~Voisin.
\newblock Miroirs et involutions sur les surfaces {$K3$}.
\newblock {\em Ast\'erisque}, (218):273--323, 1993.
\newblock Journ{\'e}es de G{\'e}om{\'e}trie Alg{\'e}brique d'Orsay (Orsay,
  1992).

\bibitem[Voi99]{voisinmirror}
C.~Voisin.
\newblock {\em Mirror symmetry}, volume~1 of {\em SMF/AMS Texts and
  Monographs}.
\newblock American Mathematical Society, Providence, RI, 1999.

\bibitem[WK65]{wolfkor}
J.~A. Wolf and A.~Kor{\'a}nyi.
\newblock Generalized {C}ayley transformations of bounded symmetric domains.
\newblock {\em Amer. J. Math.}, 87:899--939, 1965.

\bibitem[Zar83]{zarhin}
Yu.~G. Zarhin.
\newblock Hodge groups of {$K3$} surfaces.
\newblock {\em J. Reine Angew. Math.}, 341:193--220, 1983.

\end{thebibliography}
\end{document}